\documentclass[a4paper, 11pt]{amsart}
\usepackage[top=1.2in, bottom=1.2in, left=.9in, right=.9in]{geometry}
\usepackage{amsmath, amssymb,amsmath,amscd,amsfonts,amsthm,mathrsfs, graphicx}
\usepackage[arrow,matrix,curve,cmtip,ps]{xy}
\usepackage{enumitem}
\usepackage{pinlabel}
\usepackage{tikz-cd}
\usepackage{color}
\usepackage{xcolor}
\usepackage{caption}
\usepackage{subcaption}
\usepackage{soul}

\usepackage{xfrac}% for \sfrac

\usepackage{cancel}
\usepackage{comment}

\usepackage{setspace}
\usepackage{relsize}

\usepackage{centernot}

\usepackage[dvipsnames]{xcolor}
\usepackage[pagebackref,hypertexnames=false,colorlinks=true,
            linkcolor=NavyBlue,
            citecolor=NavyBlue,
            urlcolor=NavyBlue]{hyperref} 
\usepackage{cleveref}
\usepackage[alphabetic]{amsrefs}%%%CHIRS:  I MOVED THIS TO AFTER HYPERREF TO AVOID A COMPILER ERROR.  I DO NOT UNDERSTAND WHY THIS WORKS

\newcommand{\CAT}{\operatorname{CAT}}

\usepackage{pinlabel}	
\newtheorem{theorem}{Theorem}[section]
\newtheorem{proposition}[theorem]{Proposition}
\newtheorem{corollary}[theorem]{Corollary}
\newtheorem{lemma}[theorem]{Lemma}
\newtheorem{remark}[theorem]{Remark}

\newtheorem{problem}[theorem]{Problem}

\newtheorem*{theorem*}{Theorem}
\newtheorem*{proposition*}{Proposition}
\newtheorem*{lemma*}{Lemma}
\newtheorem*{corollary*}{Corollary}
\newtheorem*{rep@theorem}{\rep@title}
\newcommand{\newreptheorem}[2]{
\newenvironment{rep#1}[1]{
 \def\rep@title{#2 \ref{##1}}
 \begin{rep@theorem}}
 {\end{rep@theorem}}}
\makeatother
\newtheorem*{rep@proposition}{\rep@title}
\newcommand{\newrepproposition}[2]{
\newenvironment{rep#1}[1]{
 \def\rep@title{#2 \ref{##1}}
 \begin{rep@proposition}}
 {\end{rep@proposition}}}
\makeatother

\theoremstyle{definition}
\newtheorem{definition}[theorem]{Definition}
\newtheorem{question}[theorem]{Question}

\newreptheorem{theorem}{Theorem}
\newreptheorem{lemma}{Lemma}
\newreptheorem{proposition}{Proposition}
\newreptheorem{corollary}{Corollary}

\newcommand{\bdry}{\partial}

\newcommand{\Q}{\mathbb{Q}}
\newcommand{\N}{\mathbb{N}}
\newcommand{\Z}{\mathbb{Z}}

\newcommand{\lk}{\operatorname{lk}}

\newcommand{\Id}{\operatorname{Id}}

\newcommand{\Bl}{\mathcal{B}\ell}

\newcommand{\PD}{\operatorname{PD}}

\renewcommand{\hbar}{\overline{h}}

\newcommand{\Swap}{S}%{\operatorname{Sw}}
\newcommand{\swap}{\Swap}%{\operatorname{Sw}}

\newcommand{\Sum}{\displaystyle \sum}

\newcommand{\pref}[1]{(\ref{#1})}

\newcommand{\noteta}{h}

\begin{document}
\title[Lifting Milnor Invariants for 3-Component Links]{Lifting Milnor Invariants for 3-Component Links}
%Send to: mathematische annalen 

\author{Christopher W.\ Davis}
\address{Department of Mathematics, University of Wisconsin--Eau Claire}
\email{daviscw@uwec.edu}

\author{JungHwan Park}
\address{Department of Mathematical Sciences, KAIST}
\email{jungpark0817@kaist.ac.kr}

% \thanks{$^{\dag}$This work was partially supported by Samsung Science and Technology Foundation SSTF-BA2102-02 and the NRF grant RS-2025-00542968.}

\date{\today}

%\subjclass[2000]{57M25}
\def\subjclassname{\textup{2020} Mathematics Subject Classification}
\expandafter\let\csname subjclassname@1991\endcsname=\subjclassname
%\expandafter\let\csname subjclassname@2000\endcsname=\subjclassname

\begin{abstract}
We define a sequence of integer-valued invariants $\gamma^k(L)$ for a $3$-component link $L$. We prove that the resulting $\gamma$-invariants are invariant under concordance, and more generally under weak cobordism, and that they lift
%Note from Chris:  I changes "it lifts" to "they lift" to agree with "invariants"
certain Milnor invariants of 3-component links. To establish this, we introduce an invariant $h(L)$, a $3$-component analogue of the Kojima--Yamasaki $\eta$-invariant, and show that it recovers the $\gamma$-invariants. As applications, we obtain a weak-cobordism classification when the distinguished component has trivial Alexander polynomial and characterize knots that bound continuously embedded disks in $B^4$ whose complements have fundamental group $\mathbb{Z}$.
\end{abstract}

\dedicatory{Dedicated to the memory of Tim D. Cochran.}

\maketitle
%==============================================================

\section{Introduction}

In 1954, John Milnor introduced the $\bar{\mu}$-invariants, higher-order linking numbers associated to a link
$L=(L_1,\ldots,L_m)\subset S^3$. They are defined inductively by comparing the lower central series of the
link group $\pi_1(S^3\smallsetminus L)$ with that of the unlink~\cite{Milnor1954}. Since then, they have been
studied extensively. For instance, building on Stallings' work~\cite{Sta65}, Casson showed that Milnor invariants are invariants of topological locally flat link concordance~\cite{Cas75}.  They have also been reinterpreted in a variety of ways; see, e.g.,~\cites{Dw, Tur76, Porter80, Orr89, Coc90, Hab1, HM00, CST14}, and have since been developed in several different directions; see, e.g.,~\cites{Orr:1991, Krushkal:1998, Cha:2006, CST:2012, CST:2017, Cha:2018, DNOP:2020, Park-Powell:2022, Stees:2025, Kuzbary:2024, Cha-Orr:2024, CST:2025}.

These invariants generalize the classical linking number. Indeed, the length-two Milnor invariants
$\bar{\mu}_L(ij)$ coincide with the pairwise linking numbers of the components $L_i$ and $L_j$ of $L$.
Given an $m$-component link $L$ and a multi-index $I=(i_1,i_2,\dots,i_n)$ with each
$i_k\in\{1,\ldots,m\}$, the length-$n$ Milnor invariant $\bar{\mu}_L(I)$ is defined modulo the greatest
common divisor of the lower-length invariants $\bar{\mu}_L(I')$ over all proper subindices $I'\subset I$.
For example, if $L$ is a two-component link with linking number $\bar{\mu}_L(12)=1$, then every
higher-order Milnor invariant of $L$ is completely indeterminate.

In 1985, Tim Cochran~\cite{C3} discovered a beautiful method for lifting Milnor invariants
of $2$-component links with vanishing linking number, introducing a sequence of integer-valued
invariants $\bigl(\beta^1(L),\beta^2(L),\dots\bigr)\in \Z^\infty$, called the \emph{$\beta$-invariants}.
These give lifts of certain Milnor invariants in the sense that
\[
\beta^k(L)\equiv \bar{\mu}_L\bigl(1^{2k}22\bigr)
\]
modulo the indeterminacy of Milnor invariants. (Here $1^{2k}=11\cdots 1$ is the multi-index of
length $2k$ whose every entry is $1$.)

In this article, we extend Cochran's idea to $3$-component links and produce a sequence
\[ \bigl(\gamma^0(L),\gamma^1(L),\gamma^2(L),\dots\bigr)\in \Z^\infty, \] which we call the \emph{$\gamma$-invariants}. These invariants are well-defined modulo a certain action on $\Z^\infty$.
The $\gamma$-invariants lift Milnor invariants in the sense that
\[
\gamma^k(L)\equiv \bar{\mu}_L\bigl(1^k23\bigr),
\]
 modulo the indeterminacy of Milnor invariants.  {In \cite{Tsukamoto-Yasuhara:2007}, Tsukamoto and Yasuhara consider the same
sequence under the notation $\alpha^k_F(L_2,L_3)$, where $F$ is a Seifert surface for $L_1$. They use this sequence to produce a factorization of the Conway polynomial of $L$, whereas our focus is to produce a concordance invariant that lifts Milnor's invariants.}

To state the main theorem precisely, we first introduce some terminology and notation. Unless otherwise specified, all links under consideration are ordered, oriented $3$-component links
$L=(L_1,L_2,L_3)\subset S^3$. 
We refer to the first component $L_1$ as the \emph{distinguished} component, and we further assume that
$\lk(L_1,L_2)=\lk(L_1,L_3)=0$.
We define $\gamma^0(L)$ to be the linking number of $L_2$ and $L_3$, that is,
\[
\gamma^0(L):=\lk(L_2,L_3)=\bar{\mu}_L(23).
\]
To define the higher invariants $\gamma^k(L)$, we require the notion of \emph{derivatives of links},
a geometric construction introduced in~\cite{C3}.

Given a link $L$, the derivative link $D(L)$ is defined by first finding Seifert surfaces $G_1$ and $G_2$ for $L_1$ and $L_2$ that intersect in a knot, with the extra assumption that $G_1$ is disjoint from $L_3$. The result of pushing this knot off of $G_1$ in the positive normal direction is denoted $L_{12}$. Replacing $L_2$ with $L_{12}$ yields a new $3$-component
link:
\[
D(L)=(L_1,L_{12},L_3).
\]
More generally, we define the $k$-fold derivative by iterating this process, obtaining
\[
D^k(L)=\bigl(L_1,L_{1^k2},L_3\bigr).
\]
We then define
\[
\gamma^k(L):=\bar{\mu}_{D^k(L)}(23)=\lk(L_{1^k2},L_3)
\]
and let $\gamma(L)=\bigl(\gamma^0(L),\gamma^1(L),\dots\bigr)\in \Z^\infty$ be the sequence of these
$\gamma$-invariants.

We remark that it is important to fix $G_1$  throughout this iterative process. Indeed, the precise value of the $\gamma$-invariants  depends on this choice of $G_1$, which introduces an indeterminacy.  When we want to emphasize this dependence on the Seifert surface we will write $D(L, G_1) = (L_1,L_{12}, L_3, G_1)$ and $\gamma^k(L, G_1)$. While we discuss this indeterminacy in more detail
in Section~\ref{sect eta invt}, we briefly summarize it here. Let $T\colon \Z^\infty\to\Z^\infty$ be the
\emph{right-shift operator},
\[
T(a^0,a^1,\dots):=(0,a^0,a^1,\dots).
\]
The sequence $\bigl(\gamma^0(L),\gamma^1(L),\dots\bigr)$ is an invariant of the concordance class of $L$
modulo the action of $T+\mathrm{Id}$.

\begin{theorem}\label{thm:main}
Let $L=(L_1,L_2,L_3)$ be a $3$-component link with a distinguished component. Then
\[
\gamma(L)=\bigl(\gamma^0(L),\gamma^1(L),\ldots\bigr)\in \Z^\infty
\]
is a link concordance invariant modulo the action of $T+\mathrm{Id}$. Moreover, we have
\[
\gamma^k(L)\equiv \bar{\mu}_L(1^k23)
\pmod{\gcd\left\{\, \bar{\mu}_L(1^i23)\mid i<k \,\right\}}.
\]
\end{theorem}

We remark that, given two sequences of $\gamma$-invariants, it is straightforward to determine whether they represent the same element modulo the action of $T+\mathrm{Id}$; see Corollary~\ref{cor:lift}. In fact, we will see that if $\gamma^k(L)$ is the first nonvanishing $\gamma$-invariant of $L$, then $\gamma^k$ is an invariant\footnote{The invariance of this first nonvanishing term was also proved in \cite[Corollary~1.5]{Tsukamoto-Yasuhara:2007}.} and $\gamma^{k+1}$ is well defined modulo $\gamma^k$, as in the case of Milnor invariants. After making a normalization to account for this indeterminacy, $\gamma^n\in \Z$ is well defined for all $n>k+1$. Therefore, we do indeed obtain examples that are indistinguishable by Milnor invariants but can be distinguished by the $\gamma$-invariants; see Section~\ref{subsec:computations}.

Cochran~\cite{C3} further related his $\beta$-invariants to a $\Q(t)$-valued invariant $\eta(L)$ due to
Kojima--Yamasaki~\cite{KojYam79}, where $\Q(t)$ denotes the field of rational functions. This invariant
can be interpreted as a self-linking invariant of the lift of $L_2$ to $\widetilde{E(L_1)}$, the infinite
cyclic cover of the exterior of $L_1$. Our proof of the invariance of $\gamma(L)$ proceeds by defining
$\noteta(L)\in \Q(t)$, an analogue of $\eta$ for $3$-component links, checking the invariance of  $\noteta(L)$, and then establishing the connection between $\noteta$ and $\gamma$.

More precisely, we define $\noteta(L)$ in terms of the linking of chosen lifts of $L_2$ and $L_3$ in $\widetilde{E(L_1)}$. Since this quantity depends on the choice of lifts, it is defined only up to an indeterminacy. In Proposition~\ref{prop h is invariant}, we show that $\noteta(L)$ nevertheless determines a well-defined element of $\sfrac{\Q(t)}{\doteq}$, where $\doteq$ is the equivalence relation generated by
\[
p(t)\doteq t p(t).
\]
We then prove that the Taylor coefficients of $\noteta(L)$ at $t=1$ recover the $\gamma$-invariants of~$L$.

\begin{theorem}\label{thm: gamma as h}
Let $L=(L_1,L_2,L_3)$ be a $3$-component link with a distinguished component. Then $\noteta(L)\in \sfrac{\Q(t)}{\doteq}$ is a link concordance invariant, and
\[
\noteta(L)\doteq \sum_{k=0}^\infty \gamma^k(L)\,(t-1)^k.
\]
\end{theorem}

\begin{remark}\label{rmk:iequiv}
The reader may notice that we have not specified whether we are working in the smooth or  locally flat category. This is because the invariants considered here are in fact invariants of the weakest natural notion of concordance for tame links, namely $I$-equivalence. For convenience, we work primarily in the smooth category.  In Remark~\ref{rmk: what if I-equiv} we explain how to remove the smoothness assumptions.
\end{remark}

To obtain the main results, we develop several structural and computational consequences of the $\gamma$-invariants. We place the construction in the framework of \emph{weak cobordism}, a weaker equivalence relation than concordance in which the distinguished components cobound an annulus and the remaining components cobound disjoint surfaces subject to a  homological condition. In this setting the derivative construction becomes well defined; see Section~\ref{sec:weakcobordism} for a precise definition. We also work out explicit examples. In Section~\ref{subsec:computations}, we show that arbitrary finitely supported integer sequences can be realized as $\gamma$-invariants, and we develop some computational tools.

% prove that the $\gamma$-invariants admit an explicit formula in terms of a Seifert matrix.

% \begin{theorem}\label{thm:seifert-matrix}
% Let $L=(L_1,L_2,L_3)$ be a $3$-component link whose distinguished component bounds a Seifert surface $G$ disjoint from $L_2$ and $L_3$. Let $\{a_1,\ldots,a_{2g}\}$ be a basis for $H_1(G)$ and let
% $\{\alpha_1,\ldots,\alpha_{2g}\}$ be the dual basis for $H_1(S^3\smallsetminus G)$. Let $V$ be the
% resulting Seifert matrix for $G$, and let $v_2,v_3\in \Z^{2g}$ be the column vectors representing
% $[L_2]$ and $[L_3]$ in $H_1(S^3\smallsetminus G)$ with respect to $\{\alpha_i\}$. Set $A:=V-V^T$.
% Then, for every positive integer $k$,
% \[
% \gamma^k(L,G)
% =
% \Bigl(A^{-1}(VA^{-1})^{k-1}v_2\Bigr)^T v_3.
% \]
% \end{theorem}

We also connect the $\gamma$-invariants to Cochran's $\beta$-invariants by showing that the latter can be recovered explicitly from the $\gamma$-invariants of the $3$-component link obtained by adjoining a $0$-framed push-off of the second component:

\begin{theorem}\label{thm:beta-from-gamma}
Let $L=(L_1,L_2)$ be a $2$-component link with $\lk(L_1,L_2)=0$. Let $G$ be a Seifert surface for
$L_1$ disjoint from $L_2$, let $L_2^0$ denote the $0$-framed push-off of $L_2$, and let
$L^\ast=(L_1,L_2,L_2^0)$. Then
\[
\beta^k(L)=(-1)^k\sum_{j=1}^k \binom{k-1}{j-1}\gamma^{k+j}(L^\ast,G).
\]
\end{theorem}

% We also consider mixed iterated derivatives obtained by alternating the second and third components. This l
% These mixed derivatives carry no new information: they are determined by the original $\gamma$-invariant,
% and in particular they recover Cochran's $\beta$-invariants.

We also develop a broader classification framework for weak cobordism, where it is important to specify the category under consideration. More precisely, the definition depends on whether the embeddings involved are smooth, locally flat, or continuous; see Section~\ref{sect: classify}. Since the full statement involves additional data from the infinite cyclic cover, we highlight here the particularly clean case in which the distinguished component has trivial Alexander polynomial. Let $\sim$ denote the equivalence relation on $\Z^\infty$ generated by
\(
a\sim (T+\operatorname{Id})(a).
\)

\begin{corollary}\label{cor:alexander-one}
Let $L=(L_1,L_2,L_3)$ and $L'=(L_1',L_2',L_3')$ be links with a distinguished component. Suppose that $L_1$ and $L_1'$ have trivial Alexander polynomial. Then the following are equivalent:
\begin{enumerate}
\item \label{item: same invts}
$\beta(L_1,L_2)=\beta(L_1',L_2')$, $\beta(L_1,L_3)=\beta(L_1',L_3')$, and
$\gamma(L)=\gamma(L')\in\Z^\infty /\sim $.

\item \label{item: top weak cob}
$L$ and $L'$ are continuously weakly cobordant.

\item \label{item: LF weak cob}
$L$ and $L'$ are locally flat weakly cobordant.
\end{enumerate}
If we additionally assume that $L_1$ and $L_1'$ cobound a smoothly embedded annulus in $S^3\times[0,1]$, then these are further equivalent to:
\begin{enumerate}\setcounter{enumi}{3}
\item \label{item: smooth weak cob}
$L$ and $L'$ are smoothly weakly cobordant.
\end{enumerate}
\end{corollary}

Finally, we show that the perspective of studying linking numbers of lifts in the infinite cyclic cover has consequences beyond the $3$-component setting. In particular, it yields a characterization of those knots that bound an embedded disk in $B^4$ whose complement has fundamental group $\mathbb{Z}$.

\begin{theorem}\label{thm:pi1-Z-disk}
Let $K\subseteq S^3$ be a knot. The following are equivalent:
\begin{enumerate}
\item \label{item:alexander-one}
$K$ has trivial Alexander polynomial.

\item \label{item:locally-flat-Z-disk}
$K$ bounds a locally flat embedded disk $D\subset B^4$ with
$\pi_1(B^4\smallsetminus D)\cong \Z$.

\item \label{item:continuous-Z-disk}
$K$ bounds a continuously embedded disk $D\subset B^4$ with
$\pi_1(B^4\smallsetminus D)\cong \Z$.
\end{enumerate}
\end{theorem}

The equivalence of \eqref{item:alexander-one} and \eqref{item:locally-flat-Z-disk}
is the well-known theorem of Freedman~\cite{Freedman1, Freedman-Quinn:1990-1, Garoufalidis-Teichner:2004-1},
and the implication \eqref{item:locally-flat-Z-disk} $\Rightarrow$
\eqref{item:continuous-Z-disk} is immediate. Thus the new content is the implication
\eqref{item:continuous-Z-disk} $\Rightarrow$ \eqref{item:alexander-one}: a continuously
embedded disk $D\subset B^4$ with
$\pi_1(B^4\smallsetminus D)\cong \mathbb{Z}$
already implies that $K$ has trivial Alexander polynomial.

\begin{remark}\label{rmk:homology-sphere}
All of our techniques are built on computations in the homology groups of knot
complements, complements of disks in the $4$-ball, complements of embedded
annuli in $S^3\times[0,1]$, and their infinite cyclic covers. Thus, all of the
results of this paper apply equally well if every instance of $S^3$, $B^4$, and
$S^3\times[0,1]$ is replaced, respectively, by an integral homology sphere, an
integral homology ball, and an integral homology cobordism.
\end{remark}

\subsection*{Notation and conventions}
Throughout the article, all links are in $S^3$, and they are ordered and oriented. Homology is taken with $\mathbb{Z}$ coefficients unless otherwise specified. Given a knot $K$, we use $-K$ to denote the reverse of the mirror image of $K$.  {If $L=(L_1,L_2,L_3)$ is a link with distinguished component $L_1$, then any Seifert surface for $L_1$ will be required to be disjoint from $L_2$ and $L_3$.  Similarly any Seifert surfaces for $L_2$ and $L_3$ will be disjoint from $L_1$.  When we specify Seifert surfaces for $L_1$ and $L_2$, then unless we specify otherwise,  we will assume that they intersect in a knot. }
%\footnote{\chris{Use this to simplify theorem statements.} \JP{For the statment in the body (not in the intro) I think it is okay to be detailed because often readers forget about conventions. But I am okay with leaving this red portion in the intro and not changing the statements in the body.}}

% \subsection*{Organization}
% Section~2 introduces weak cobordism for $3$-component links, defines the relative invariants
% $\gamma^k(L,G)$, develops examples, and proves the Seifert-matrix formula above. Section~3 studies
% swapping and mixed derivatives, showing that they yield no new information and recovering Cochran's
% $\beta$-invariants from the $\gamma$-invariants. Section~4 defines the invariant $\noteta(L)$ in the infinite
% cyclic cover and proves that its Taylor coefficients recover $\gamma(L)$. Section~5 proves that
% $\gamma^k(L)$ lifts $\bar{\mu}_L(1^k23)$ using Massey products. Section~6 develops a classification framework
% for weak cobordism and obtains a clean characterization when the distinguished component has trivial
% Alexander polynomial. Section~7 applies the same viewpoint to embedded disks with $\pi_1=\Z$.

\subsection*{Organization}
We begin in Section~2 by introducing weak cobordism for $3$-component links, defining the invariants $\gamma^k(L,G)$, which depend on the choice of a Seifert surface $G$ for the distinguished component, and presenting examples. Section~3 studies 
the effect of swapping components $L_2$ and $L_3$, finds that the resulting $\gamma$-invariants carry the same information, and uses this philosophy to determine Cochran's $\beta$-invariants from the $\gamma$-sequence,
%swapping and mixed derivatives; we show that they yield no new information and use this to recover Cochran's $\beta$-invariants from the $\gamma$-sequence, 
proving Theorem~\ref{thm:beta-from-gamma}. In Section~4, we define the invariant $h(L)$ and show that its Taylor expansion recovers $\gamma(L)$, thereby proving Theorem~\ref{thm: gamma as h}. Section~5 then shows, via Massey products, that $\gamma^k(L)$ lifts the Milnor invariant $\bar{\mu}_L(1^k23)$, completing the proof of Theorem~\ref{thm:main}. Section~6 develops a classification framework for weak cobordism,  with Corollary~\ref{cor:alexander-one} as a consequence. Finally, Section~7 applies the viewpoint of our paper to knots bounding embedded disks in $B^4$ whose complements have fundamental group isomorphic to $\mathbb{Z}$, establishing Theorem~\ref{thm:pi1-Z-disk}.  In Section~\ref{sect: problems}, we close with some natural questions for further study.  

\subsection*{Acknowledgements}
The second author is partially supported by the Samsung Science and Technology Foundation
(SSTF-BA2102-02) and by the NRF grant RS-2025-00542968.

\section{Weak cobordism, the $\gamma$-invariant and some examples}\label{sec:weakcobordism}

We begin this section by discussing ideas fundamental to~\cite{C3}. There, Cochran introduced an operation on a $2$-component link, which he called a \emph{derivative}. This operation replaces a link $(L_1,L_2)$ with $\lk(L_1,L_2)=0$ by a new link $(L_1,L_{12})$, obtained by choosing Seifert surfaces for $L_1$ and $L_2$ that intersect in a simple closed curve and then taking this curve as $L_{12}$. This operation is not well defined, since it depends on the choice of Seifert surfaces. However, it becomes well defined modulo a weaker equivalence relation than concordance, called \emph{weak cobordism}.

Consequently, the derivative operation may be iterated to produce a sequence of weak cobordism classes of links
\[
(L_1,L_2),\ (L_1,L_{12}),\ (L_1,L_{112}),\ \dots,
\]
each of which is an invariant of the original link.\footnote{We note that, in~\cite{C3}, it is not necessary to use the same Seifert surface for the distinguished component at each stage. However, as mentioned in the introduction, 
%the definition of the $\gamma$-invariant 
our work
requires us to fix a Seifert surface.}
The derived component $L_{1^k2}$ inherits a framing coming from the two Seifert surfaces on which it lies. Cochran then defines a sequence of numerical invariants $\beta^k(L)$ by recording the linking number between $L_{1^k2}$ and a push-off determined by this framing. Our work is based on adapting these notions to the study of $3$-component links.

\begin{definition}\label{defn:weak cob}
Let $L=(L_1, L_2, L_3)$ and $L'=(L_1', L_2', L_3')$ be $3$-component links satisfying $\lk(L_1,L_i)=\lk(L_1',L_i')=0$ for $i=2,3$. We say that $L$ is \emph{weakly cobordant} to $L'$ and write $L\simeq_w L'$ if the following hold:
\begin{enumerate}
\item \label{item: weak cob - concordance} $L_1\times\{0\}$ and $L_1'\times\{1\}$ cobound in $S^3\times[0,1]$ a smoothly embedded annulus $A$.
\item \label{item: weak cob - cobordism} For $i=2,3$, $L_i\times\{0\}$ and $L_i'\times \{1\}$ cobound in $S^3\times[0,1]$ a compact, oriented, smoothly embedded surface $Y_i$.
\item \label{item: weak cob - disjoint} The surfaces $A$, $Y_2$, and $Y_3$ are pairwise disjoint.
\item \label{item: weak cob - lifting} For $i=2,3$, the inclusion-induced map
\[
H_1(Y_i)\to H_1(S^3\times[0,1]\smallsetminus A)\cong \Z
\]
is the zero homomorphism.
\end{enumerate}
\end{definition}

Recall that we refer to the first component $L_1$ as the \emph{distinguished component}.

\begin{proposition}\label{prop: equiv to weak cobordism}
Let $L$ and $L'$ be $3$-component links with a distinguished component. Let $A$, $Y_2$, and $Y_3$ be surfaces satisfying conditions \pref{item: weak cob - concordance}, \pref{item: weak cob - cobordism}, and \pref{item: weak cob - disjoint} of Definition~\ref{defn:weak cob}. Then condition \pref{item: weak cob - lifting} is equivalent to each of the following:
\begin{enumerate}[label=(\Alph*)]
\item \label{item weak cob lifting A}
There exist Seifert surfaces $G$ and $G'$ for $L_1$ and $L_1'$, respectively, such that $G$ is disjoint from $L_2\cup L_3$ and $G'$ is disjoint from $L_2'\cup L_3'$, and
\[
[G\cup A\cup -G']=0 \in H_2\bigl(S^3\times[0,1]\smallsetminus (Y_2\cup Y_3)\bigr).
\]
\item \label{item weak cob lifting B}
There exist Seifert surfaces $G$ and $G'$ for $L_1$ and $L_1'$, respectively, such that $G\cup A\cup -G'$ bounds a compact, oriented, embedded $3$-manifold $N$ in $S^3\times[0,1]$ that is disjoint from $Y_2$ and $Y_3$.
\end{enumerate}
\end{proposition}

\begin{proof}

The equivalence of \ref{item weak cob lifting A} and \ref{item weak cob lifting B} is clear. It is also clear that \ref{item weak cob lifting B} implies \pref{item: weak cob - lifting}, since the $3$-manifold $N$ is necessarily Poincar\'e dual to the meridian of $L_1$, which generates $H_1(S^3\times[0,1]\smallsetminus A)$.

To obtain the converse, assume that all of the conditions of Definition~\ref{defn:weak cob} hold, and fix Seifert surfaces $G$ and $G'$ for $L_1$ and $L_1'$, respectively. Since $H_2(S^3\times[0,1])=0$, there exists a compact, oriented $3$-manifold $N$ with
\[
\partial N = G\cup A\cup -G'.
\]
If necessary, we may attach tubes to $N$ along arcs in $S^3\times[0,1]$ to arrange that $N$ is connected.

Since $N$ and each $Y_i$ intersect transversely, the intersection
\[
y_i = Y_i\cap N
\]
is a union of simple closed curves. If $y_i$ does not separate $Y_i$, then there exists some $\alpha\in H_1(Y_i)$ with nonzero algebraic intersection number with $y_i$. It follows that $\alpha$ has nonzero intersection with $N$ in $S^3\times[0,1]$, contradicting the assumption that the inclusion-induced map
\[
H_1(Y_i)\to H_1(S^3\times[0,1]\smallsetminus A)
\]
is the zero homomorphism. Thus $y_i$ separates $Y_i$, and hence $[y_i]=k[L_i]$ in $H_1(Y_i)$ for some $k\in \Z$. Our next step is to modify the Seifert surface $G$ and the $3$-manifold $N$ so as to arrange that $k=0$.

Let $\nu(L_i)\subseteq S^3\times \{0\}$ be a closed tubular neighborhood of $L_i$, and let $T_i=\bdry \nu(L_i)$ be its boundary. Choose an arc $a$ from $G$ to $T_i$ in $S^3\smallsetminus L$, and let $G^*$ be the result of tubing $G$ and $T_i$ together along $a$, as in Figure~\ref{fig:tubing}. Let $N^*$ be the $3$-manifold cobordism from $G^*$ to $G$ obtained by pushing the interior of $G^*$ slightly into $S^3\times[0,\epsilon]$ and then taking the union with $\nu(L_i)\cup \nu(a)\subseteq S^3\times\{\epsilon\}$. Note that
\[
N^*\cap \bigl(L_i\times[0,\epsilon]\bigr) = L_i\times\{\epsilon\}.
\]
Consequently, $(N\cup N^*)\cap Y_i$ is homologous in $Y_i$ to $k+1$ copies of $L_i$ (or to $k-1$ copies if we use the opposite orientation on $T_i$). Iterating this construction, we may arrange that $Y_i\cap N$ represents $0$ in $H_1(Y_i)$.

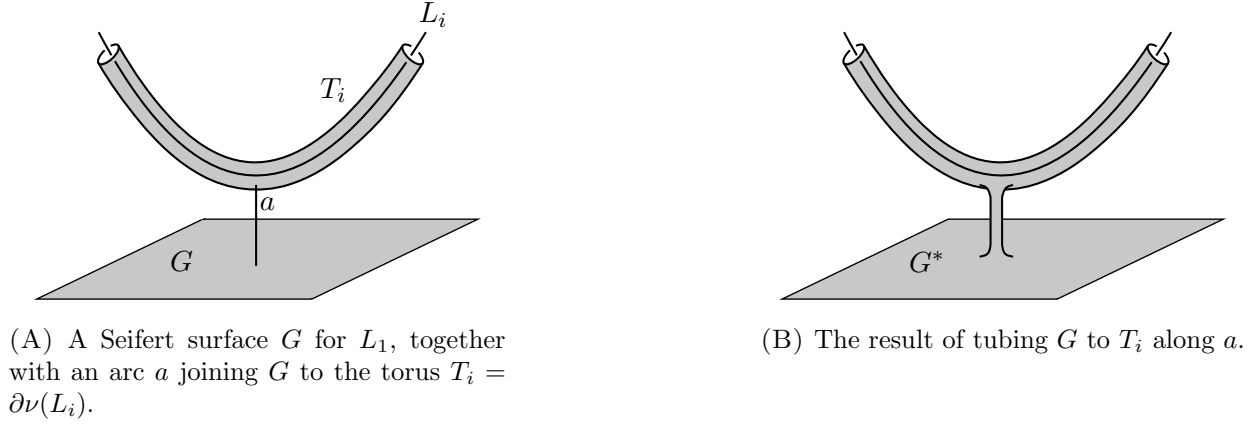
\begin{figure}[!htbp]
     \centering
     \begin{subfigure}[t]{0.4\textwidth}
     \centering
         \begin{tikzpicture}
         \node at (0,0){\includegraphics[width=.9\textwidth]{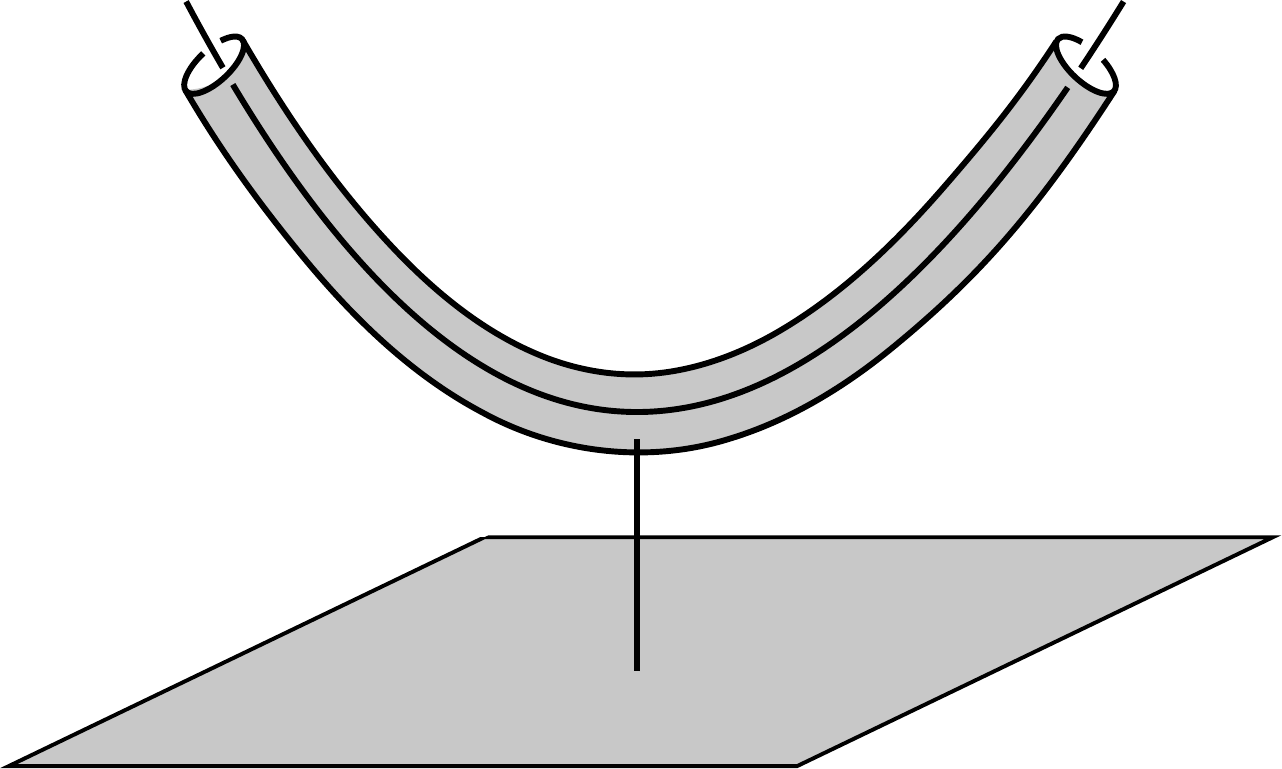}};
         \node[below] at (-1,-1) {$G$};
         \node[right] at (-.1,-.5) {$a$};
         \node[right] at (.7,1) {$T_i$};
         \node[right] at (2,2) {$L_i$};
         \end{tikzpicture}
         \caption{A Seifert surface $G$ for $L_1$, together with an arc $a$ joining $G$ to the torus $T_i=\partial\nu(L_i)$.}\label{fig:pretubing}
         \end{subfigure}
         \hfill
     \begin{subfigure}[t]{0.4\textwidth}
     \centering
         \begin{tikzpicture}
         \node at (0,0){\includegraphics[width=.9\textwidth]{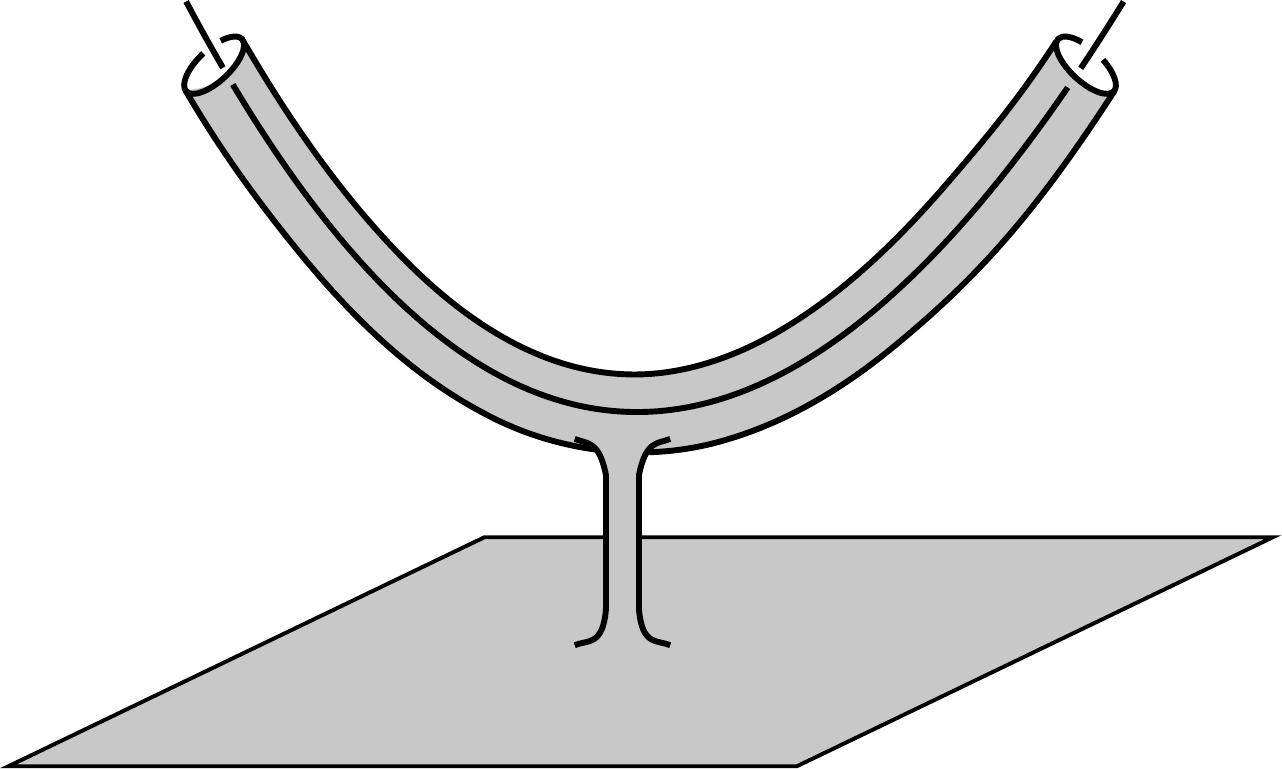}};
         \node[below] at (-1,-1) {$G^*$};
         \end{tikzpicture}
         \caption{The result of tubing $G$ to $T_i$ along $a$.}\label{fig:posttubing}
         \end{subfigure}
     \caption{Tubing a Seifert surface $G$ to the torus $T_i$.}\label{fig:tubing}
     \end{figure}

Thus $Y_i\cap N$ bounds a compact surface $Y_i^0\subseteq Y_i$. We may now modify $N$ by an ambient surgery as follows. Remove $\nu(Y_i\cap N)$ from $N$ and replace it with the boundary of a tubular neighborhood of $Y_i^0$. This completes the proof that \pref{item: weak cob - lifting} implies \ref{item weak cob lifting B}.
\end{proof}

% Notice in the proof above that the choice of Seifert surfaces $G$ and $G'$ appears to matter, so we introduce a definition to reflect this.

We introduce the following definition.

\begin{definition}\label{defn: weak cob surf}
Let $L$ and $L'$ be weakly cobordant $3$-component links with a distinguished component. Let $G$ and $G'$ be Seifert surfaces for $L_1$ and $L_1'$. If these surfaces satisfy the conditions of Proposition~\ref{prop: equiv to weak cobordism}, then we say that $(L,G)$ is \emph{weakly cobordant} to $(L',G')$ and write $(L,G)\simeq_w (L',G')$.
\end{definition}

Observe that Proposition~\ref{prop: equiv to weak cobordism} can now be read as saying that $L\simeq_w L'$ if and only if there exist Seifert surfaces $G$ and $G'$ such that $(L,G)\simeq_w (L',G')$. In~\cite[Proposition~3.3]{C3}, Cochran shows that the choice of $G$ does not matter for $2$-component links; the argument uses that $H_2(S^3\smallsetminus L_2)=0$. We cannot argue similarly here, since
$H_2(S^3\smallsetminus (L_2\cup L_3))\cong \Z$.
Thus the choice of Seifert surface carries an additional integer ambiguity. The
following proposition gives a useful criterion for when two choices give weakly cobordant pairs.

\begin{proposition}\label{prop: When are surfaces weak cob}
Let $L=(L_1, L_2, L_3)$ be a $3$-component link with a distinguished component. Let $G$ and $G'$ be Seifert surfaces for $L_1$.
%, both disjoint from $L_2$ and $L_3$.
If $[G\cup -G']=0$ in $H_2\bigl(S^3\smallsetminus (L_2\cup L_3)\bigr)$, then $(L,G)$ is weakly cobordant to $(L,G')$.
\end{proposition}

\begin{proof}
We construct a weak cobordism. Let $A=L_1\times[0,1]$ and, for $i=2,3$, let $Y_i=L_i\times[0,1]$. Since $[G\cup -G']=0$ in $H_2\bigl(S^3\smallsetminus (L_2\cup L_3)\bigr)$, it follows that
\[
[G\cup A\cup -G']=0 \in H_2\bigl(S^3\times[0,1]\smallsetminus (Y_2\cup Y_3)\bigr).
\]
This is exactly condition~\ref{item weak cob lifting A} of Proposition~\ref{prop: equiv to weak cobordism}, completing the proof.
\end{proof}

 % The group $H_2\bigl(S^3\smallsetminus (L_2\cup L_3)\bigr)\cong \Z$ is generated by the boundary of a tubular neighborhood of $L_2$. Denote this torus by $T_2$. Thus, if $G$ and $G'$ are any Seifert surfaces for $L_1$ disjoint from $L_2$ and $L_3$, then in $H_2\bigl(S^3\smallsetminus (L_2\cup L_3)\bigr)$ the classes $[G]$ and $[G']$ differ by tubing in some number of parallel copies of $T_2$, as in Figure~\ref{fig:tubing}.

 The group $H_2\bigl(S^3\smallsetminus (L_2\cup L_3)\bigr)\cong \Z$ is generated by the boundary of a tubular neighborhood of $L_2$. Denote this torus by $T_2$. Thus, if $G$ and $G'$ are Seifert surfaces for $L_1$ disjoint from $L_2$ and $L_3$, then the closed surface $G\cup -G'$ represents an integer multiple of $[T_2]$. Equivalently, up to this integer ambiguity, changing the Seifert surface amounts to tubing in parallel copies of $T_2$, as in Figure~\ref{fig:tubing}.

\begin{corollary}
Let $L$ and $L'$ be weakly cobordant links, and let $G$ and $G'$ be Seifert surfaces for $L_1$ and $L_1'$, respectively.
%, such that $G$ is disjoint from $L_2\cup L_3$ and $G'$ is disjoint from $L_2'\cup L_3'$.
For any integer $n$, let $G^n$ denote the surface obtained from $G$ by tubing it to $n$ parallel copies of $T_2$; for $n<0$, we use oppositely oriented copies of $T_2$. Then there exists an integer $n$ such that $(L,G^n)$ is weakly cobordant to $(L',G')$. \qed
\end{corollary}

We adapt Cochran's derivative construction~\cite[Section~4]{C3} to our setting as follows.

% As mentioned earlier, in~\cite[Section~4]{C3} Cochran introduced a notion called a \emph{derivative}. We adapt this construction to our setting as follows.

\begin{definition}\label{defn:deriv}
Let $L=(L_1, L_2, L_3)$ be a $3$-component link with a distinguished component, $G_1$ be a Seifert surface for $L_1$ 
%disjoint from $L_2$ and $L_3$,
 and  $G_2$ be a Seifert surface for $L_2$. 
 %disjoint from $L_1$. 
 After adding handles to $G_2$, arrange that $G_1\cap G_2$ is a single simple closed curve (that is, a knot). This knot inherits an orientation from the orientations of $G_1$ and $G_2$. To be explicit, let $n_{G_i}$ denote the positive normal vector to $G_i$. We orient $L_{12}=G_1\cap G_2$ so that its positive tangent vector is the cross product $n_{G_1}\times n_{G_2}$. The result of pushing this curve off both $G_1$ and $G_2$ in the positive normal direction is called the \emph{derivative} of $L$ with respect to $G_1$ and $G_2$, and we denote it by $L_{12}$. We then define
\[
D(L,G_1)=\left(L_1,\, L_{12},\, L_3,\, G_1\right).
\]
\end{definition}

%Note that the handles added above are disjoint from $L$, and hence preserve the weak cobordism class of $(L,G_1)$ by Proposition~\ref{prop: When are surfaces weak cob}. Furthermore, 
While the definition appears to depend on both $G_1$ and $G_2$, the following lemma shows that the choice of $G_2$ does not matter, at least up to weak cobordism. Compare with~\cite[Theorem~4.2]{C3}.

\begin{lemma}\label{lem:independentofG2}
Let $L$ and $L'$ be links with a distinguished component. Let $G_1$, $G_1'$, $G_2$, and $G_2'$ be Seifert surfaces for $L_1$, $L_1'$, $L_2$, and $L_2'$, respectively.
%, with $G_1\cap L_i=G_1'\cap L_i'=\emptyset$ for $i=2,3$ and with $L_1\cap G_2=L_1'\cap G_2'=\emptyset$.
Let $L_{12}$ and $L_{12}'$ be the resulting derivatives. If $(L,G_1)$ is weakly cobordant to $(L',G_1')$, then $D(L,G_1)$ is weakly cobordant to $D(L',G_1')$.
\end{lemma}

\begin{proof}
The proof is essentially the same as~\cite[Theorem~4.2]{C3}, with only some bookkeeping needed to track the component $L_3$. We summarize the main ideas.

Let $A$, $Y_2$, and $Y_3$ be the surfaces from Definition~\ref{defn:weak cob}, and let $N$ be a $3$-manifold bounded by $G_1\cup A\cup -G_1'$. Since $H_2(S^3\times[0,1])=0$, the surface $G_2\cup Y_2\cup -G_2'$ bounds a $3$-manifold $N_2$. By transversality, $N\cap N_2$ is a properly embedded surface. If we push $N\cap N_2$ off both $N$ and $N_2$ in the positive normal directions, we obtain a surface bounded by $L_{12}\cup -L_{12}'$; denote this surface by $Y_{12}$. Since $N$ is disjoint from $Y_3$, the surface $Y_{12}$ is disjoint from $A$ and $Y_3$. Thus $(L_1, L_{12}, L_3, G_1)$ is weakly cobordant to $(L_1', L_{12}', L_3', G_1')$, as claimed.
\end{proof}

According to Lemma~\ref{lem:independentofG2}, the assignment $(L,G)\mapsto D(L,G)$ gives a well-defined function modulo weak cobordism. We are now ready to define our invariant. For any link $L=(L_1,L_2,L_3)$ with a distinguished component, and any Seifert surface $G$ for $L_1$ disjoint from $L_2$ and $L_3$, we write the $k$-fold iterate of the derivative operator as
\[
D^k(L,G)=(L_1,L_{1^k2},L_3,G).
\]
We then define
\[
\gamma^k(L,G):=\lk(L_{1^k2},L_3).
\]
Since linking number is invariant under weak cobordism and the derivative operator is well defined on weak cobordism classes, we obtain the following:

\begin{theorem}\label{thm: gamma is invt of weak cob with surface}
For each $k\in \N$, the integer $\gamma^k(L,G)\in \Z$ depends only on the weak cobordism class of $(L,G)$. \qed
\end{theorem}

Of course, we do not want an invariant that depends on the choice of Seifert surface $G$. In Section~\ref{sect eta invt}, we prove Theorem~\ref{thm:main}, which explains how to extract from $\gamma^k$  an invariant independent of this choice. For now, we carry out some computations, both to build intuition and to illustrate the computability of these invariants.

%\begin{proof}
%As we have already proven that the derivative operator (and so any iterate of the derivative operator) is well defined on weak cobordism, we need only prove that $\lk(L_2,L_3)$ is invariant under weak cobordism. 

%This is a straightforward exercise in what linking number means,  CITE ROLFSEN but we include a summary.  Let $Y_i$ be the surface bounded by $L_i\cup -L_i'$ and $G_2,G_3, G_2',G_3'$ be any Seifert surfaces. Push $G_3$ and $G_3'$ slightly into $S^3\times[0,1]$.  

%Then $Z_i=G_i\cup Y_i\cup G_i'$ is a surface in $S^3\times[0,1]$ for $i=2,3$. The algebraic intersection is now given by $Z_2\cdot Z_3 = L_2\cdot G_3 - L_2'\cdot G_3'$.  Since $H_2(S^3\times[0,1])=0$, the algebraic intersection is zero.  The proof is completed by recalling that  $L_2\cdot G_3 = \lk(L_2,L_3)$ and $L_2'\cdot G_3'=\lk(L_2,L_3')$.  

%\end{proof}

\subsection{Some computations}\label{subsec:computations}We begin with an example that illustrates how easily $\gamma^k(L,G)$ can be computed. Consider the link $L=(L_1, L_2, L_3)$ in Figure~\ref{fig:example1}, together with the Seifert surface $G=G_1$ for $L_1$. Then
\[
\gamma^0(L,G_1)=\lk(L_2,L_3)=1.
\]
As in Figure~\ref{fig:example1SeifSurf}, the component $L_2$ bounds a genus one Seifert surface $G_2$ that is disjoint from $L_1$, obtained by starting with a disk intersecting $L_1$ in two points with opposite signs and tubing those two points together. Recall that the derivative $L_{12}=G_1\cap G_2$ is oriented so that its positive tangent vector is $n_{G_1}\times n_{G_2}$. For the resulting derivative $L_{12}$, we have
\[
\gamma^1(L,G_1)=\lk(L_{12},L_3)=-1.
\]
Notice that $L_{12}$ is homologous to the reverse of $L_2$ in the complement of $G_1$. From here, a direct induction shows that
\[
\gamma^k(L,G_1)=(-1)^k
\]
for all $k$, so $\gamma(L,G_1)=(1,-1,1,-1,\dots)$.

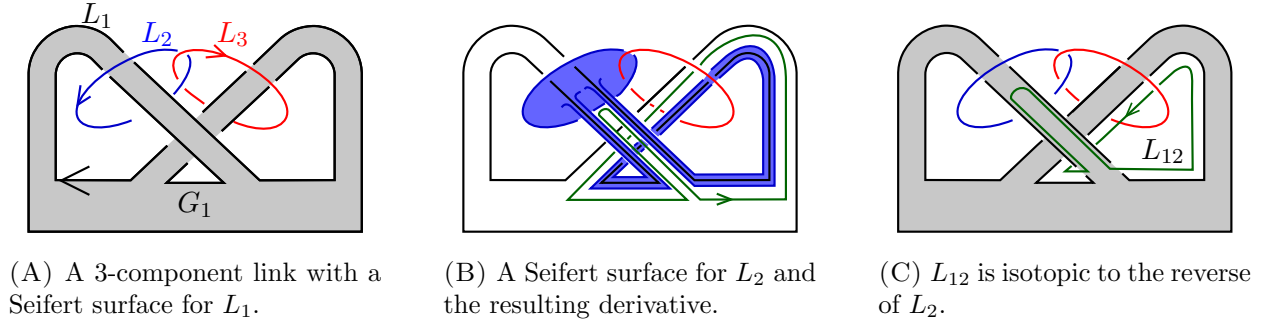
\begin{figure}[!htbp]
     \centering
     \begin{subfigure}[t]{0.3\textwidth}
         \centering
         \begin{tikzpicture}
         \node at (0,0){\includegraphics[width=.9\textwidth]{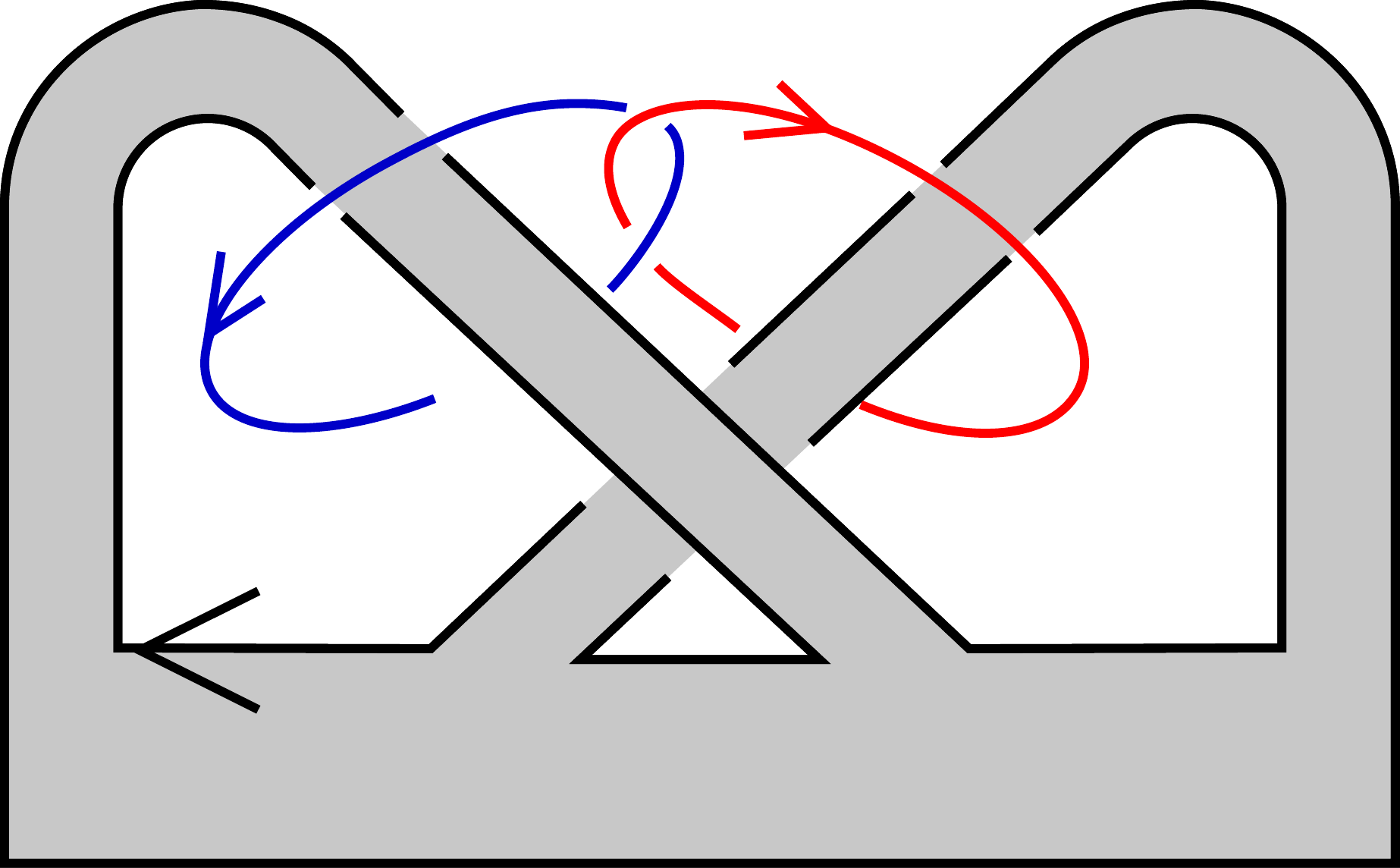}};
         \node at (-1.3,1.5){$L_1$};
         \node at (0,-1){$G_1$};
         \node[blue] at (-.5,1.25){$L_2$};
         \node[red] at (.5,1.25){$L_3$};
         \end{tikzpicture}
         \caption{A $3$-component link with a Seifert surface for $L_1$.}\label{fig:example1}
     \end{subfigure}
     \hfill
     \begin{subfigure}[t]{0.3\textwidth}
     \centering
         \begin{tikzpicture}
         \node at (0,0){\includegraphics[width=.9\textwidth]{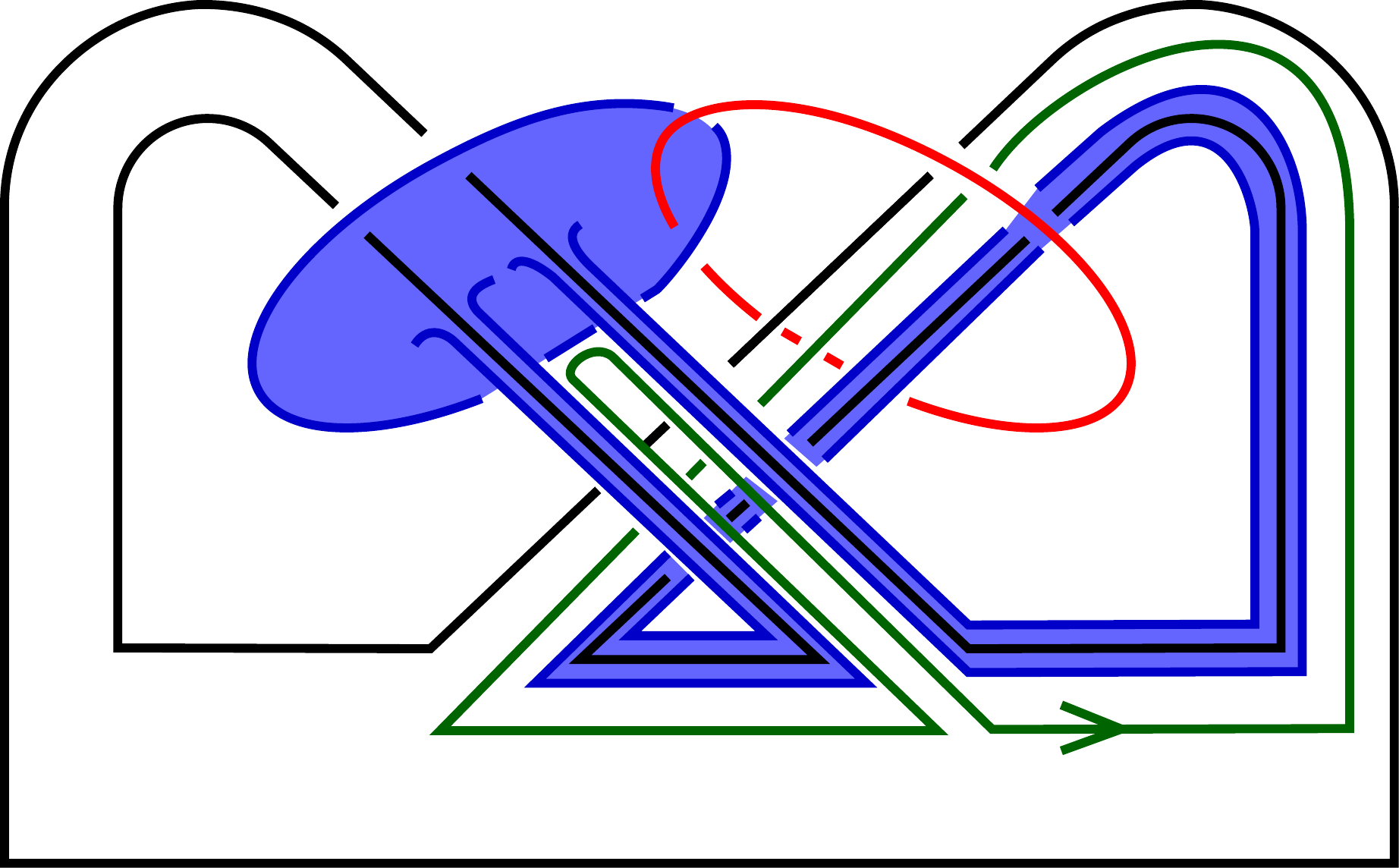}};
         \end{tikzpicture}
         \caption{A Seifert surface for $L_2$ and the resulting derivative.}\label{fig:example1SeifSurf}
         \end{subfigure}
         \hfill
     \begin{subfigure}[t]{0.3\textwidth}
     \centering
         \begin{tikzpicture}
         \node at (0,0){\includegraphics[width=.9\textwidth]{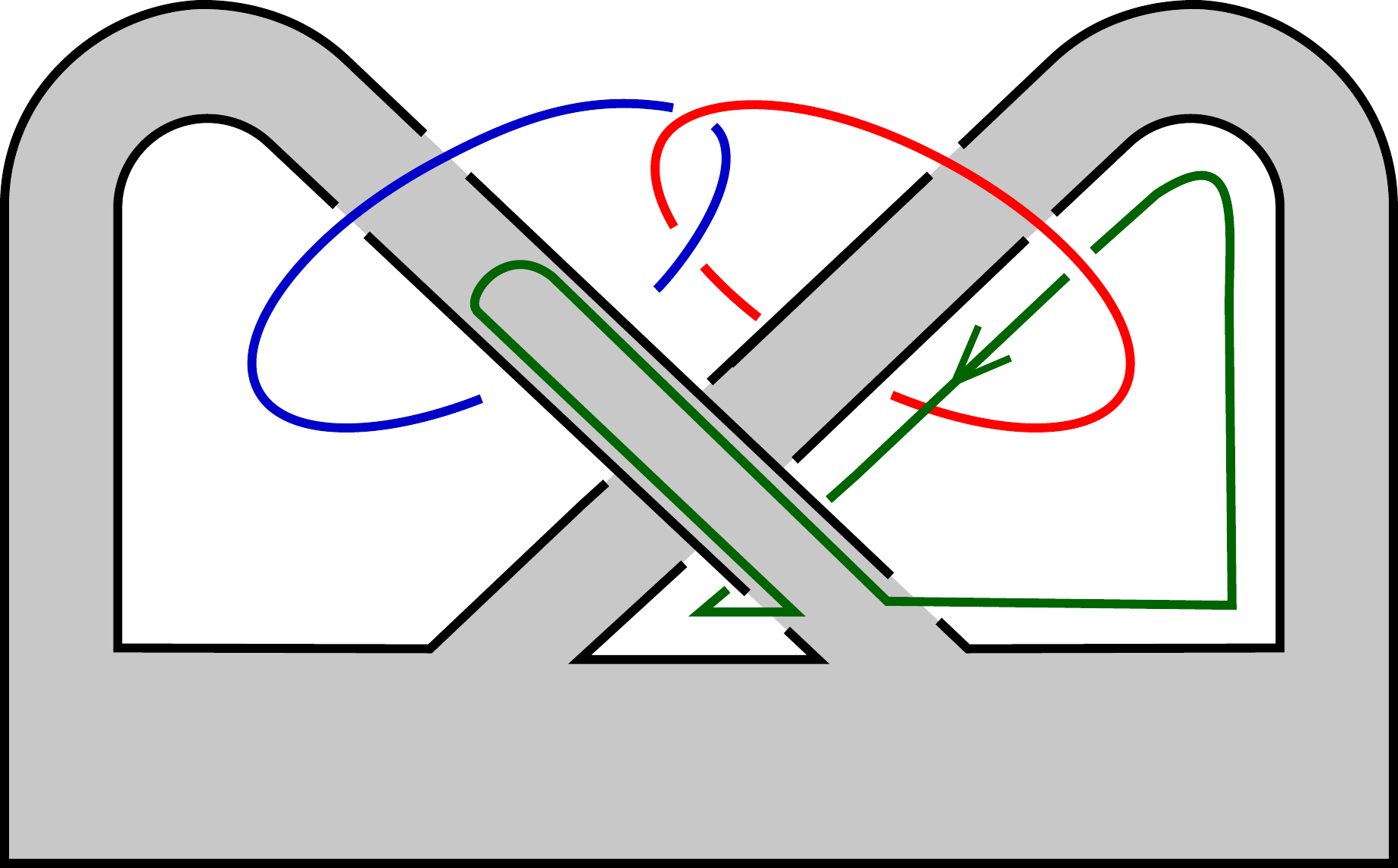}};
         \node at (1.3,-.3) {$L_{12}$};
         \end{tikzpicture}
         \caption{$L_{12}$ is isotopic to the reverse of $L_2$.}\label{fig:example1Deriv}
         %\label{fig:claspersurgery}
     \end{subfigure}
     \caption{A $3$-component link with a chosen Seifert surface and the resulting derivative.}\label{fig: Example1Deriv}
     \end{figure}

The next example illustrates that $\gamma^k(L,G)$ can depend on the choice of Seifert surface $G$. In Figure~\ref{fig:example1DiffSurf} we consider the same link $L=(L_1, L_2, L_3)$ as in Figure~\ref{fig:example1}, but with a different Seifert surface $G_1'$.  This results in $G_1'\cap G_2$ being disconnected.  Adding a tube to $G_2$ results in the derivative $L_{12}$ seen in Figure~\ref{fig:example1DiffSurfDeriv}. Then
\[
\gamma^1(L,G_1')=\lk(L_{12},L_3)=0.
\]
In fact, since $L_{12}$ is null-homologous in the exterior of $G_1'$, we have $\gamma^k(L,G_1')=0$ for all $k>0$. This is consistent with Theorem~\ref{thm:main}, since $(T+\Id)(1,-1,1,-1,\dots)=(1,0,0,\dots)$. Recall that $T$ denotes the right shift operator.

\begin{figure}[!htbp]
     \centering
     \begin{subfigure}[t]{0.3\textwidth}
         \centering
         \begin{tikzpicture}
         \node at (0,0){\includegraphics[width=.85\textwidth]{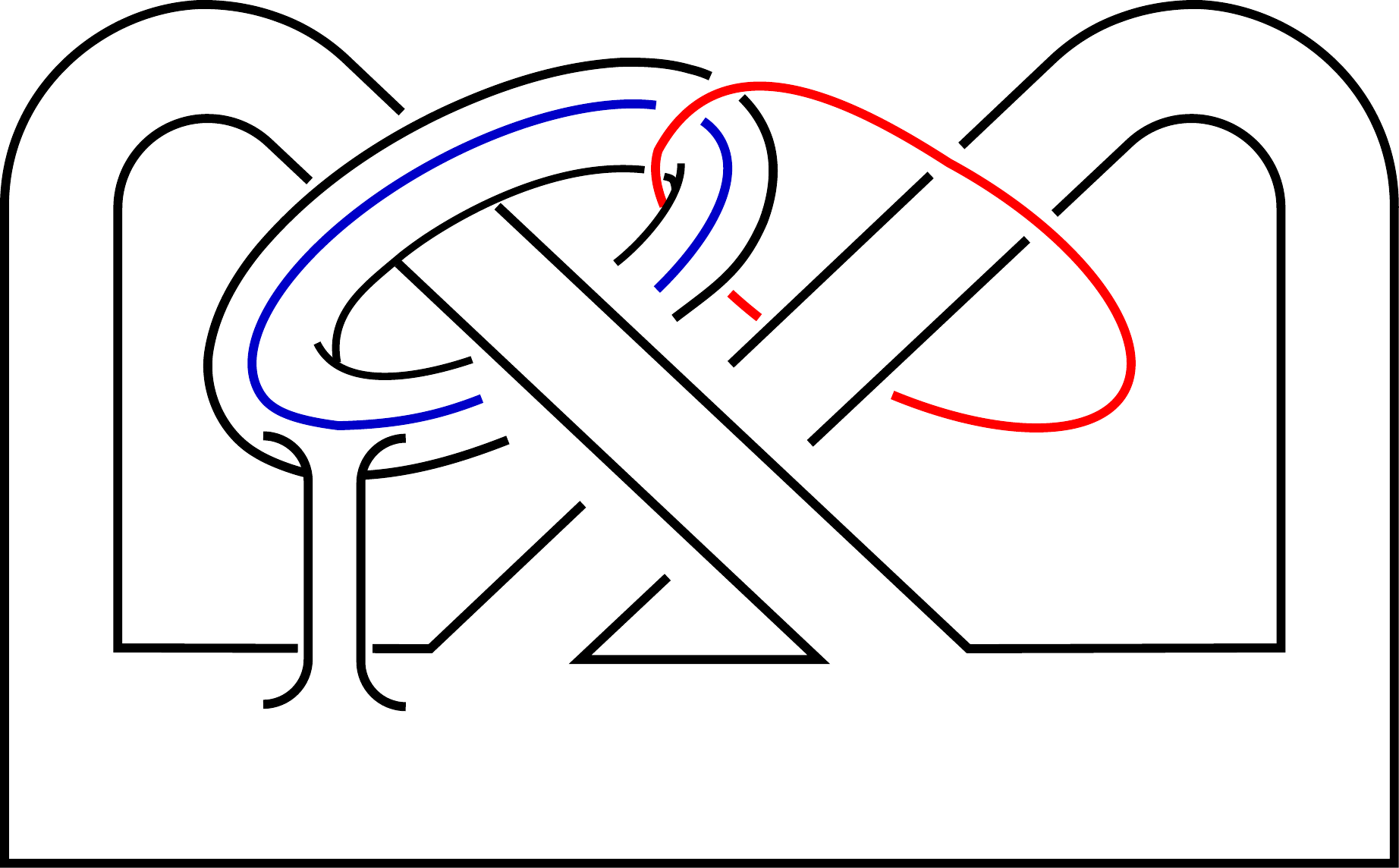}};
         %\node at (-1.3,1.5){$L_1$};
         \node at (0,-1){$G_1'$};
         %\node[blue] at (-.5,1.25){$L_2$};
         %\node[red] at (.5,1.25){$L_3$};
         \end{tikzpicture}
         \caption{The same link as in Figure~\ref{fig:example1} with a different Seifert surface.}\label{fig:example1DiffSurf}
     \end{subfigure}
     \hfill
     \begin{subfigure}[t]{0.3\textwidth}
     \centering
         \begin{tikzpicture}
         \node at (0,0){\includegraphics[width=.85\textwidth]{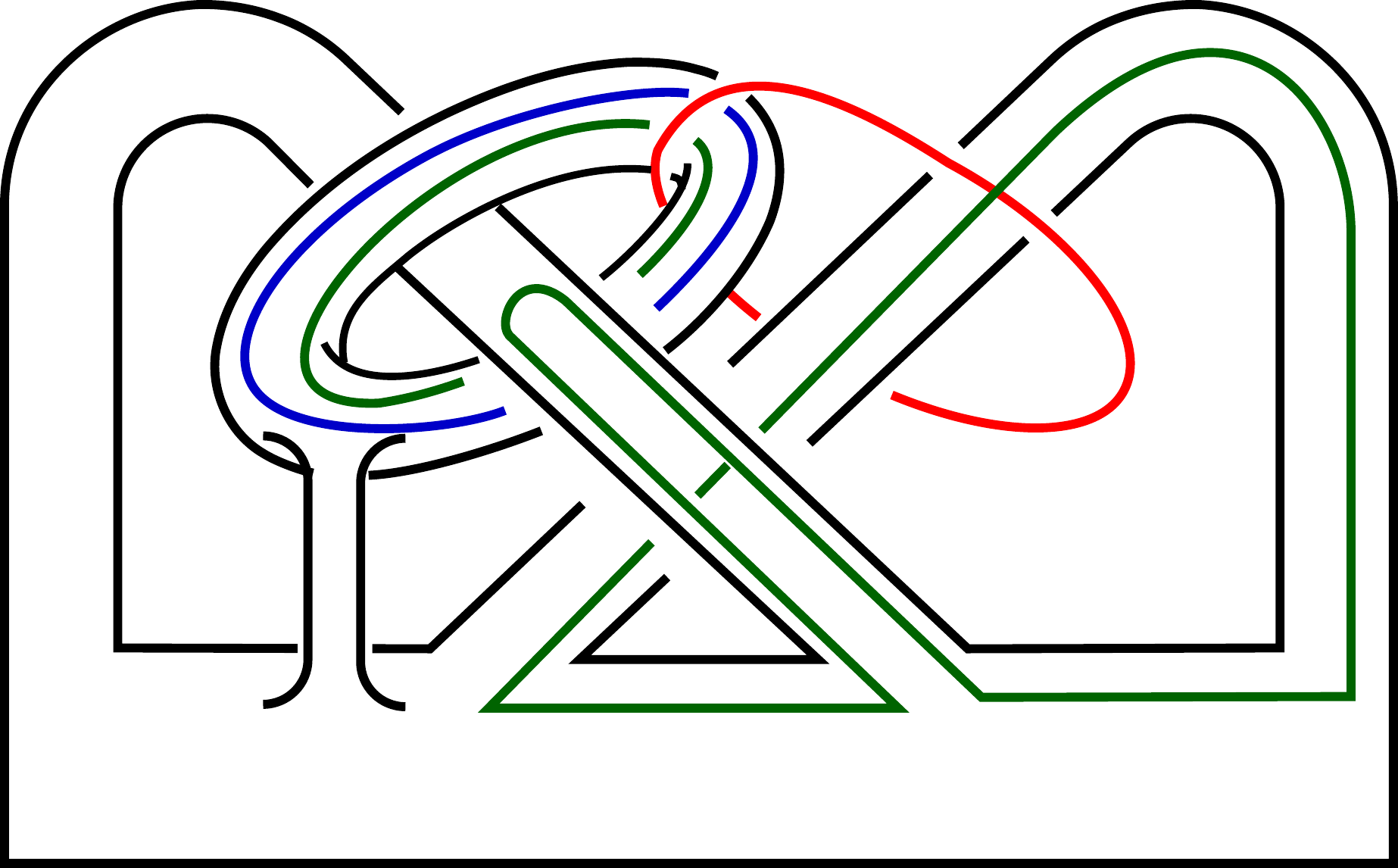}};
         \end{tikzpicture}
         \caption{Using the same surface as in Figure~\ref{fig:example1SeifSurf} results in $G_1'\cap G_2$ being disconnected.}\label{fig:example1DiffSurfAlmostDeriv}
         \end{subfigure}
         \hfill
     \begin{subfigure}[t]{0.3\textwidth}
     \centering
         \begin{tikzpicture}
         \node at (0,0){\includegraphics[width=.85\textwidth]{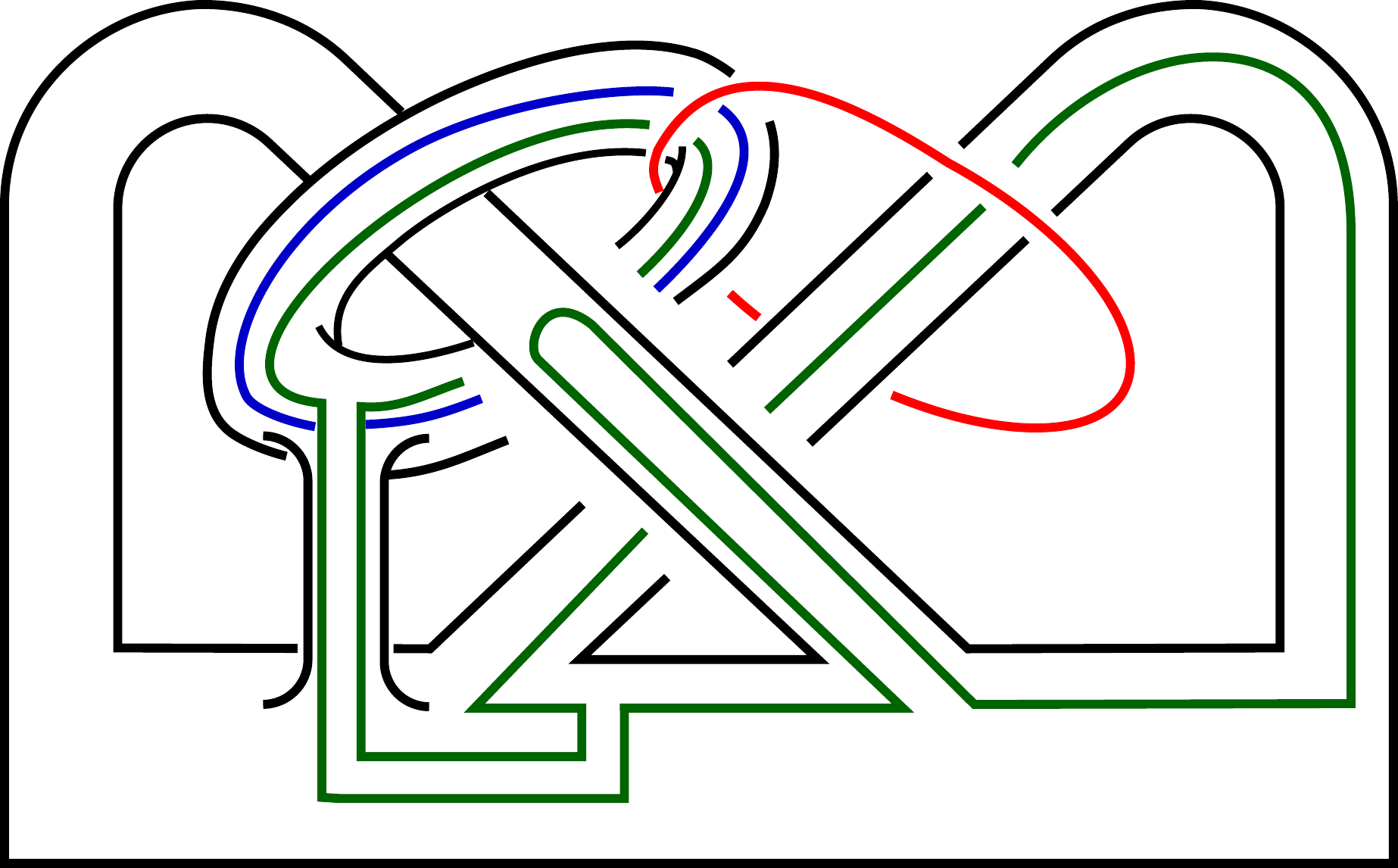}};
         \node at (1.3,-1.1) {$L_{12}$};
         \end{tikzpicture}
         \caption{The resulting derivative $L_{12}$ is unknotted and unlinked from $G_1'\cup L_3$.}\label{fig:example1DiffSurfDeriv}
         %\label{fig:claspersurgery}
     \end{subfigure}
     \caption{A different choice of Seifert surface and the resulting derivative.}\label{fig: Example1DiffSurfDeriv}
     \end{figure}

%We will resolve the ill-definedness issue in the next section.  For now, let us illustrate nontriviality. 

\begin{figure}[!htbp]
     \centering
         \begin{subfigure}{.45\textwidth}
\begin{tikzpicture}
\node[] at (0.2, 1.2){$L_1$};
\node[red] at (-2.45, 0){$L_2$};
\node[blue] at (2.5, 1.4){$L_3$};
         \node at (0,0){\includegraphics[angle=0,width=.9\textwidth]{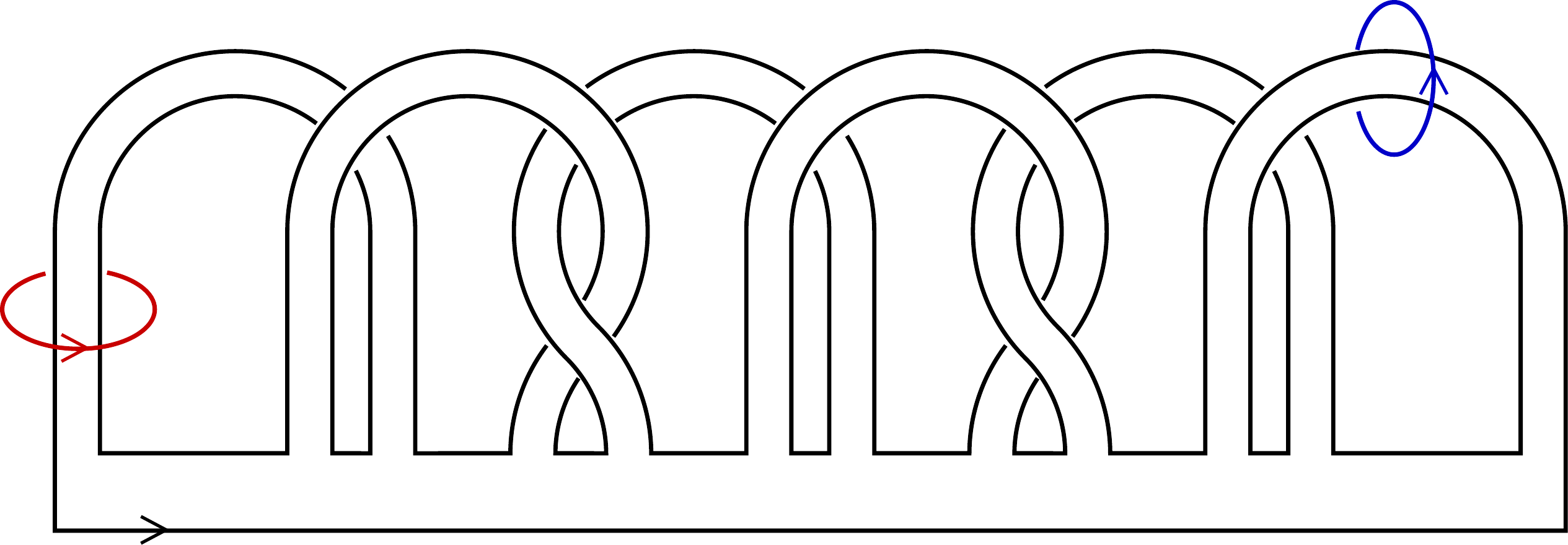}};
         \end{tikzpicture}
        \caption{}\label{nonzero gamma 3 1}
         \end{subfigure}
         %%%%%%%%%%%%%%
         \begin{subfigure}{.45\textwidth}
\begin{tikzpicture}
\node[] at (0.2, 1.2){$L_1$};
\node[red] at (-2.4, .45){$G_2$};
\node[blue] at (2.5, 1.4){$L_3$};
         \node at (0,0){\includegraphics[angle=0,width=.9\textwidth]{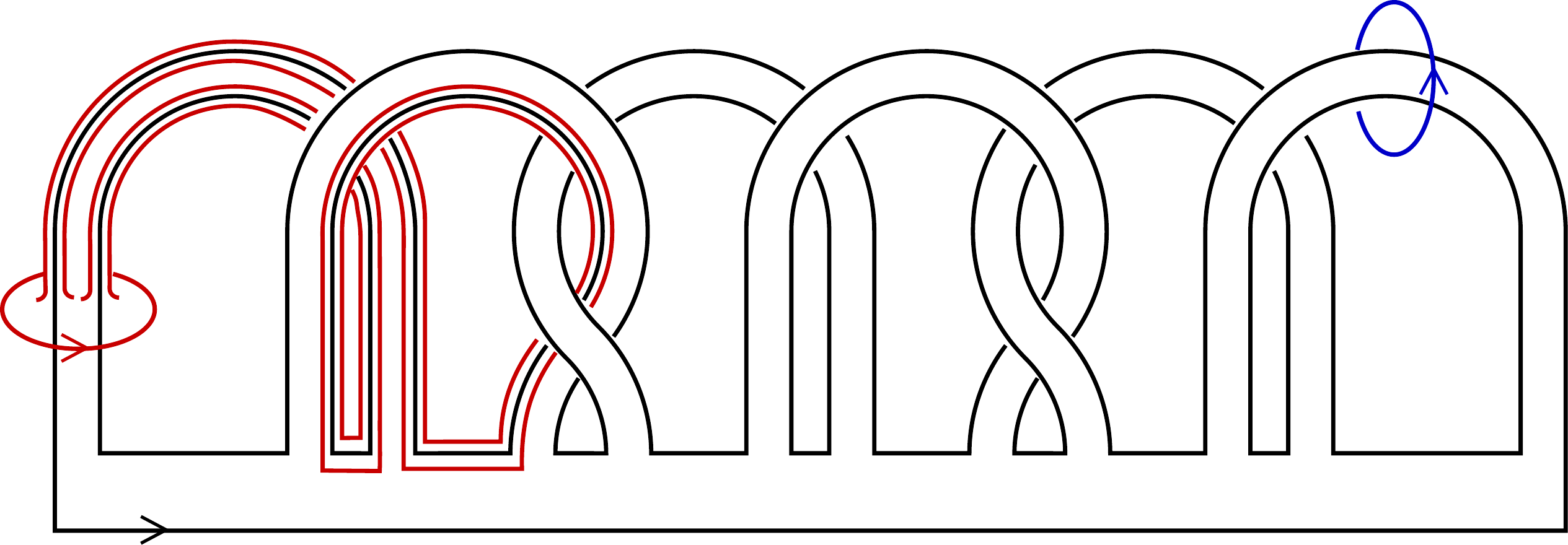}};
         \end{tikzpicture}
        \caption{}
         \end{subfigure}
         %%%%%%%%%%%%
         \begin{subfigure}{.45\textwidth}
\begin{tikzpicture}
\node[] at (0.2, 1.2){$L_1$};
\node[red] at (-2.55, .45){$L_{12}$};
\node[blue] at (2.5, 1.4){$L_3$};
         \node at (0,0){\includegraphics[angle=0,width=.9\textwidth]{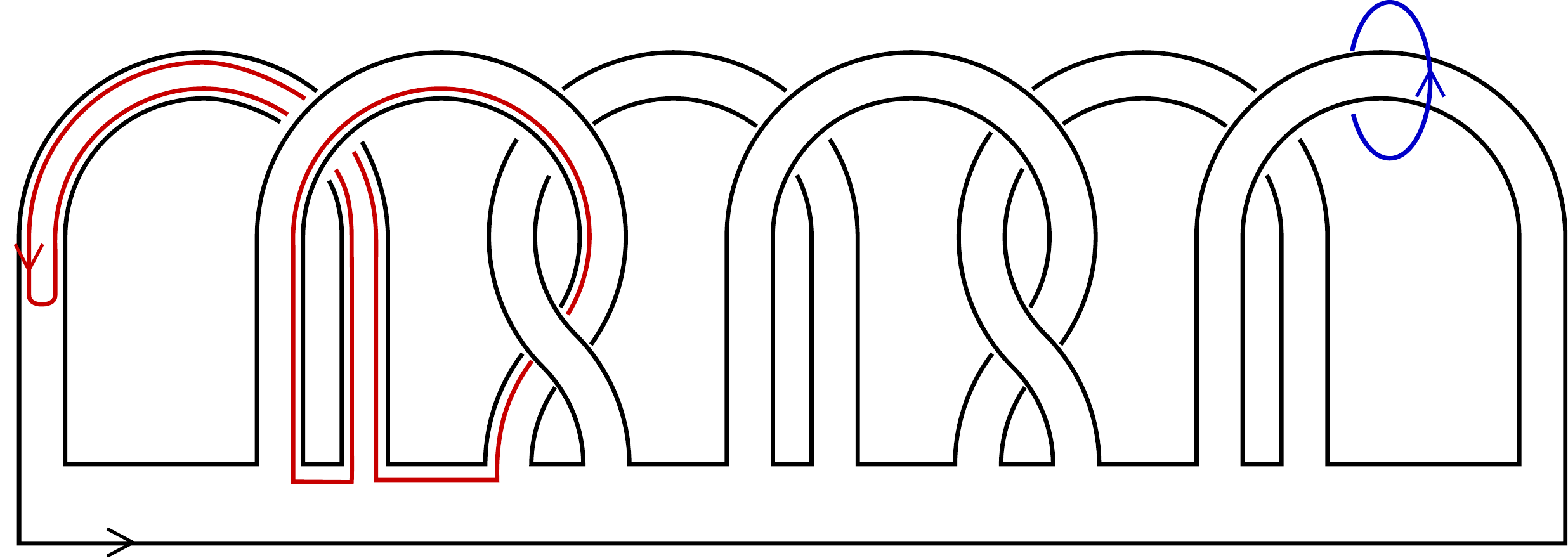}};
         \end{tikzpicture}
        \caption{}\label{nonzero gamma 3 2}
         \end{subfigure}
         %%%%%%%%%%%%
         \begin{subfigure}{.45\textwidth}
\begin{tikzpicture}
\node[] at (0.2, 1.2){$L_1$};
\node[red] at (-1.45, -.2){\small{$L_{12}$}};
\node[blue] at (2.5, 1.4){$L_3$};
         \node at (0,0){\includegraphics[angle=0,width=.9\textwidth]{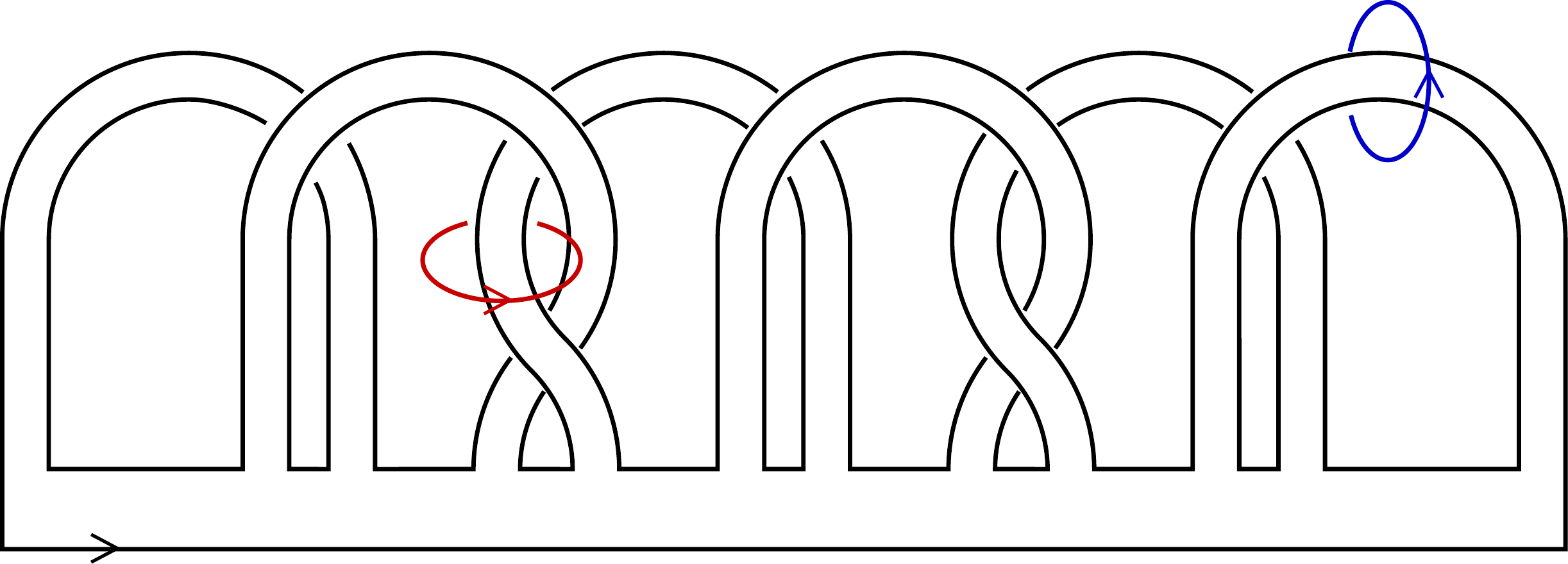}};
         \end{tikzpicture}
        \caption{}\label{nonzero gamma 3 3}
         \end{subfigure}
         %%%%%%%%%%%%
         \begin{subfigure}{.45\textwidth}
\begin{tikzpicture}
\node[] at (0.2, 1.2){$L_1$};
\node[red] at (.6, -.3){\small{$L_{112}$}};
\node[blue] at (2.5, 1.4){$L_3$};
         \node at (0,0){\includegraphics[angle=0,width=.9\textwidth]{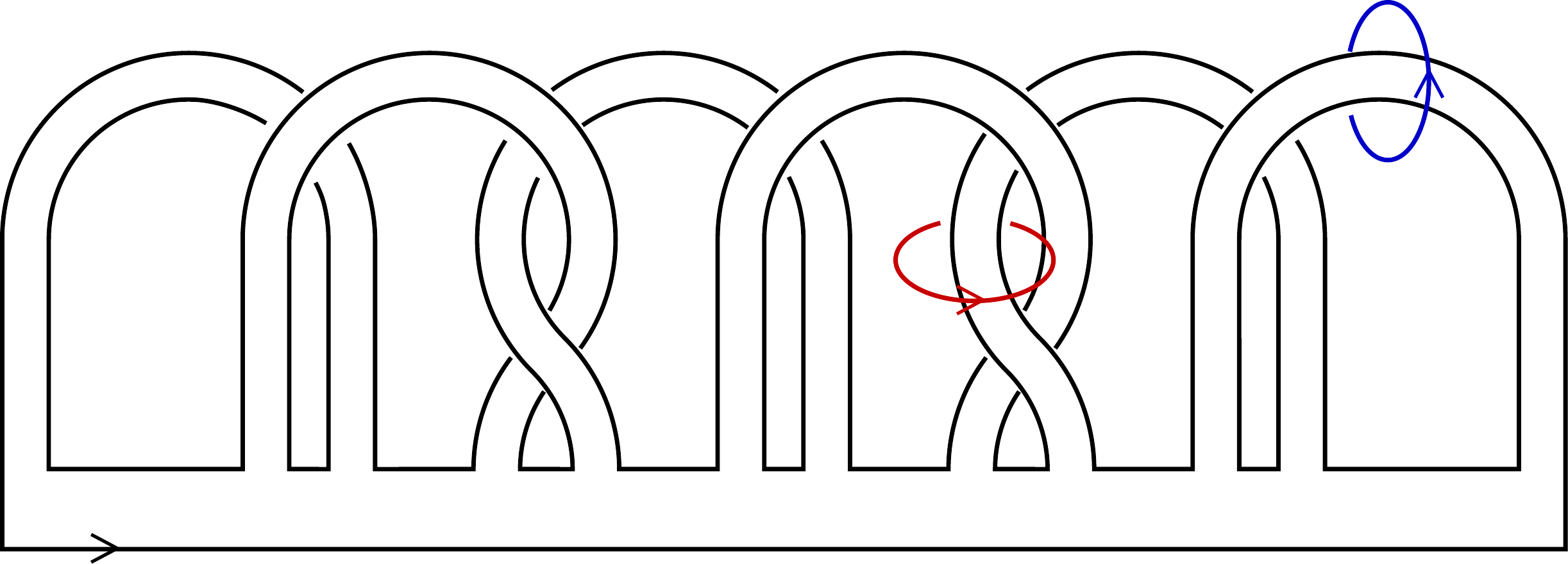}};
         \end{tikzpicture}
        \caption{}\label{nonzero gamma 3 4}
         \end{subfigure}
         %%%%%%%%%%%%
         \begin{subfigure}{.45\textwidth}
\begin{tikzpicture}
\node[] at (0.2, 1.2){$L_1$};
\node[red] at (2.8, .1){\small{$L_{1112}$}};
\node[blue] at (2.5, 1.4){$L_3$};
         \node at (0,0){\includegraphics[angle=0,width=.9\textwidth]{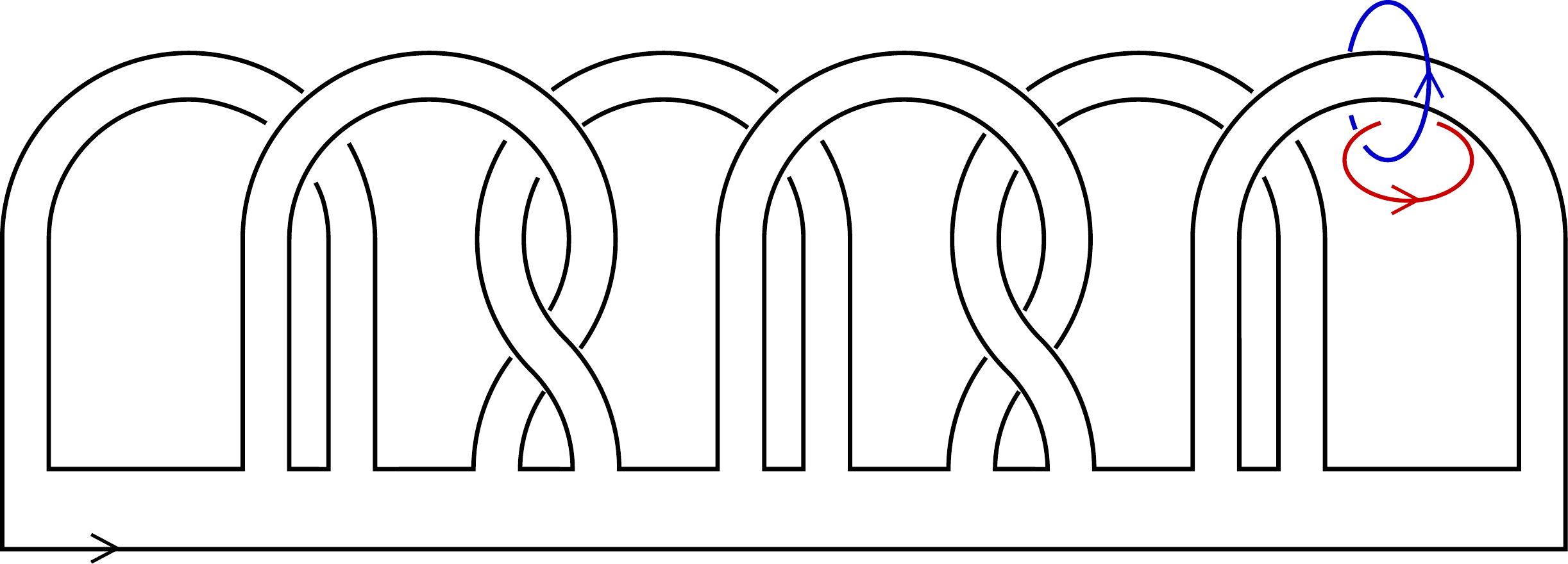}};
         \end{tikzpicture}
        \caption{}\label{nonzero gamma 3 5}
         \end{subfigure}
        \caption{A $3$-component link and its first three derivatives. 
(A) The link $(L_1,L_2,L_3)$, where $L_1$ bounds the evident Seifert surface $G_1$. 
(B) A Seifert surface $G_2$ for $L_2$. 
(C) The derivative $L_{12}$. 
(D) An isotopy of $L_{12}$ disjoint from $G_1$. 
(E) The derivative $L_{112}$. 
(F) The derivative $L_{1112}$.}
        \label{fig: nonzero gamma 3}
\end{figure}
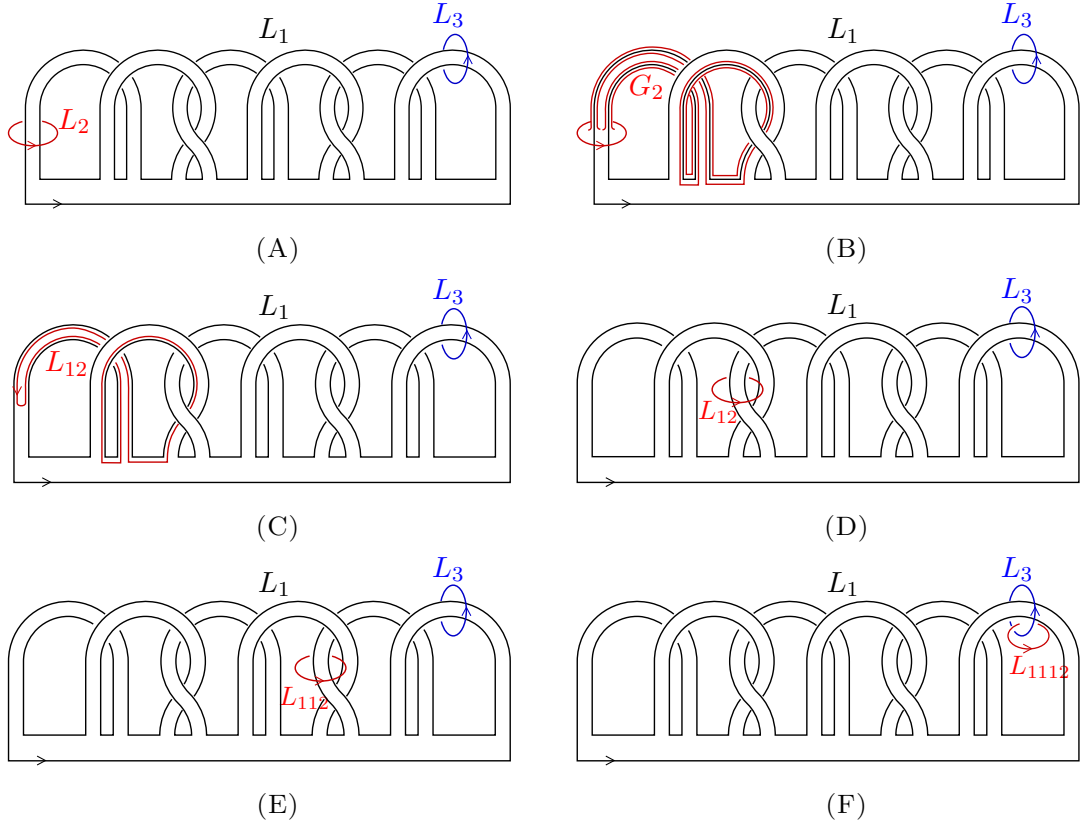

In Figure~\ref{nonzero gamma 3 1} we see a $3$-component link $L=(L_1, L_2, L_3)$ together with its first three derivatives. This shows that $\gamma^3(L)=1$ and that $\gamma^k(L)=0$ for $k\neq 3$. This example readily generalizes to produce, for any $p$, links satisfying $\gamma^p(L)=1$ and $\gamma^k(L)=0$ for all $k\neq p$. 

More generally, given any finite sequence $a_0,a_1,a_2,\dots,a_n$ of integers, we can produce a link with $\gamma^k(L)=a_k$ for $k\le n$ and $\gamma^k(L)=0$ for $k>n$. We provide an explicit example.  Start with the link in Figure~\ref{fig: arbitrary gamma}, let $G_1$ be the obvious Seifert surface for $L_1$ and construct $L_3$ by banding together $a_0$ copies of $L_3^0$, $a_1$ copies of $L_3^1$, and so on, using bands disjoint from $G_1$. 

In~\cite{Jin91}, Jin classifies which sequences arise as Cochran's $\beta$-invariants; his work is based on a connection to Kojima--Yamasaki's $\eta$-invariant~\cite{KojYam79}. While we do not pursue this here, we expect that one can similarly characterize exactly which sequences can arise as $\gamma$-invariants of links. See Problem~\ref{problem: classify gamma sequences}.

\begin{figure}[!htbp]
     \centering
         \begin{tikzpicture}
         \node at (0,0){\includegraphics[width=.8\textwidth]{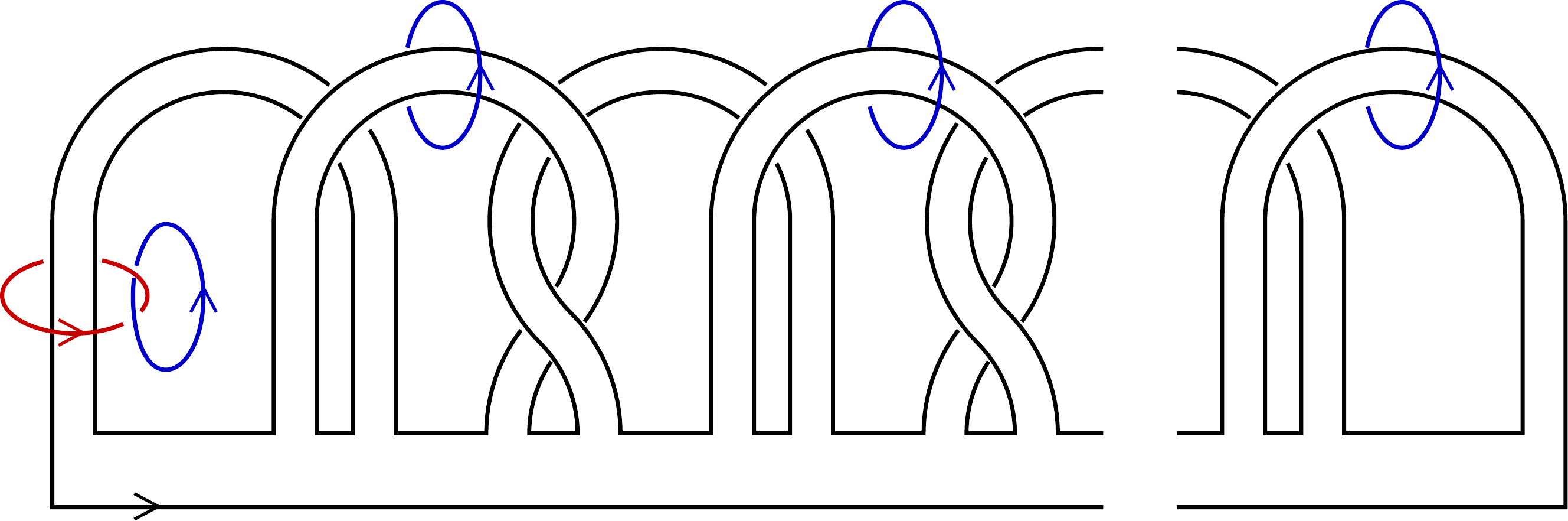}};
         %\node at (1,2){$\dots$};
         %\node at (1,-2){$\dots$};
         \node at (-6.5,-.75){$L_2$};
          \node at (-5,.5){$L_3^0$};
         \node at (0,-1.75){$L_1$};
         \node at (-2.9,.7){$L_3^1$};
         \node at (.9,.7){$L_3^2$};
         \node at (5.2,.7){$L_3^{n}$};
         \node at (3.05,1.6){$\dots$};
         \node at (3.05,-1.75){$\dots$};
         \end{tikzpicture}
        \caption{Constructing $L_3$ by banding together $a_0$ copies of $L_3^0$, $a_1$ copies of $L_3^1$, and so on, produces a link with $\gamma^k(L)=a_k$ for $k\le n$ and $\gamma^k(L)=0$ for $k>n$.}
        \label{fig: arbitrary gamma}
        \end{figure}
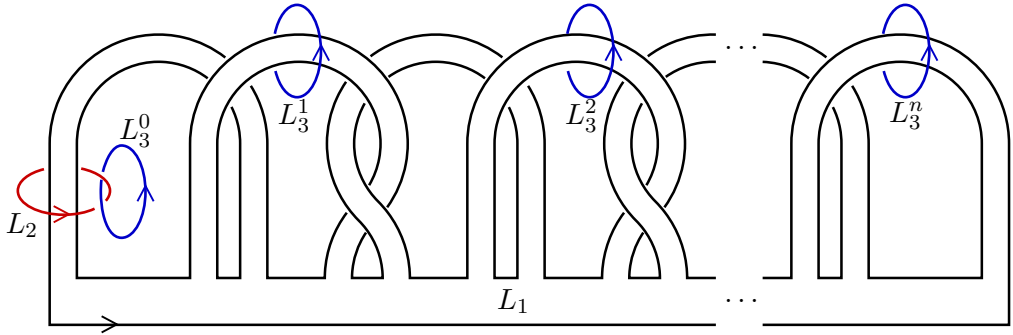

\subsection{Computation from a Seifert matrix}

In~\cite[Section~8]{C3}, Cochran explains how to compute his iterated derivatives and the resulting $\beta$-invariants in terms of a Seifert matrix. We now do the same for our $\gamma$-invariants. Indeed, let $G_1$ and $G_2$ be Seifert surfaces for $L_1$ and $L_2$, respectively, that intersect in a simple closed curve $L_{12}^0$. (We do not yet denote this curve by $L_{12}$, since we have yet to push it off from $G_1$.) Cutting $G_2$ open along $G_1$ yields a cobordism from $i^+(L_{12}^0)-i^-(L_{12}^0)$ to $L_2$. Equivalently,
\[
(i^+_*-i^-_*)[L_{12}^0]=[L_2]\in H_1(S^3\smallsetminus G_1).
\]
Here $i^+,i^-\colon G_1\to S^3\smallsetminus G_1$ are the maps obtained by pushing $G_1$ off itself in the positive and negative normal directions, respectively. Since $i^+_*-i^-_*\colon H_1(G_1)\to H_1(S^3\smallsetminus G_1)$ is an isomorphism, the class $[L_{12}^0]\in H_1(G_1)$ is determined by $[L_2]$ and, in particular, is independent of the choice of $G_2$. It follows that
\[
\gamma^1(L,G_1)=\lk(L_{12},L_3)=\lk(L_{12}^0,L_3)
\]
depends only on $[L_2]$ and $[L_3]$ in $H_1(S^3\smallsetminus G_1)$. Since $[L_{12}]=i^+_*([L_{12}^0])$, a direct induction shows that for all $k>1$, the quantity $\gamma^k(L,G_1)$ depends only on the classes $[L_2]$ and $[L_3]$ in $H_1(S^3\smallsetminus G_1)$. {Expressing all of these relations in homology with respect to a chosen basis gives the following. A more general result appears in \cite[Theorem~4.1]{Tsukamoto-Yasuhara:2007}.  We include our proof here as an illustration of the computability of this invariant.}

%The proof of the following theorem amounts to unpacking these maps with respect to a choice of basis. This yields Theorem~\ref{thm:seifert-matrix}, which we now recall.

\begin{proposition}[{\cite[Theorem~4.1]{Tsukamoto-Yasuhara:2007}}]\label{thmbody:seifert-matrix}
Let $L=(L_1,L_2,L_3)$ be a $3$-component link whose distinguished component bounds a Seifert surface $G$. Let $\{a_1,\ldots,a_{2g}\}$ be a basis for $H_1(G)$ and let
$\{\alpha_1,\ldots,\alpha_{2g}\}$ be the dual basis for $H_1(S^3\smallsetminus G)$. Let $V$ be the
resulting Seifert matrix for $G$, and let $v_2,v_3\in \Z^{2g}$ be the column vectors representing
$[L_2]$ and $[L_3]$ in $H_1(S^3\smallsetminus G)$ with respect to $\{\alpha_i\}$. Set $A:=V-V^T$.
Then, for every positive integer $k$,
\[
\gamma^k(L,G)
=
\Bigl(A^{-1}(VA^{-1})^{k-1}v_2\Bigr)^T v_3.
\]
\end{proposition}

As an example, we compute all of the $\gamma$-invariants of the link in Figure~\ref{fig: gamma 3 example}. In the notation of Proposition~\ref{thmbody:seifert-matrix},
\[
v_2=\begin{pmatrix}1\\0\end{pmatrix},\qquad
v_3=\begin{pmatrix}0\\1\end{pmatrix},\qquad
V=\begin{pmatrix}0&2\\1&0\end{pmatrix}.
\]
Thus,
\[
A=\begin{pmatrix}0&1\\-1&0\end{pmatrix},\qquad
A^{-1}=\begin{pmatrix}0&-1\\1&0\end{pmatrix},\qquad
VA^{-1}=\begin{pmatrix}2&0\\0&-1\end{pmatrix}.
\]
Finally, by Proposition~\ref{thmbody:seifert-matrix},
\[
\gamma^k(L,G_1)
=
\left(
\begin{pmatrix}0&-1\\1&0\end{pmatrix}
\begin{pmatrix}2&0\\0&-1\end{pmatrix}^{k-1}
\begin{pmatrix}1\\0\end{pmatrix}
\right)^T
\begin{pmatrix}0\\1\end{pmatrix}
=
2^{k-1}.
\]
Hence $\gamma(L,G_1)=(1,1,2,4,8,\dots)$.

\begin{figure}
     \centering
         \begin{tikzpicture}
         \node at (-.3,1) {$L_1$};
         \node at (.6,.5) {$a_1$};
         \node at (2.6,.55) {$a_2$};
         \node[blue] at (.8,2.88) {$L_2$};
         \node[red] at (2.8,2.88) {$L_3$};
         \node[above right] at (0,0){\includegraphics[width=.2\textwidth]{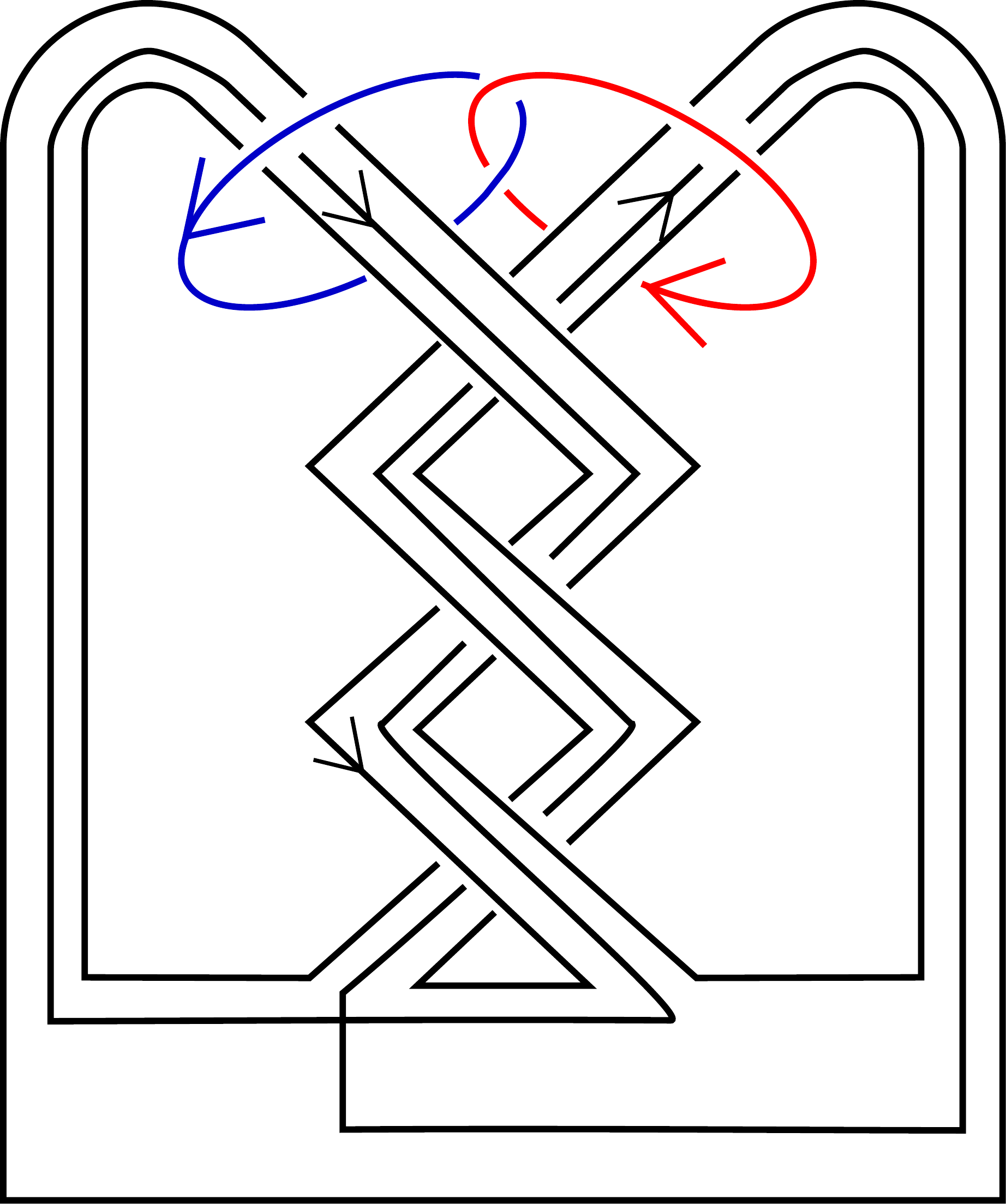}};
        
         \end{tikzpicture}
        \caption{A link $L$ whose distinguished component $L_1$ bounds the Seifert surface $G$, together with a symplectic basis $\{a_1,a_2\}$ for $H_1(G)$.}
        \label{fig: gamma 3 example}
\end{figure}
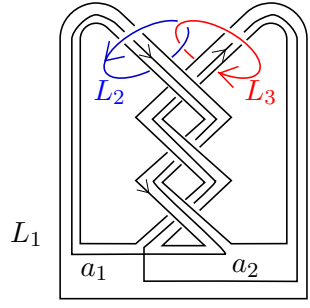

\begin{proof}[Proof of Proposition~\ref{thmbody:seifert-matrix}]
Let $\{a_1,\dots,a_{2g}\}$ be a basis for $H_1(G)$ and let $\{\alpha_1,\dots,\alpha_{2g}\}$ be the linking dual basis for $H_1(S^3\smallsetminus G)$. With respect to these bases, the map
\[
i^+_*:H_1(G)\to H_1(S^3\smallsetminus G)
\]
is represented by the Seifert matrix $V$, and $i^-_*$ is represented by $V^T$.

Thus,
\[
[L_{12}^0]=(i^+_*-i^-_*)^{-1}[L_2],
\]
or, in coordinates with respect to $\{a_i\}$ and $\{\alpha_i\}$,
\[
[L_{12}^0]=(V-V^T)^{-1}v_2.
\]
Next,
\[
[L_{12}]=i^+_*[L_{12}^0]=V(V-V^T)^{-1}v_2.
\]
Iterating, we obtain
\[
[L_{1^k2}]=\left(V(V-V^T)^{-1}\right)^k v_2
\qquad\text{and}\qquad
[L_{1^k2}^0]=(V-V^T)^{-1}\left(V(V-V^T)^{-1}\right)^{k-1}v_2.
\]
Here, as for $L_{12}^0$, the curve $L_{1^k2}^0$ denotes the intersection of $G$ with a Seifert surface for $L_{1^{k-1}2}$.

Now let $x\in H_1(G)$ and write $x=\sum_{i=1}^{2g} x_i a_i$, so that its coordinate vector is $\vec x\in \Z^{2g}$. Write $[L_3]=\sum_{i=1}^{2g} (v_3)_i\,\alpha_i$. By the definition of the linking dual basis,
\[
\lk(x,L_3)=\sum_{i=1}^{2g} x_i(v_3)_i = \vec x^{\,T}v_3.
\]
Putting these observations together, we have
\[
\gamma^k(L,G)=\lk(L_{1^k2}^0,L_3)
=
\left((V-V^T)^{-1}\left(V(V-V^T)^{-1}\right)^{k-1}v_2\right)^T v_3,
\]
as claimed.
\end{proof}

\section{Swapping components and recovering Cochran's $\beta$-invariant}

The reader will have noticed that the operator $D(L_1, L_2, L_3)=(L_1, L_{12}, L_3)$ is only one of the two possible choices of derivatives we might have considered (for the moment, we suppress the choice of Seifert surface). One could equally well define a derivative in the third component, producing $(L_1, L_2, L_{13})$. This would lead to mixed iterated derivatives
\[
(L_1, L_2, L_3)\longmapsto (L_1, L_{1^{p}2}, L_{1^{l}3}),
\]
which one might hope would carry more information. As we will prove in this section, the $\gamma$-invariants coming from these mixed derivatives carry no new information.    As a consequence, we explain how to recover Cochran's $\beta$-invariants in terms of $\gamma$-invariants.  

Let $\swap$ denote the operator on $3$-component links that swaps the second and third components:
\[
\swap(L_1, L_2, L_3)=(L_1, L_3, L_2).
\]
The mixed derivative $(L_1, L_{1^{p}2}, L_{1^{l}3})$ from the preceding paragraph can now be described as
\[
\swap D^{l}  \swap  D^{p}(L_1, L_2, L_3).
\]
The following lemma shows how to compute all of these mixed $\gamma$-invariants in terms of the $\gamma$-invariants we have already defined. For the sake of precision, we now include the choice of Seifert surface.

\begin{lemma}\label{lem:swap}
Let $L=(L_1, L_2, L_3)$ be a link with a distinguished component. Let $G_1$ be a Seifert surface for $L_1$. Then: %\footnote{\chris{A thought: would this be easier using the philosophy of Theorem~\ref{thmbody:seifert-matrix} This might be worth checking.  }}
\begin{enumerate}
\item \label{item: DSDS vs SDSD} $\displaystyle
\gamma^0 \bigl( \Swap  D \Swap  D(L,G_1)\bigr)
=
\gamma^0  \bigl(D \Swap  D \Swap(L,G_1)\bigr)+\gamma^1(L,G_1)$.

\item \label{item: deriv commute} For any $p \geq 1$, $D^p D \Swap D \Swap(L,G_1)$ and $D \Swap D \Swap D^p(L,G_1)$ are weakly cobordant.

\item \label{item: gamma0 and S} $\displaystyle
\gamma^0 \bigl(\Swap(L,G_1)\bigr)=\gamma^0\bigl(L,G_1\bigr)$.

\item \label{item: gamma^k Swap} For any $k\ge 1$, $$
\gamma^k\bigl(\Swap(L,G_1)\bigr)
=
(-1)^k\sum_{j=1}^k {k-1 \choose j-1}\gamma^j\bigl(L,G_1\bigr).$$

\item \label{item: Mixed gammas} For any $p\ge 0$ and $l\ge 1$, $$
\gamma^0\bigl(\Swap D^l \Swap D^p(L,G_1)\bigr)
=
(-1)^l\sum_{j=1}^l {l-1 \choose j-1}\gamma^{p+j}\bigl(L,G_1\bigr).$$
\end{enumerate}
\end{lemma}

%\textcolor{red}{A similar result appears in \cite[Corollary~4.2]{Tsukamoto-Yasuhara:2007}, with a difference being that the derivatives defined in that paper are not all pushed off of a Seifert surface for $L_1$ in the positive normal direction.}

\begin{proof}
Let $G_1$, $G_2$, and $G_3$ be Seifert surfaces for $L_1$, $L_2$, and $L_3$.
%, chosen so that $G_1\cap G_2$ and $G_1\cap G_3$ are each simple closed curves.
Then $\Swap D\Swap D(L,G_1)$ has first component $L_1$, second component obtained by pushing $G_1\cap G_2$ off of $G_1$, and third component obtained by pushing $G_1\cap G_3$ off of $G_1$. As usual, throughout this process we ensure that the components of the link remain disjoint.

Motivated by this description, we denote this link by $(L_1, L_{12}^{+}, L_{13})$. Similarly, $D\Swap D\Swap(L,G_1)=(L_1, L_{12}^{+}, L_{13}^{++})$ is obtained by performing the same construction in the opposite order. Thus, $L_{13}^{++}\cup-L_{13}$ bounds an annulus that intersects $L_{12}$ once for each intersection point of $G_1\cap G_2$ with $G_1\cap G_3$ in $G_1$. This count agrees with the algebraic intersection number of $G_1\cap G_2$ with $G_3$, which in turn recovers $\lk(L_{12},L_3)=\gamma^1(L,G_1)$. Claim~\pref{item: DSDS vs SDSD} now follows by checking that the orientations in these intersection counts agree.
%\chris{This feels like a really tedious exercise.  Maybe we need to review the actual singular chain complex definition of the intersection form.}

Since the annulus cobounded by $L_{13}$ and $L_{13}^{++}$ lies in the complement of $G_1$, it follows that, as long as we use the same Seifert surface for $L_{12}^+$ in each instance, the links
\[
D(L_1, L_{12}^+, L_{13}, G_1)
\qquad\text{ and }\qquad
D(L_1, L_{12}^+, L_{13}^{++}, G_1)
\]
are isotopic. Since the choice of Seifert surface for the second component does not change the weak cobordism class by Lemma~\ref{lem:independentofG2}, we conclude that $D(L_1, L_{12}^+, L_{13}, G_1)$ and $D(L_1, L_{12}^+, L_{13}^{++}, G_1)$ are weakly cobordant. This proves Claim~\pref{item: deriv commute} when $p=1$. An easy induction completes the proof for any $p>1$.

Claim~\pref{item: gamma0 and S} follows from the symmetry of linking number:
\[
\gamma^0(\Swap(L,G_1))=\lk(L_3,L_2)=\lk(L_2,L_3)=\gamma^0(L,G_1).
\]

Claim~\pref{item: gamma^k Swap} is proved by induction. Recall that $G_2$ and $G_3$ are Seifert surfaces that intersect $G_1$ in simple closed curves (these curves may intersect each other). Let
\[
L_{12}=(G_1\cap G_2)^+
\qquad\text{and}\qquad
L_{13}=(G_1\cap G_3)^+
\]
be the resulting derivatives. Since linking number is given by the signed count of intersections between one component and a Seifert surface for the other, $\gamma^1(L,G_1)$ is the signed count of intersections between $G_3$ and $L_{12}$. An intersection point is positive if the tangent vector to $L_{12}$ has positive dot product with the positive normal vector to $G_3$ (and negative otherwise). Equivalently, the sign is determined by the sign of
\[
\bigl(n_{G_1}\times n_{G_2}\bigr)\cdot n_{G_3}.
\]
Similarly, $\gamma^1(\Swap(L,G_1))=\lk(L_2,L_{13})$ is given by the same intersection count, but with each sign reversed. Thus,
\[
\gamma^1(\Swap(L,G_1))=-\gamma^1(L,G_1),
\]
as predicted by Claim~\pref{item: gamma^k Swap} when $k=1$.

Now assume $k>1$ and argue by induction. Using $\Swap^2=\Id$ and the inductive hypothesis, we have
\begin{align*}
\gamma^k(\Swap(L,G_1))
&=\gamma^{k-1}\bigl(D\Swap(L,G_1)\bigr)
=\gamma^{k-1}\bigl(\Swap\Swap D\Swap(L,G_1)\bigr) \\
&=(-1)^{k-1}\sum_{j=1}^{k-1} {k-2 \choose j-1}\,
\gamma^j\bigl(\Swap D\Swap(L,G_1)\bigr).
\end{align*}
Next, Claim~\pref{item: deriv commute} implies
\[
\gamma^j(\Swap D\Swap(L,G_1))
=\gamma^0\bigl(D^{j-1}D\Swap D\Swap(L,G_1)\bigr)
=\gamma^0\bigl(D\Swap D\Swap D^{j-1}(L,G_1)\bigr).
\]
Applying Claim~\pref{item: DSDS vs SDSD} now gives
\begin{align*}
\gamma^k(\Swap(L,G_1))
&=(-1)^{k-1}\sum_{j=1}^{k-1} {k-2 \choose j-1}\,
\gamma^0\bigl(D\Swap D\Swap D^{j-1}(L,G_1)\bigr) \\
&=(-1)^{k-1}\sum_{j=1}^{k-1} {k-2 \choose j-1}\,
\left(\gamma^0\bigl(\Swap D\Swap D^{j}(L,G_1)\bigr)-\gamma^1\bigl(D^{j-1}(L,G_1)\bigr)\right).
\end{align*}
Using $\gamma^1(D^{j-1}(L,G_1))=\gamma^j(L,G_1)$ and
\[
\gamma^0\bigl(\Swap D\Swap D^j(L,G_1)\bigr)
=\gamma^0\bigl(D\Swap D^j(L,G_1)\bigr)
=\gamma^1\bigl(\Swap D^j(L,G_1)\bigr)
=-\gamma^1\bigl(D^j(L,G_1)\bigr)
=-\gamma^{j+1}(L,G_1),
\]
we obtain
\[
\gamma^k(\Swap(L,G_1))
=
(-1)^k\sum_{j=1}^{k-1} {k-2 \choose j-1}\,
\left(\gamma^{j+1}(L,G_1)+\gamma^j(L,G_1)\right).
\]
Reindexing the sum yields
\begin{align*}
\gamma^k(\Swap(L,G_1))
&=
(-1)^k\sum_{j=1}^{k}
\left({k-2 \choose j-2}+{k-2 \choose j-1}\right)\gamma^j(L,G_1) \\
&=
(-1)^k\sum_{j=1}^{k} {k-1 \choose j-1}\gamma^j(L,G_1).
\end{align*}

Finally, Claim~\pref{item: Mixed gammas} follows immediately from Claim~\pref{item: gamma^k Swap}:
\begin{align*}
\gamma^0\bigl(\Swap D^l\Swap D^p(L,G_1)\bigr)
&=\gamma^0\bigl(D^l\Swap D^p(L,G_1)\bigr)=\gamma^l\bigl(\Swap D^p(L,G_1)\bigr) \\
&=(-1)^{l}\sum_{j=1}^{l} {l-1 \choose j-1}\,
\gamma^j\bigl(D^p(L,G_1)\bigr) \\
&=(-1)^l\sum_{j=1}^{l} {l-1 \choose j-1}\gamma^{p+j}(L,G_1),
\end{align*}
which completes the proof.
\end{proof}

Our next goal is to compare our invariants with Cochran's $\beta$-invariants. We begin by recalling their definition, which is closely analogous to ours but applies to $2$-component links.

Let $L=(L_1,L_2)$ be a link with $\lk(L_1,L_2)=0$. Let $G_1$ and $G_2$ be Seifert surfaces for $L_1$ and $L_2$, respectively, intersecting transversely in a single simple closed curve $L_{12}$. Define
\[
\Delta(L_1,L_2)=(L_1,L_{12}).
\]
Unlike in our $3$-component setting, Cochran proves that the weak cobordism class of $(L_1,L_{12})$ is independent of the choice of Seifert surfaces and depends only on the weak cobordism class of $L$. The $k$-fold iterate is given by
\[
\Delta^k(L_1,L_2)=(L_1,L_{1^k2}).
\]
Let $L_{1^k2}^+$ be a push-off of $L_{1^k2}$ via Seifert surfaces. We then define
\[
\beta^k(L):=\lk\bigl(L_{1^k2},L_{1^k2}^+\bigr).
\]

\begin{reptheorem}{thm:beta-from-gamma}
Let $L=(L_1,L_2)$ be a $2$-component link with $\lk(L_1,L_2)=0$. Let $G$ be a Seifert surface for
$L_1$ disjoint from $L_2$, let $L_2^0$ denote the $0$-framed push-off of $L_2$, and let
$L^\ast=(L_1,L_2,L_2^0)$. Then
\[
\beta^k(L)=(-1)^k\sum_{j=1}^k \binom{k-1}{j-1}\gamma^{k+j}(L^\ast,G).
\]
\end{reptheorem}

% \begin{theorem}\label{thm:Compare to beta}
% Let $L=(L_1,L_2)$ be a $2$-component link with $\lk(L_1,L_2)=0$. Let $G$ be a Seifert surface for
% $L_1$ disjoint from $L_2$, let $L_2^0$ denote the $0$-framed push-off of $L_2$, and let
% $L^\ast=(L_1,L_2,L_2^0)$. Then
% \[
% \beta^k(L)=(-1)^k\sum_{j=1}^k \binom{k-1}{j-1}\gamma^{k+j}(L^\ast,G).
% \]
% \end{theorem}

\begin{proof}
Notice that
\[
(L_1, L_{1^k2}, L_{1^k2}^+, G)=\Swap D^k \Swap D^k(L^*,G).
\]
Thus, by Claim~\pref{item: Mixed gammas} of Lemma~\ref{lem:swap},
\begin{align*}
\beta^k(L)
&=\gamma^0\bigl(\Swap D^k \Swap D^k(L^*,G)\bigr) \\
&=(-1)^k\sum_{j=1}^{k}{k-1\choose j-1}\gamma^{k+j}(L^*,G),
\end{align*}
as claimed.
\end{proof}

\section{A variant of the Kojima--Yamasaki $\eta$-invariant,\\ and its relation to the $\gamma$-invariants}
\label{sect eta invt}

In this section and the next, we prove our main theorem, which we now recall.
\begin{reptheorem}{thm:main}
Let $L=(L_1,L_2,L_3)$ be a $3$-component link with a distinguished component. Then
\[
\gamma(L)=\bigl(\gamma^0(L),\gamma^1(L),\ldots\bigr)\in \Z^\infty
\]
is a link concordance invariant modulo the action of $T+\mathrm{Id}$. Moreover, we have
\[
\gamma^k(L)\equiv \bar{\mu}_L(1^k23)
\pmod{\gcd\left\{\, \bar{\mu}_L(1^i23)\mid i<k \,\right\}}.
\]
\end{reptheorem}
\noindent In fact, in this section, specifically in Corollary~\ref{corollary: well defined}, we prove the stronger statement that $\gamma(L)$ depends only on the weak cobordism class of $L$. The latter part of Theorem~\ref{thm:main}, connecting $\gamma(L)$ to Milnor invariants, will be proved in the next section, specifically in Theorem~\ref{thm: gamma as Milnor}.

Cochran's argument in \cite{C3} that the $\beta$-invariants are independent of the choice of Seifert surface does not apply in our setting. Indeed, as seen in Figures~\ref{fig: Example1Deriv} and \ref{fig: Example1DiffSurfDeriv}, the $\gamma$-invariants depend on the choice of Seifert surface. In this section, we take inspiration from \cite[Theorem~7.1]{C3}, where Cochran shows that the $\beta$-invariants recover the coefficients of the $\eta$-invariant defined by Kojima and Yamasaki in \cite{KojYam79}. We define a variant of the $\eta$-invariant for $3$-component links, describe its indeterminacy, and show that the $\gamma$-invariants recover its coefficients.

We begin by defining this new invariant. Let $L=(L_1,L_2,L_3)$ be a $3$-component link with a distinguished component. Let $\widetilde{L_2}$ and $\widetilde{L_3}$  be lifts of $L_2$ and $L_3$ to the infinite cyclic cover, $\widetilde{S^3\smallsetminus L_1}$. Let $t\colon \widetilde{S^3\smallsetminus L_1}\to \widetilde{S^3\smallsetminus L_1}$ be the deck transformation corresponding to the meridian of $L_1$ and $t_*$ be the induced map on the Alexander module, $A(L_1) = H_1(\widetilde{S^3\smallsetminus L_1})$. Since $A(L_1)$ is a $\Z[t,t^{-1}]$-torsion module, there is a polynomial $\Delta(t)$ such that the $1$-cycle $\Delta(t) \widetilde{L_2}$ bounds a $2$-chain $X\in C_2(\widetilde{S^3\smallsetminus L_1})$. Kojima and Yamasaki define
$$
\eta(L_1, L_2)=\eta(L)=\dfrac{1}{\Delta(t)}\Sum_{k\in \Z} \bigl(X, t^k\widetilde{L_2^+}\bigr)t^k \in \Q(t).
$$
Here $(X, t^k\widetilde{L_2^+})$ indicates the algebraic intersection number between $X$ and the translate $t^k\widetilde{L_2^+}$ of the lift of the $0$-framed push-off of $L_2$. This can be viewed as recording the self-linking of a lift of $L_2$ in the infinite cyclic cover of $S^3\smallsetminus L_1$. We define a {rational function} by instead recording the linking between a lift of $L_2$ and a lift of $L_3$. Define
$$
\noteta(L_1,\widetilde{L_2},\widetilde{L_3})=\dfrac{1}{\Delta(t)}\Sum_{k\in \Z} \bigl(X, t^k\widetilde{L_3}\bigr)t^k\in \Q(t),
$$
where $X\in C_2(\widetilde{S^3\smallsetminus L_1})$ is a $2$-chain with $\partial X=\Delta(t) \widetilde{L_2}$, as above. {This rational function also appears in \cite{Tsukamoto-Yasuhara:2007} under the notation $\operatorname{lk}(\widetilde{L_2},\widetilde{L_3})$, where the lifts are specified by a choice of a Seifert surface for $L_1$.}

According to \cite[Theorem~2]{KojYam79}, \(\eta(L_1,L_2)\in \Q(t)\) is independent of the choice of lift of \(L_2\) and is, moreover, a concordance invariant. In contrast, \(\noteta(L_1,\widetilde{L_2},\widetilde{L_3})\) does depend on the choice of lifts, and hence we include the lifts in the notation.  We will show that changing the lifts multiplies \(\noteta(L_1,\widetilde{L_2},\widetilde{L_3})\) by a power of \(t\). The following proposition shows that, modulo this ambiguity, \(\noteta(L_1,\widetilde{L_2},\widetilde{L_3})\) defines a weak cobordism invariant. Recall that \(\doteq\) is the equivalence relation on \(\Q(t)\) generated by
\[
p(t)\doteq t p(t).
\]

\begin{definition}
Let $L=(L_1,L_2,L_3)$ be a $3$-component link with a distinguished component. Choose lifts
\(\widetilde{L_2}\) and \(\widetilde{L_3}\) of \(L_2\) and \(L_3\) to the infinite cyclic cover
\(\widetilde{S^3\smallsetminus L_1}\), and set
\[
h(L):=\big[\noteta(L_1,\widetilde{L_2},\widetilde{L_3})\big]\in \Q(t)/\doteq.
\]
\end{definition}

The following gives a stronger statement than Theorem~\ref{thm: gamma as h}, which only states that $h(L)$ is an invariant of the concordance class of $L$. The fact that the Taylor coefficients of $h(L)$ at $t=1$ recover the $\gamma$-invariants of $L$ will be proved in Proposition~\ref{thm: h is gamma}, combined with the following proposition, which should be compared with Proposition~1 and Theorem~2 in \cite{KojYam79}.

\begin{proposition}\label{prop h is invariant}
Let $L=(L_1,L_2,L_3)$ be a $3$-component link with a distinguished component.
\begin{enumerate}
\item \label{indep of X and Delta}
The quantity $\noteta(L_1,\widetilde{L_2},\widetilde{L_3})$ is independent of the choice of $X$ and $\Delta(t)$.

\item \label{choice of lifts}
For any integers $n_2$ and $n_3$,
\[
\noteta(L_1,t^{n_2}\widetilde{L_2},t^{n_3}\widetilde{L_3})
=
t^{n_2-n_3}\noteta(L_1,\widetilde{L_2},\widetilde{L_3})\in \Q(t).
\]

\item \label{item: weak cob}
If $(A,Y_2,Y_3)$ is a weak cobordism from $L$ to $L'$, then each lift $\widetilde{Y_i}$ in
$\widetilde{S^3\times[0,1]\smallsetminus A}$ has boundary
\[
\partial \widetilde{Y_i}=\widetilde{L_i}-t^{n_i}\widetilde{L_i'}
\]
for some integer $n_i$. Moreover,
\[
\noteta(L_1,\widetilde{L_2},\widetilde{L_3})
=
\noteta(L_1',t^{n_2}\widetilde{L_2'},t^{n_3}\widetilde{L_3'}) \in \Q(t).
\]
\end{enumerate}
In particular, the quantity $h(L)$ is an invariant of the weak cobordism class of $L$.
\end{proposition}

% \begin{corollary}\label{cor h is invariant}
% The quantity \(h(L)\)  is an invariant of the weak cobordism class of \(L\).
% \end{corollary}

% \begin{proposition}\label{prop h is invariant}
% Let $L=(L_1,L_2,L_3)$ be a $3$-component link and $\widetilde{L_2}$ and $\widetilde{L_3}$  be lifts of $L_2$ and $L_3$ to the infinite cyclic cover, $\widetilde{S^3\smallsetminus L_1}$. The class of $\noteta(L_1,\widetilde{L_2},\widetilde{L_3})$ in $\sfrac{\Q(t)}{\doteq}$ depends only on the weak cobordism class of $L$. Furthermore, $\noteta(L_1, t^{i}\widetilde{L_2}, t^{j}\widetilde{L_3}) = t^{i-j}\noteta(L_1, \widetilde{L_2}, \widetilde{L_3})$.
% \end{proposition}

% \JP{added the definition and changed the statement of prop 4.2.}

\begin{proof}
Let $L=(L_1,L_2,L_3)$ be a $3$-component link, and let $\widetilde{L_2}$, $\widetilde{L_3}$, $\Delta(t)$, and $X$ be as in the definition of $\noteta$. We first prove \pref{choice of lifts}.  We have $\bdry(t^{n_2}X)=\Delta(t)\,t^{n_2}\widetilde{L_2}$. Also, we have that
\[
\bigl(t^{n_2} X, t^k(t^{n_3}\widetilde{L_3})\bigr)
=
\bigl(X, t^{k+n_3-n_2}\widetilde{L_3}\bigr).
\]
By reindexing the sum, we obtain
\begin{equation*}\label{eqn: same noteta}
\begin{array}{rll}
h(L_1, t^{n_2}\widetilde{L_2},t^{n_3}\widetilde{L_3})&=\frac{1}{\Delta(t)}\sum_{k\in \mathbb{Z}} \bigl(t^{n_2}X, t^k(t^{n_3}\widetilde{L_3})\bigr)t^k
&=
\frac{1}{\Delta(t)}\sum_{k\in \mathbb{Z}} \bigl(X, t^{k+n_3-n_2}\widetilde{L_3}\bigr)t^k \\
&=
\left(
\frac{1}{\Delta(t)}\sum_{k\in \mathbb{Z}} \bigl(X, t^k\widetilde{L_3}\bigr)t^k
\right)t^{n_2-n_3}
&=h(L_1,\widetilde{L_2}, \widetilde{L_3})t^{n_2-n_3}.
\end{array}
\end{equation*}
%Thus, changing the choices of lifts of $L_2$ and $L_3$ changes $\noteta(L)$ by multiplication by a power of $t$. This proves the second statement.

We now check all at once for invariance under weak cobordism and independence from the choices of $\Delta(t)$ and $X$. We continue to work in the smooth setting for the sake of ease. The details needed to work over $I$-equivalence appear in Remark~\ref{rmk: what if I-equiv}. Let $L=(L_1,L_2,L_3)$ and $L'=(L_1',L_2',L_3')$ be weakly cobordant links with a distinguished component. Let $A$, $Y_2$, and $Y_3$ be the surfaces of Definition~\ref{defn:weak cob}.

Let $\widetilde{Y}_2$ and $\widetilde{Y}_3$ be the respective lifts of $Y_2$ and $Y_3$ to the infinite cyclic cover of $E(A) = S^3 \times [0, 1] \smallsetminus A$, chosen with orientations satisfying $\bdry Y_i = L_i - L_i'$ for $i = 2, 3$. As we have already determined the effect of the choice of lifts, we may as well pick the lifts of $L_i$, $L_i'$ and $Y_i$ so that  $\bdry \widetilde{Y_i}=\widetilde{L_i}-\widetilde{L_i'}$.

Let $X$ and $X'$ be $2$-chains bounded by $\Delta(t) \widetilde{L_2}$ and $\Delta'(t) \widetilde{L_2'}$, respectively. Then
\[
Q=\Delta'(t)X+\Delta(t)\Delta'(t)\widetilde{Y_2}-\Delta(t)X'
\]
is a $2$-cycle. Since $H_2(\widetilde{E(A)})$ is torsion as a $\Z[t,t^{-1}]$-module, it follows that for some $q(t)\in \Z[t,t^{-1}]$, the cycle $q(t) Q$ is a boundary. For any $k\in\Z$,
\[
0=\bigl(q(t) Q, t^k\widetilde{Y_3}\bigr).
\]
Here, $\bigl(q(t) Q, t^k\widetilde{Y}_3\bigr)$ is the pairing with respect to the intersection form $$H_2(\widetilde{E(A)}) \times H_2(\widetilde{E(A)}, \bdry\widetilde{E(A)}) \to \Z.$$
Since $Y_2$ and $Y_3$ are disjoint, the same is true for their lifts. Thus, the only intersections between $q(t) Q$ and $t^k\widetilde{Y_3}$ occur where $t^k\widetilde{L_3}$ intersects $q(t)\Delta'(t)X$ and where $t^k\widetilde{L_3'}$ intersects $q(t)\Delta(t)X'$. Therefore,
\[
0=\bigl(q(t)\Delta'(t)X, t^k\widetilde{L_3}\bigr)-\bigl(q(t)\Delta(t)X', t^k\widetilde{L_3'}\bigr).
\]
Summing over all $k\in \mathbb{Z}$, we conclude that
\[
q(t)\Delta'(t)\sum_{k\in \mathbb{Z}} \bigl(X, t^k\widetilde{L_3}\bigr)t^k
=
q(t)\Delta(t)\sum_{k\in \mathbb{Z}} \bigl(X', t^k\widetilde{L_3'}\bigr)t^k.
\]  Dividing by $q(t)\Delta(t)\Delta'(t)$, we conclude
\[
\noteta(L_1,\widetilde{L_2},\widetilde{L_3})
=
\noteta(L_1',\widetilde{L_2'},\widetilde{L_3'})\in \Q(t),
\]
which completes the proof.
\end{proof}

We remark that Proposition~\ref{prop h is invariant} also holds under $I$-equivalence:

\begin{remark}\label{rmk: what if I-equiv}
Just as in the setting of Kojima--Yamasaki, the argument above goes through with only minor changes if one drops the smoothness conditions from the definition of weak cobordism and allows continuously embedded annuli and surfaces, provided that the links $L$ and $L'$ are still tame. Indeed, much of the content of this remark is an appeal to \cite[Lemma~5]{KojYam79}. If $A$ is a continuously embedded annulus in $S^3\times[0,1]$, we write
\[
E(A):=S^3\times[0,1]\smallsetminus A.
\]

% First we note that even without smoothness, $H_1(S^3\times[0,1]\smallsetminus A)\cong \Z$, so that we may still consider its infinite cyclic cover. Indeed, consider the embedded $2$-sphere $S$ in $S^4$ given by $A$ together with the cones on $L_1$ and $L_1'$. By Alexander--Pontryagin duality (see, for example, \cite[Theorem~74.1]{Munkres84}), we have $H_1(S^4\smallsetminus S)\cong \check{H}^2(S)\cong \Z$. Since $S^3\times(0,1)\smallsetminus A$ is homeomorphic to $S^4\smallsetminus S$, we conclude that $H_1(S^3\times[0,1]\smallsetminus A)\cong \Z$. Similarly, we also conlude $H_2(S^3\times[0,1]\smallsetminus A)=0$.

% \newpage

First note that, since $A$ is closed, $E(A)$ is 
an open subset of $S^3\times[0,1]$, which is a smooth
manifold with boundary, 
and hence $E(A)$ is a smooth (noncompact)
manifold with boundary. Let $S\subset S^4$ be the embedded $2$-sphere obtained
from $A$ by capping off its boundary components with cones on $L_1$ and $L_1'$.
Then $\operatorname{int}(E(A))$ is homeomorphic to $S^4\smallsetminus S$.
Since a manifold with boundary is homotopy equivalent to its interior, it
follows that $E(A)$ is homotopy equivalent to $S^4\smallsetminus S$.

By Alexander--Pontryagin duality (see, for example, \cite[Theorem~74.1]{Munkres84}),
\[
\widetilde{H}_i(E(A))\cong \widetilde{H}_i(S^4\smallsetminus S)\cong \check{\widetilde{H}}^{3-i}(S),
\]
where $\widetilde{H}$ indicates reduced  homology and $\check{\widetilde{H}}$ indicates reduced \v{C}ech cohomology.   Since $E(A)$ is connected and \(S\cong S^2\),  we obtain
\(
H_0(E(A))\cong H_1(E(A))\cong \Z
\)
and all other reduced homology groups vanish. Equivalently,
\[
H_*(E(A))\cong H_*(S^1).
\]
In particular, \(E(A)\) admits an infinite cyclic cover corresponding to a
generator of \(H_1(E(A))\), and \cite[Lemma~5]{KojYam79} 
%applies.
%Next, as observed in \cite[Lemma~5]{KojYam79}, this is enough to 
implies that the intersection form
$$H_2(\widetilde{E(A)})\times H_2(\widetilde{E(A)}, \bdry\widetilde{E(A)})\to \Z$$
vanishes. With these observations, the proof of Proposition~\ref{prop h is invariant} goes through unchanged if $A$ is continuously embedded, and $Y_2$ and $Y_3$ are any embedded surfaces which lift to the infinite cyclic cover of $E(A)$.
\end{remark}

% By a slight abuse of notation, we will use $h(L)$ to denote any representative
% of the class of $h(L_1,\widetilde{L_2},\widetilde{L_3}) \in \sfrac{\Q(t)}{\doteq}$. 

The denominator of $h(L_1,\widetilde{L_2},\widetilde{L_3})\in \Q(t)$ can be taken to be any annihilator of the class of $\widetilde{L_2}$ in the Alexander module of $L_1$. The Alexander polynomial of $L_1$ is such an annihilator. Thus, the denominator of $h(L_1,\widetilde{L_2},\widetilde{L_3})$ is coprime to $(t-1)$, and so we may expand $h(L_1,\widetilde{L_2},\widetilde{L_3})$ as a power series in $t-1$:
$$
h(L_1,\widetilde{L_2},\widetilde{L_3})=\sum_{k=0}^\infty (t-1)^k p_k\in \Q(t).
$$
{The following relation between the $\gamma$-invariants and
$h(L_1,\widetilde{L_2},\widetilde{L_3})$ was established in
\cite[Theorem~1.4]{Tsukamoto-Yasuhara:2007}; compare also
\cite[Theorem~7.1]{C3}. Here again, we provide a proof for completeness.}

\begin{proposition}[{\cite[Theorem~1.4]{Tsukamoto-Yasuhara:2007}}]\label{thm: h is gamma}
Let $L=(L_1,L_2,L_3)$ be a link with  a distinguished component. Consider any Seifert surface $G_1$ for $L_1$.
%disjoint from $L_2$ and $L_3$. 
If $\widetilde{L_2}$ and $\widetilde{L_3}$ both sit in the same lift of $S^3\smallsetminus G_1$, then
$$
\noteta(L_1, \widetilde{L_2},\widetilde{L_3})= \sum_{k=0}^\infty (t-1)^k \gamma^k(L,G_1).
$$
\end{proposition}

\begin{proof}
Let $G_1$ be a Seifert surface for $L_1$.
%disjoint from $L_2$ and $L_3$. 
Then $S^3\smallsetminus G_1$ forms a fundamental domain for the infinite cyclic cover $\widetilde{S^3\smallsetminus L_1}$. Let $\widetilde{L_2}$ and $\widetilde{L_3}$ both sit in the same lift of $S^3\smallsetminus G_1$.
Let $G_2$ be a Seifert surface for $L_2$ 
%disjoint from $L_1$ and
intersecting $G_1$ in a derivative $L_{12}$. Then the lift of $G_2$ cut open along $L_{12}$ gives a $2$-chain $H$ in $\widetilde{S^3\smallsetminus L_1}$ bounded by $\widetilde{L_2}-(t-1)\widetilde{L_{12}}$, with
\[
\bigl(H, t^k \widetilde{L_3}\bigr)=
\begin{cases}
\lk(L_2,L_3) & \text{if } k=0,\\
0 & \text{otherwise.}
\end{cases}
\]
If $X$ is a $2$-chain bounded by $\Delta(t)\widetilde{L_{12}}$, then
\[
\bdry\bigl((t-1)X+\Delta(t)H\bigr)=\Delta(t)\widetilde{L_2},
\]
and so
\begin{align*}
\noteta(L_1,\widetilde{L_2},\widetilde{L_3})
&=
\frac{1}{\Delta(t)}\sum_{k\in \mathbb{Z}} \bigl((t-1)X+\Delta(t)H, t^k\widetilde{L_3}\bigr)t^k \\
&=
\frac{1}{\Delta(t)}\sum_{k\in \mathbb{Z}} \bigl((t-1)X, t^k\widetilde{L_3}\bigr)t^k
+
\frac{1}{\Delta(t)}\sum_{k\in \mathbb{Z}} \bigl(\Delta(t)H, t^k\widetilde{L_3}\bigr)t^k \\
&=
(t-1)\noteta(L_1,\widetilde{L_{12}},\widetilde{L_3})+\lk(L_2, L_3).
\end{align*}

Applying the same argument to $\noteta(D(L,G_1))$ yields
\begin{align*}
\noteta(L_1,\widetilde{L_2},\widetilde{L_3})
&=(t-1)\bigl((t-1)\noteta(L_1,\widetilde{L_{112}},\widetilde{L_3})+\lk(L_{12},L_3)\bigr)+\lk(L_2,L_3) \\
&=(t-1)^2\noteta(L_1,\widetilde{L_{112}},\widetilde{L_3})+(t-1)\gamma^1(L,G_1)+\gamma^0(L,G_1).
\end{align*}
Repeating this process $k$ times gives
\[
\noteta(L_1,\widetilde{L_2},\widetilde{L_3})=(t-1)^{k+1}\noteta(L_1,\widetilde{L_{1^{k+1}2}},\widetilde{L_3})+\sum_{i=0}^{k}(t-1)^i\gamma^i(L,G_1).
\]
Note that similarly to $\noteta(L_1,\widetilde{L_2},\widetilde{L_3})$, the rational function $\noteta(L_1,\widetilde{L_{1^{k+1}2}},\widetilde{L_3})$ admits an expansion as a power series in $(t-1)$:
\[
\noteta(L_1,\widetilde{L_{1^{k+1}2}},\widetilde{L_3})=\sum_{j=0}^\infty r_j (t-1)^j
\]
for some sequence $(r_0,r_1,\dots)$. Thus,
\[
\noteta(L_1,\widetilde{L_2},\widetilde{L_3})=\sum_{j=0}^\infty r_j (t-1)^{j+k+1}+\sum_{i=0}^{k}(t-1)^i\gamma^i(L,G_1).
\]
In particular, for all $i\le k$, the coefficient of $(t-1)^i$ in $\noteta(L_1,\widetilde{L_2},\widetilde{L_3})$ is $\gamma^i(L,G_1)$. Since $k$ is arbitrary, it follows that the coefficient of $(t-1)^i$ in $\noteta(L_1,\widetilde{L_2},\widetilde{L_3})$ is $\gamma^i(L,G_1)$ for every $i\ge 0$. This proves the theorem.
\end{proof}

Proposition~\ref{prop h is invariant} and Proposition~\ref{thm: h is gamma} together establish the invariance of $\gamma$. Indeed, if $G$ and $G'$ are two Seifert surfaces for $L_1$ disjoint from $L_2$ and $L_3$, then by Proposition~\ref{prop h is invariant} and Proposition~\ref{thm: h is gamma}, both
\[
\sum_{k=0}^\infty (t-1)^k \gamma^k(L,G)
\qquad\text{ and }\qquad
\sum_{k=0}^\infty (t-1)^k \gamma^k(L,G')
\]
represent $h(L)\in \sfrac{\Q(t)}{\doteq}$. Hence, they differ by multiplication by a power of $t$. Substituting $x=t-1$, we obtain
\[
\sum_{k=0}^\infty x^k \gamma^k(L,G)=(x+1)^i\sum_{k=0}^\infty x^k \gamma^k(L,G')
\]
for some $i\in \Z$. Since Proposition~\ref{prop h is invariant} also shows that $h(L)$ is an invariant of the weak cobordism class of $L$, the following corollary is immediate.

\begin{corollary}\label{corollary: well defined}
Let $\Z^\infty=\{(a_0,a_1,\dots)\}$ denote the set of all integer-valued sequences, and let $T\colon \Z^\infty\to \Z^\infty$ be the right shift operator defined by
\[
T(a_0,a_1,\dots)=(0,a_0,a_1,\dots).
\]
Let $\sim$ be the equivalence relation on $\Z^\infty$ generated by
$a\sim (T+\mathrm{Id})(a)$
for each $a\in \Z^\infty$. Then the class of
\[
\gamma(L)=(\gamma^0(L),\gamma^1(L),\dots)\in \Z^\infty/{\sim}
\]
depends only on the weak cobordism class of $L$. \qed
\end{corollary}

In fact, it is straightforward to determine whether two sequences in $\Z^\infty$ are equivalent with respect to the equivalence relation of Corollary~\ref{corollary: well defined}.  We start with an illustrative example. In Figure~\ref{fig: arbitrary gamma}, we exhibited links for which $(\gamma^0(L),\dots,\gamma^n(L))$ can take any prescribed sequence of values, with $\gamma^k(L)=0$ for $k>n$. Consider two links with $\gamma(L)=(1,3,r,0,0,\dots)$ and $\gamma(L')=(1,4,q,0,0,\dots)$ for some $q,r\in \Z$. We take a moment to analyze whether these are equal modulo the action of $(T+\mathrm{Id})$. A direct induction shows that $$(T+\mathrm{Id})^m(1,3,\dots)=(1,3+m,\dots),$$ so the only way to obtain a sequence beginning with $(1,4,\dots)$ from $(1,3,r,0,0,\dots)$ is to take $m=1$. In that case,
\[
(T+\mathrm{Id})(1,3,r,0,0,\dots)=(1,4,3+r,r,0,\dots).
\]
This agrees with $(1,4,q,0,0,\dots)$ only if $r=0$ and $q=3$.

This example illustrates the following. If $\gamma^k(L)$ is the first nonvanishing $\gamma$-invariant of $L$, then $\gamma^k\in \Z$ is an invariant, $\gamma^{k+1}$ is well defined modulo $\gamma^k$, and after accounting for that indeterminacy, $\gamma^n\in \Z$ is well defined for all $n>k+1$. The following makes this explicit.

\begin{corollary}\label{cor:lift}
Let $L=(L_1, L_2, L_3)$ and $L'=(L_1',L_2', L_3')$ be weakly cobordant links, and let $k\in \N$. Assume that $\gamma^j(L)=0$ for all $j<k$ and $\gamma^k(L)\neq 0$. Then:
\begin{enumerate}
\item \label{item:cor-gamma-k}
$\gamma^k(L')=\gamma^k(L)$.
\item \label{item:cor-gamma-k+1}
$\gamma^{k+1}(L')-\gamma^{k+1}(L)=p \gamma^k(L)$ for some $p\in \Z$.
\item \label{item:cor-gamma-ell}
If $\gamma^{k+1}(L')-\gamma^{k+1}(L)=p \gamma^k(L)$ with $p\ge 0$, then 
\begin{eqnarray*}\gamma^{k+2}(L')&=&\gamma^{k+2}(L)+p\gamma^{k+1}(L)+{p\choose 2}\gamma^k(L),\\
\gamma^{k+3}(L')&=&\gamma^{k+3}(L)+p\gamma^{k+2}(L)+{p\choose 2}\gamma^{k+1}(L)+{p\choose 3}\gamma^{k}(L),
\end{eqnarray*}
and in general, if we follow the convention that $\gamma^i(L)=0$ when $i<0$ then for all $\ell>k+1$,
\begin{eqnarray*}
\gamma^\ell(L') &=& \gamma^\ell(L)+p\gamma^{\ell-1}(L)+{p\choose 2}\gamma^{\ell-2}(L)+\dots + p\gamma^{\ell-p+1}(L)+\gamma^{\ell-p}(L) \\
&=& \sum_{q=0}^p {p\choose q}\,\gamma^{\ell-q}(L).
\end{eqnarray*}
\end{enumerate}
\end{corollary}

\begin{proof}
Since $L$ and $L'$ are weakly cobordant, it follows (after the change of variables $x=t-1$) that
$$
\noteta(L')=(1+x)^n\noteta(L)
$$
for some $n\in \Z$. By swapping $L$ and $L'$ as appropriate, we may assume that $n\ge 0$. By the binomial theorem and recalling that $\gamma^j(L)=0$ for all $j<k$, we have
$$\noteta(L')=\Sum_{q=0}^n {n\choose q}x^q\cdot \Sum_{j=k}^\infty \gamma^j(L)\,x^j.$$
Expanding the product, reindexing by $j+q=m$, and using that $\gamma^{m-q}(L)=0$ when $m-q<0$, 
\begin{equation}\label{eqn: expand prod}
\noteta(L')
=
\Sum_{m=k}^\infty
\left(\Sum_{q=0}^{n}{n\choose q}\gamma^{m-q}(L)\right)x^m.
\end{equation}
We can compute the $x_k$-coefficient to get
$$
\gamma^k(L')=\Sum_{q=0}^{n}{n\choose q}\gamma^{k-q}(L)=\gamma^k(L),
$$
where the second equality follows since $\gamma^j(L)=0$ for all $j<k$.  This proves \pref{item:cor-gamma-k}.

Comparing the coefficient of $x^{k+1}$ in equation \pref{eqn: expand prod} yields
$$
\gamma^{k+1}(L')={n\choose 0}\gamma^{k+1}(L)+{n\choose 1}\gamma^{k}(L)
$$
so $\gamma^{k+1}(L')-\gamma^{k+1}(L)=n \gamma^k(L)$. This proves claim \pref{item:cor-gamma-k+1} (with $p=n$).

Finally, assume $\gamma^{k+1}(L')-\gamma^{k+1}(L)=p \gamma^k(L)$ with $p\ge 0$. Since $\gamma^k(L)\neq 0$, the previous paragraph implies that $p=n$. Comparing the coefficient of $x^\ell$ in equation~\eqref{eqn: expand prod} for any $\ell>k+1$ gives
$$
\gamma^\ell(L')=\Sum_{q=0}^{p}{p\choose q}\gamma^{\ell-q}(L),
$$
proving \pref{item:cor-gamma-ell}.
\end{proof}

\section{The $\gamma$-invariant as a lift of Milnor invariants}\label{sect: lift of Milnor}

As mentioned in the introduction, the Milnor invariant $\bar{\mu}_I(L)$ is, in general, not an integer-valued invariant; rather, it is only well defined modulo the indeterminacy~\cite{M2}. More precisely, for any link $L$ and any fixed multi-index $I=(i_1 \dots i_n)$, let $J_I\subseteq \Z$ be the ideal generated by
\[
\bigl\{\bar{\mu}_J(L)\,\bigm|\,
J=(i_{k_1}\dots i_{k_m}),\ 1< m<n,\ 1\le k_1<\cdots<k_m\le n\bigr\}.
\]
Then $\bar{\mu}_I(L)$ is well defined as an element of $\Z/J_I$. 

In \cite{Porter80}, Porter interprets Milnor's $\bar{\mu}$-invariants in terms of Massey products in the link complement; in particular, this yields link invariants with the same indeterminacy as Milnor's $\bar{\mu}$-invariants. In \cite[Theorem~6.10]{Coc90}, Cochran uses Massey products to prove that his integer-valued $\beta$-invariant recovers $\bar{\mu}(1^{2k}22)$ modulo this indeterminacy. In this section we follow a similar path to prove that $\gamma^k$ recovers $\bar{\mu}(1^{k}23)$.

\begin{theorem}\label{thm: gamma as Milnor}
Let $I=(1^k23)$ and let $L$ be a $3$-component link with $\lk(L_1,L_2)=\lk(L_1,L_3)=0$. Then $\gamma^k(L)\equiv \bar{\mu}_L(I)\mod J_I$.
\end{theorem}

We begin by recalling Porter's formulation of Milnor invariants in terms of Massey products.
Let $X_1,\dots,X_p$ be topological spaces and let $R$ be a commutative ring. For each $i$,
let $x_i\in H^1(X_i;R)$.
A \emph{defining system} for the Massey product $\langle x_1,\dots,x_p\rangle$ with respect to
$\{X_1,\dots,X_p\}$ and coefficients in $R$ is a collection of $1$-cochains
\[
\{m_{i,j}\mid 1\le i\le j\le p,\ (i,j)\neq (1,p)\}
\]
satisfying:
\begin{enumerate}
\item $m_{i,j}\in C^1(X_i\cap\cdots\cap X_j;R)$,
\item $m_{i,i}$ is a cocycle representative for $x_i$,
\item For $i<j$, we have
\[
\delta(m_{i,j})=\sum_{k=i}^{j-1} m_{i,k}\cup m_{k+1,j},
\]
where $\cup$ denotes the cup product on $X_i\cap\cdots\cap X_j$.
\end{enumerate}
The Massey product $\langle x_1,\dots,x_p\rangle$ is the subset of
$H^2(X_1\cap\cdots\cap X_p;R)$ consisting of the collection of all
\[
\left[\sum_{k=1}^{p-1} m_{1,k}\cup m_{k+1,p}\right],
\]
as the defining system varies. Porter shows that, for link complements, this recovers the
corresponding Milnor invariant.

\begin{theorem}[{\cite[Theorem~3]{Porter80}}]\label{thm: Porter}
Let $L=(L_1,\ldots, L_n)$ be a link in $S^3$, let $I=(i_1,\dots,i_p)$ be a sequence with
$1\le i_k\le n$, and let $R=\Z/J_I$. For each $i=1,\dots,n$, let
$u_i\in H^1(S^3\smallsetminus L_i;R)\cong R$ be the class dual to the meridian of $L_i$.
Then the Massey product $\langle u_{i_1},\dots,u_{i_p}\rangle$ with respect to
$\{S^3\smallsetminus L_{i_1},\dots,S^3\smallsetminus L_{i_p}\}$ is defined and contains only
$(-1)^p\,\bar{\mu}_L(I)\cdot t_{i_1,i_p}$. Here $t_{i_1,i_p}\in H^2(S^3\smallsetminus L;R)$ is the
Lefschetz dual in $S^3\smallsetminus L$ to an arc running from $L_{i_1}$ to $L_{i_p}$.
\end{theorem}

\begin{proof}[Proof of Theorem~\ref{thm: gamma as Milnor}]
Let $L=(L_1,L_2,L_3)$ be a $3$-component link with a 
distinguished component.
%$\lk(L_1,L_2)=\lk(L_1,L_3)=0$.
Let $G_1$ be a Seifert surface for $L_1$
and
%disjoint from $L_2$ and $L_3$.  Let 
$G_2$ and $G_3$ be Seifert surfaces for $L_2$ and $L_3$.
%disjoint from $L_1$.
Let $I=(1^k23)$. Note that $G_1$, $G_2$, and $G_3$ determine
%$R:=\Z/J_I$
intersection duals to the meridians $m_1$, $m_2$, and $m_3$ of $L$.

Let $u_i=\PD(G_i)$ be the Poincar\'e dual of $G_i$. For $r=1,\dots,k$, let $F_r$ be a Seifert surface for the derivative $L_{1^r2}$, chosen so the interior of $F_r$ intersects $G_1$ in $L_{1^{r+1}2}$, and write $v_r=\PD(F_r)$. We now construct a defining system for
\[
\left\langle \overbrace{u_1,\dots,u_1}^{k\text{ times}},u_2,u_3\right\rangle.
\]
The entries we prescribe are summarized by the table

\begin{equation*}
\begin{array}{c|ccccccc}
m_{i,j}&j=1&j=2&j=3&\dots&j=k&j=k+1&j=k+2\\\hline
i=1&u_1&0&0&\dots&0&v_{k}&
\\
i=2&&u_1&0&\dots&0&v_{k-1}&m_{2,k+2}
\\
i=3&&&u_1&\dots&0&v_{k-2}&m_{3,k+2}
\\
\vdots&&&&\ddots
\\
i=k&&&&&u_1&v_1&m_{k,k+2}
\\
i=k+1&&&&&&u_2&m_{k+1,k+2}
\\
i=k+2&&&&&&&u_3
\end{array}
\end{equation*}

The fact that $m_{i,j}=0$ when $i<j<k+1$ follows from the identity
\[
u_1\cup u_1=\PD(G_1\cap G_1)=0.
\]
The equality $m_{k,k+1}=v_1$ follows since
\[
u_1\cup u_2=\PD(G_1\cap G_2)=\PD(L_{12})
\qquad\text{and}\qquad
\delta(v_1)=\PD(\partial F_1)=\PD(L_{12}).
\]
More generally, for $r\ge 2$ we have $\partial F_r=L_{1^r2}=G_1\cap F_{r-1}$, hence
\[
\delta(v_r)=\PD(\partial F_r)=\PD(G_1\cap F_{r-1})=u_1\cup v_{r-1},
\]
so the remaining entries in the column $j=k+1$ satisfy the defining-system equation.

For the last column, since $\bar{\mu}_L(23)=\lk(L_2,L_3)\in J_I$, the cup product $u_2\cup u_3$ vanishes in $H^2(S^3\smallsetminus (L_2\cup L_3);R)$. Hence we may choose $m_{k+1,k+2}$ with
\[
\delta(m_{k+1,k+2})=u_2\cup u_3.
\]
Now suppose inductively that $m_{r,k+2}$ has been chosen for all $r>i$, where $2\le i\le k$.
Then
\[
u_1\cup m_{i+1,k+2}+v_{k-i+1}\cup u_3
\]
is a cocycle representing the shorter Massey product
\[
\left\langle \overbrace{u_1,\dots,u_1}^{k-i+1\text{ times}},u_2,u_3\right\rangle.
\]
By Theorem~\ref{thm: Porter}, this class equals
\[
(-1)^{k-i+3}\bar{\mu}_L(1^{k-i+1}23)\,t_{1,3}.
\]
Since $\bar{\mu}_L(1^{k-i+1}23)\in J_I$, it vanishes in $H^2(S^3\smallsetminus L;R)$. Hence we may choose $m_{i,k+2}$ with
\[
\delta(m_{i,k+2})=u_1\cup m_{i+1,k+2}+v_{k-i+1}\cup u_3.
\]
This completes the defining system.

Therefore the Massey product contains the class represented by
\[
u_1\cup m_{2,k+2}+v_k\cup u_3.
\]
Appealing again to Theorem~\ref{thm: Porter}, to compute $\bar{\mu}_L(1^k23)$ it suffices to evaluate
\[
u_1\cup m_{2,k+2}+v_k\cup u_3
\]
on $[T]$, the class in $H_2$ of the torus about $L_3$.
To do so, represent $m_{2,k+2}\in C^1(S^3\smallsetminus L)$ as $\PD(H)$ for some sum of surfaces $H$ in $S^3\smallsetminus L$.
Then $u_1\cup m_{2,k+2}=\PD(G_1\cap H)$. Since $G_1$ is disjoint from $L_3$, it follows that
\[
(u_1\cup m_{2,k+2})[T]=0.
\]
%and $G_1\cap H$ is a sum of ribbon, circle, and clasp intersections
%The ribbon and circle intersections do not contribute to $\bigl(u_1\cup m_{2,k+2}\bigr)[T]$,
%Each clasp intersection contributes in a way controlled by the assumption $\lk(L_1,L_3)=0$.Thus $\bigl(u_1\cup m_{2,k+2}\bigr)[T]=0$.
%so that $\bigl(u_1\cup m_{2,k+2}\bigr)[T]$ is the signed count of the clasp intersections.  Thus $\bigl(u_1\cup m_{2,k+2}\bigr)[T]=\lk(L_1,L_3)$, which is assumed to be zero.
On the other hand, $v_k\cup u_3 = \PD(F_k\cap G_3)$ and $(v_k\cup u_3)[T]$ returns the signed count of intersections between $F_k$ and $L_3$, so
\[
\bigl(v_k\cup u_3\bigr)[T]= \lk(L_{1^k2},L_3)=\gamma^k(L),
\]
which completes the proof.
\end{proof}

\section{Classification of links up to weak cobordism}\label{sect: classify}

Our next goal is to classify links up to weak cobordism. Since this classification depends in interesting ways on the category, we make that distinction explicit here. We say that links $L$ and $L'$ are \emph{smoothly} or \emph{topologically} weakly cobordant if the annulus $A$ and the surfaces $Y_2$ and $Y_3$ in Definition~\ref{defn:weak cob} are smoothly embedded or locally flatly embedded, respectively. We say that they are \emph{continuously} weakly cobordant if these surfaces are merely continuously embedded. As noted in Remark~\ref{rmk: what if I-equiv}, the $\beta$- and $\gamma$-invariants obstruct continuous weak cobordism. In this section, we construct a complete set of obstructions to smooth, topological and continuous weak cobordism.

It is clear that if $L=(L_1,L_2,L_3)$ is a $3$-component link, then the concordance class of $L_1$, Cochran's $\beta$-invariants, and the $\gamma$-invariants are all invariants of weak cobordism. There is one additional piece of data. Suppose $L$ and $L'$ are weakly cobordant, and $A$, $Y_2$, and $Y_3$ are as in Definition~\ref{defn:weak cob}. Then the lifts $\widetilde{Y_i}$, for $i=2,3$, induce an equivalence between
$[\widetilde{L_i}]$ and $[\widetilde{L_i'}]$ in $H_1(\widetilde{E(A)})$, where
\[
E(A):=S^3\times[0,1]\smallsetminus A,
\]
and $\widetilde{E(A)}$ denotes the infinite cyclic cover of $E(A)$. We will show that, together with the $\beta$- and $\gamma$-invariants, this observation yields a complete classification of links up to weak cobordism.

Let $G$ be a Seifert surface for $L_1$ disjoint from $L_2$ and $L_3$. Choosing a lift of
$S^3\smallsetminus G$ to the infinite cyclic cover of $S^3\smallsetminus L_1$ determines preferred
lifts $\widetilde{L_2}$ and $\widetilde{L_3}$.

\begin{definition}\label{defn: very weak cobordism}
Let $(L,G)$ and $(L',G')$ be $3$-component links with distinguished components, where $G$ and $G'$ are Seifert surfaces for $L_1$ and $L_1'$, respectively, disjoint from the other components. Let $\widetilde{L_i}$ and $\widetilde{L_i'}$ denote the preferred lifts determined by chosen lifts of $S^3\smallsetminus G$ and $S^3\smallsetminus G'$, respectively. Let $n\in \Z$.  We say that $(L,G)$ and $(L',G')$ are \emph{smoothly}, \emph{topologically}, or \emph{continuously very weakly $n$-cobordant} if the following hold:
\begin{enumerate}
\item \label{item: very weak cob - concordance}
$L_1\times\{0\}$ and $L_1'\times\{1\}$ cobound a smoothly embedded, locally flatly embedded, or continuously embedded annulus $A$ in $S^3\times[0,1]$, respectively.

\item \label{item: very weak cob - homology}
In $H_1(\widetilde{E(A)})$,
\[
[\widetilde{L_2}]=t^{n_2}[\widetilde{L_2'}]
\qquad\text{and}\qquad
[\widetilde{L_3}]=t^{n_3}[\widetilde{L_3'}],
\]
where $n_2-n_3=n$.
\end{enumerate}
Moreover, we say that $(L,G)$ and $(L',G')$ are \emph{very weakly cobordant} if they are very weakly $n$-cobordant for some integer $n$.
\end{definition}

From now on, we write \textsc{CAT} to denote one of the smooth, topological, or continuous categories. It follows directly from Definition~\ref{defn:weak cob}  that, if $L$ and $L'$ are weakly cobordant via surfaces $A$, $Y_2$, and $Y_3$, then, for any Seifert surfaces $G$ and $G'$ for the distinguished components, disjoint from the other components, the preferred lifts determined by chosen lifts of $S^3\smallsetminus G$ and $S^3\smallsetminus G'$ satisfy
\[
[\widetilde{L_i}]=t^{n_i}[\widetilde{L_i'}]
\qquad\text{for } i=2,3
\]
in $H_1(\widetilde{E(A)})$ for some integers $n_2,n_3$. Hence $(L,G)$ and $(L',G')$ are very weakly $n$-cobordant, where $n=n_2-n_3$. The following theorem shows that very weak cobordism, together with the $\beta$- and $\gamma$-invariants, gives a complete set of invariants for $3$-component links up to weak cobordism. Recall that $\sim$ denotes the equivalence relation on $\Z^\infty$ generated by
\[
a\sim (T+\mathrm{Id})(a).
\]

% It follows immediately from the definition of weak cobordism since the cobordism $Y_2$ and $Y_3$ lift to the infinite cyclic cover of the complement of $A$, that if $L$ and $L'$ are weakly cobordant, then for any choice of Seifert surfaces $G$ and $G'$ for the distinguished components, disjoint from the other components, the pairs $(L,G)$ and $(L',G')$ are very weakly $n$-cobordant for some integer $n$.\footnote{\JP{perhaps we can give a one line justification? something likne since the lifts represent the same class?}}

\begin{theorem}\label{thm:classify}
Let $\CAT$ denote one of the smooth, locally flat, or continuous categories.
Let $L=(L_1,L_2,L_3)$ and $L'=(L_1',L_2',L_3')$ be links with distinguished components, and let $G$ and $G'$ be Seifert surfaces for those distinguished components. Then $L$ and $L'$ are $\CAT$-weakly cobordant if and only if there exists $n\in \Z$ such that the following hold:
\begin{enumerate}
\item \label{item very weak cob}
$(L,G)$ and $(L',G')$ are $\CAT$ very weakly $n$-cobordant.

\item \label{item same beta}
$\beta(L_1,L_2)=\beta(L'_1,L_2')$ and $\beta(L_1,L_3)=\beta(L'_1,L_3')$.

\item \label{item same gamma}
$\gamma(L,G)=(T+\mathrm{Id})^n\gamma(L',G') \in\Z^\infty$.
\end{enumerate}
\end{theorem}

%\begin{theorem}\label{thm: classification}
%Let $L=(L_1,L_2,L_3)$ and $L'=(L_1',L_2',L_3')$ be 3-component links, $G_1$ and $G_1'$ be Seifert surfaces for $L_1$ and $L'_1$ disjoint from the other components.  $L$ is CAT-weakly cobordant to $L'$ if and only if there is some $n\in \N$ so that
%\begin{enumerate}
%\item $(L,G)$ is CAT-very weakly $n$-cobordant to $(L',G')$,
%\item For $i=2,3$ $\beta(L_1,L_i)=\beta(L_1',L_i')$ and
%\item For all $k\in \N$, $\gamma(L,G_1)=(T+\mathrm{Id})^n\gamma(L',G_1')$.
%\end{enumerate}
%\end{theorem}

%We want to explore what happens in the smooth, locally flat, and topological category.  We say that $L$ is (smoothly/topologically) weakly cobordant or very weakly cobordant to $L'$ if and annuli $A$ and surfaces $Y_i$ in these definitions are (smooth/locally flat) embedded.  We say that $L$ is weakly I cobordant to $L$ if instead they are just embedded (recall that a continuous function is called an embedding if it is a homeomorphism onto its image).  According to Remark~\ref{rmk: what if I-equiv}

% We are now ready to state our main theorem.  

%\chris{Add thing saying beta and gamma are not enough after this.}

Note that the theorem above indicates that the difference between smooth, topological and continuous weak cobordism is captured by the difference between smooth, topological and continuous very weak cobordism.  When the Alexander polynomials of $L_1$ and $L_1'$ are trivial, then we get the following consequence.
%As mentioned earlier, the fact that the surfaces $Y_2$ and $Y_3$ lift to the infinite cyclic cover of the complement of $A$ imposes a constraint on weakly cobordant links. When the Alexander module vanishes, this constraint yields a particularly satisfying classification. Some of the ideas appearing in the proof of this corollary will also play a key role in the proof of Theorem~\ref{thm:classify}.
%As mentioned earlier, the fact that the surfaces $Y_2$ and $Y_3$ lift to the infinite cyclic cover of the complement of $A$ imposes a constraint on weakly cobordant links. When the Alexander module vanishes, this constraint yields a particularly satisfying classification. Some of the ideas appearing in the proof of this corollary will also play a key role in the proof of Theorem~\ref{thm:classify}.

% The first comes from the observation that the surfaces $Y_2$ and $Y_3$ lift to induce equivalences in the first homology of the infinite cyclic cover of $A$, that is, in its Alexander module.

\begin{repcorollary}{cor:alexander-one}
Let $L=(L_1,L_2,L_3)$ and $L'=(L_1',L_2',L_3')$ be links with a distinguished component. Suppose that $L_1$ and $L_1'$ have trivial Alexander polynomial. Then the following are equivalent:
\begin{enumerate}
\item \label{item: same invts}
$\beta(L_1,L_2)=\beta(L_1',L_2')$, $\beta(L_1,L_3)=\beta(L_1',L_3')$, and
$\gamma(L)=\gamma(L')\in\Z^\infty /\sim $.

\item \label{item: top weak cob}
$L$ and $L'$ are continuously weakly cobordant.

\item \label{item: LF weak cob}
$L$ and $L'$ are locally flat weakly cobordant.
\end{enumerate}
If we additionally assume that $L_1$ and $L_1'$ cobound a smoothly embedded annulus in $S^3\times[0,1]$, then these are further equivalent to:
\begin{enumerate}\setcounter{enumi}{3}
\item \label{item: smooth weak cob}
$L$ and $L'$ are smoothly weakly cobordant.
\end{enumerate}
\end{repcorollary}

\begin{proof}
The implication~\eqref{item: LF weak cob} $\Rightarrow$ \eqref{item: top weak cob} is obvious, and \eqref{item: top weak cob} $\Rightarrow$ \eqref{item: same invts} follows from Theorem~\ref{thm:classify}.
%is the content of \cite[Corollary~5.2]{C2} and Corollary~\ref{corollary: well defined}.
We must prove that \eqref{item: same invts} $\Rightarrow$ \eqref{item: LF weak cob}. Also using Theorem~\ref{thm:classify}, it suffices to prove that, for every $n\in \Z$ and any Seifert surfaces $G$ and $G'$, the pairs $(L,G)$ and $(L',G')$ are topologically very weakly $n$-cobordant. %If we additionally assume that $L_1$ and $L_1'$ are smoothly concordant, then we will show that they are smoothly very weakly $n$-cobordant.
Item~\pref{item: smooth weak cob} will be addressed at the end of the proof. 

We have assumed that both $L_1$ and $L_1'$ have trivial Alexander polynomial. Then each of $L_1$ and $L_1'$ bounds a locally flat embedded disk in $B^4$ whose complement has fundamental group $\Z$ \cite[Theorem 1.13]{Freedman1}.  See also \cite[Theorem 11.7B]{Freedman-Quinn:1990-1}, and \cite[Appendix A]{Garoufalidis-Teichner:2004-1}. 
%each bounds a locally flat disk in a copy of $B^4$ whose complement has fundamental group $\Z$.
Taking the connected sum of these two copies of $B^4$ along these two disks produces a locally flat annulus $A$ in $S^3\times[0,1]$ bounded by $L_1\times\{0\}$ and $L_1'\times\{1\}$.

Choose lifts of $S^3\smallsetminus G$ and $S^3\smallsetminus G'$ to the corresponding infinite cyclic covers, and let
$\widetilde{L_2}$, $\widetilde{L_3}$, $\widetilde{L_2'}$, and $\widetilde{L_3'}$
be the resulting preferred lifts. Since the Alexander polynomials of $L_1$ and $L_1'$ are trivial, we have
\[
[\widetilde{L_i}]=0\in H_1(\widetilde{S^3\smallsetminus L_1})
\qquad\text{and}\qquad
[\widetilde{L_i'}]=0\in H_1(\widetilde{S^3\smallsetminus L_1'})
\]
for $i=2,3$. In particular, in $H_1(\widetilde{E(A)})$, we have
\[
[\widetilde{L_i}]=t^{n_i}[\widetilde{L_i'}]
\]
for any integer $n_i$. Thus, for every $n\in\Z$, the pairs $(L,G)$ and $(L',G')$ are topologically very weakly $n$-cobordant. An appeal to Theorem~\ref{thm:classify} completes the proof that \eqref{item: same invts} implies \eqref{item: LF weak cob}.

It remains to prove the smooth statement. Assume that $L_1$ and $L_1'$ cobound a smoothly embedded annulus. Then the same argument applies with this smoothly embedded annulus in place of the locally flat annulus above.
\end{proof}

\begin{proof}[Proof of Theorem~\ref{thm:classify}]
Let $L$ and $L'$ be $\CAT$-weakly cobordant, and let $G$ and $G'$ be Seifert surfaces for $L_1$ and $L_1'$, respectively. For $i=2,3$, let $\widetilde{L_i}$ and $\widetilde{L_i'}$ be the preferred lifts determined by chosen lifts of $S^3\smallsetminus G$ and $S^3\smallsetminus G'$, respectively. Let $A$, $Y_2$, and $Y_3$ be the surfaces of Definition~\ref{defn:weak cob}. The lift $\widetilde{Y_i}$ gives an equivalence in $H_1(\widetilde{E(A)})$ from $\widetilde{L_i}$ to $t^{n_i}\widetilde{L_i'}$. In particular, this means that $(L,G)$ and $(L',G')$ are very weakly $n$-cobordant, where $n=n_2-n_3$.

By \cite[Corollary~5.2]{C2}, $L$ and $L'$ have the same $\beta$-invariants. To show the claimed equality of the $\gamma$-invariants, we first appeal to Proposition~\ref{prop h is invariant} to see that
\[
\noteta(L_1,\widetilde{L_2},\widetilde{L_3})
=
\noteta(L_1',t^{n_2}\widetilde{L_2'},t^{n_3}\widetilde{L_3'})
=
t^n\noteta(L_1',\widetilde{L_2'},\widetilde{L_3'}).
\]
The equality of the $\gamma$-invariants follows from Proposition~\ref{thm: h is gamma}.

% Now assume that $(L,G)$ and $(L',G')$ are $n$-very weakly cobordant, so there is a CAT embedded annulus $A\subseteq S^3\times[0,1]$ and for $i=2,3$ a surface $Y_2\subseteq \widetilde{S^3\times[0,1]\setminus A}$ bounded by $\widetilde L_i$ and $t^{n_i}\widetilde{L_i'}$ with $n_2-n_3=n$.  

Conversely, assume that conditions~\eqref{item very weak cob}--\eqref{item same gamma} hold. In particular, there is a $\CAT$-embedded annulus
$A\subseteq S^3\times[0,1]$ and, for $i=2,3$, a surface
$Y_i \subseteq \widetilde{E(A)}$ (possibly neither embedded nor smooth) with boundary
\[
\partial Y_i=\widetilde{L_i}-t^{n_i}\widetilde{L_i'},
\]
where $n_2-n_3=n$. We will construct a $\CAT$-weak cobordism from $L$ to $L'$. Since $\widetilde{E(A)}$ is a cover of an open subset of the smooth manifold $S^3\times[0,1]$, it is smooth, even if $A$ is not. For our argument, we will need the surfaces $Y_i$ to be smoothly embedded. Thus we first approximate each $Y_i$ by a smoothly immersed surface.  First, by the Whitney approximation theorem, \cite[Theorem~10.21]{LeeSmoothManifolds} we may replace $Y_i$ by the image of a smooth map.   By \cite[Theorem~2.8]{AdachiEmbeddingsAndImmersions} we can further replace it by an immersion.  In the following paragraph we will replace this immersion by an embedding.
% Logic[CITE Whitney approximation on manifolds, Lee Introduction to smooth manifolds, Theorem 10.21] gets us to a smooth map [Adachi embeddings and immersions, Theorem 2.8] gets to an immersion but the codomain must be R^n. To get this: Project down to S^3\times[0,1]\smallsetminus A. Use product structure near boundary to get immersion near boundary. This embeds in R^4, so approximate it by an immersion. Now lift it back up to the cover.

Let $p(t)\in \Z[t,t^{-1}]$ be a polynomial for which
$p(t)\widetilde{L_2}$ bounds a $2$-chain
$X\in C_2(\widetilde{S^3\smallsetminus L_1})$. Choose
$p'(t)$ and $X'\in C_2(\widetilde{S^3\smallsetminus L_1'})$
similarly for $\widetilde{L_2'}$. It follows that
\[
p'(t)X-p(t)p'(t)Y_2-p(t)t^{n_2}X' \in H_2(\widetilde{E(A)}).
\]
As discussed in Remark~\ref{rmk: what if I-equiv}, the intersection form vanishes, and hence, for any $k\in \Z$,
\[
\bigl(p'(t)X-p(t)p'(t)Y_2-p(t)t^{n_2}X',\, t^kY_3\bigr)=0.
\]
By summing over all $k$, we obtain
\[
\sum_{k\in \Z} (Y_2,t^kY_3)t^k
=
\noteta(L_1,\widetilde{L_2},\widetilde{L_3})
-
\noteta(L_1',t^{n_2}\widetilde{L_2'},t^{n_3}\widetilde{L_3'})=\noteta(L_1,\widetilde{L_2},\widetilde{L_3})
-
t^{n_2-n_3}\noteta(L_1',\widetilde{L_2'},\widetilde{L_3'}).
\]
Next, appealing to Proposition~\ref{thm: h is gamma} and the assumed equality of the $\gamma$-invariants, we conclude that
\[
\noteta(L_1,\widetilde{L_2},\widetilde{L_3})
=
t^{n_2-n_3}\noteta(L_1',\widetilde{L_2'},\widetilde{L_3'}).
\]
Therefore $(Y_2,t^kY_3)=0$ for every $k\in\Z$. Similarly, $(Y_3,t^kY_2)=0$ for every $k\in\Z$. By \cite[Theorem~7.1]{C3} and the assumed equality of the $\beta$-invariants, $L$ and $L'$ have identical $\eta$-invariants. A similar argument then shows that
\[
(Y_2,t^kY_2)=(Y_3,t^kY_3)=0
\]
for every $k\in\Z$.

We now use the vanishing of these algebraic intersections to remove all intersections among the translates of $Y_2$ and $Y_3$. Since $(Y_i,Y_i)=0$, we may pair up oppositely signed self-intersections of $Y_i$ by arcs. Tubing $Y_i$ to itself along these arcs, we arrange that $Y_i$, and hence each translate $t^kY_i$, is smoothly embedded. Similarly, since $(Y_i,t^kY_j)=0$, we can modify $Y_i$ by adding further tubes to arrange that $Y_i$ is disjoint from $t^kY_j$.
By compactness, $Y_i$ intersects $t^kY_j$ for only finitely many values of $k$. Thus, after adding only finitely many tubes, we may arrange that all translates of $Y_2$ and $Y_3$ are smoothly embedded and pairwise disjoint. Consequently, if
\[
\pi:\widetilde{E(A)}\to E(A)=S^3\times[0,1]\smallsetminus A
\]
is the covering map, then $\pi(Y_2)$ and $\pi(Y_3)$ are disjoint smoothly embedded surfaces. Since $Y_2$ and $Y_3$ lift to the infinite cyclic cover, the triple
\[
(A,\pi(Y_2),\pi(Y_3))
\]
is a weak cobordism.
\end{proof}

We close this section with an example showing that, when the Alexander polynomial is nontrivial, topological weak cobordism is not classified by the $\beta$- and $\gamma$-invariants alone. Since we consider a $2$-component link in this example, the $\gamma$-invariants are vacuous. In the proof, we appeal to the solvable filtration of \cite{COT03}; we recall the necessary facts as they arise.

\begin{theorem}\label{thm:notweaklytotheunlink}
There is a $2$-component link $L=(L_1,L_2)$ such that $L_1$ is smoothly concordant to the unknot and $\beta^k(L_1,L_2)=0$ for all $k$, but $L$ is not weakly cobordant to the unlink in the locally flat category.
\end{theorem}

Since smooth weak cobordism implies topological weak cobordism, the examples above also apply in the smooth category.  We do not know whether the $\beta$- and $\gamma$-invariants alone are enough to classify continuous weak cobordism.  See Problem~\ref{problem: classify topological weak cobordism}.

\begin{proof}[Proof of Theorem~\ref{thm:notweaklytotheunlink}]
Consider the knot $R(J)$ in Figure~\ref{fig:not2-solvable}. This knot arises from the satellite construction with pattern given by the doubling operator $(R,\eta_1,\eta_2)$ in Figure~\ref{fig:doublingOp}. This doubling operator satisfies all of the hypotheses of \cite[Theorem~9.5]{CHL09}, and hence $R(J)$ is not $1.5$-solvable, provided that $J$ is a connected sum of a sufficiently large even number of trefoils.

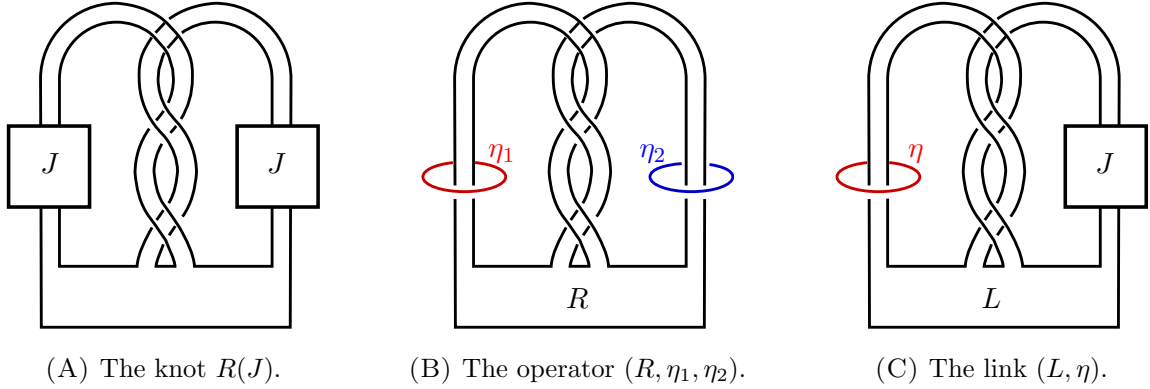
\begin{figure}[!htbp]
     \centering
     \hfill
     \begin{subfigure}[t]{0.28\textwidth}
     \centering
         \begin{tikzpicture}
         \node at (0,0){\includegraphics[width=.9\textwidth]{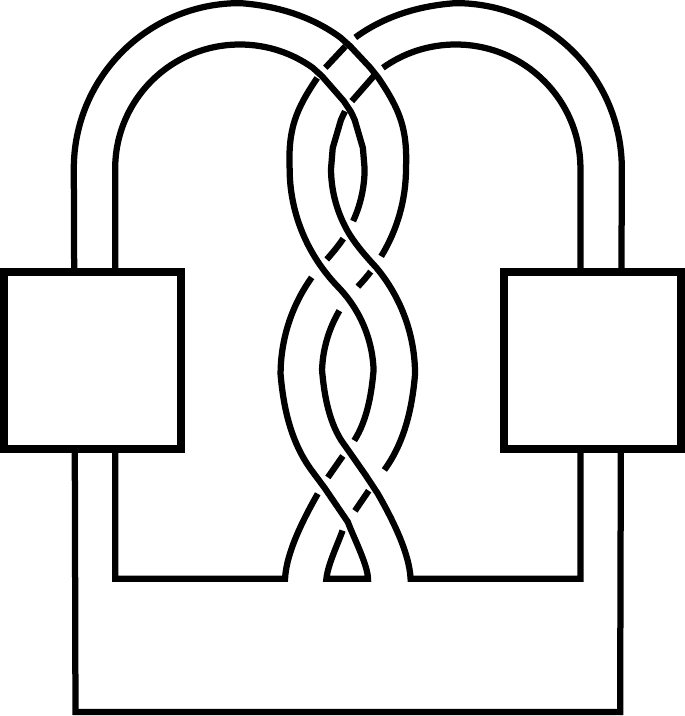}};
         \node[below] at (-1.5,.3) {$J$};
         \node[below] at (1.5,.3) {$J$};
         \end{tikzpicture}
         \caption{The knot $R(J)$.}\label{fig:not2-solvable}
         \end{subfigure}
         \hfill
     \begin{subfigure}[t]{0.28\textwidth}
     \centering
         \begin{tikzpicture}
         \node at (0,0){\includegraphics[width=.9\textwidth]{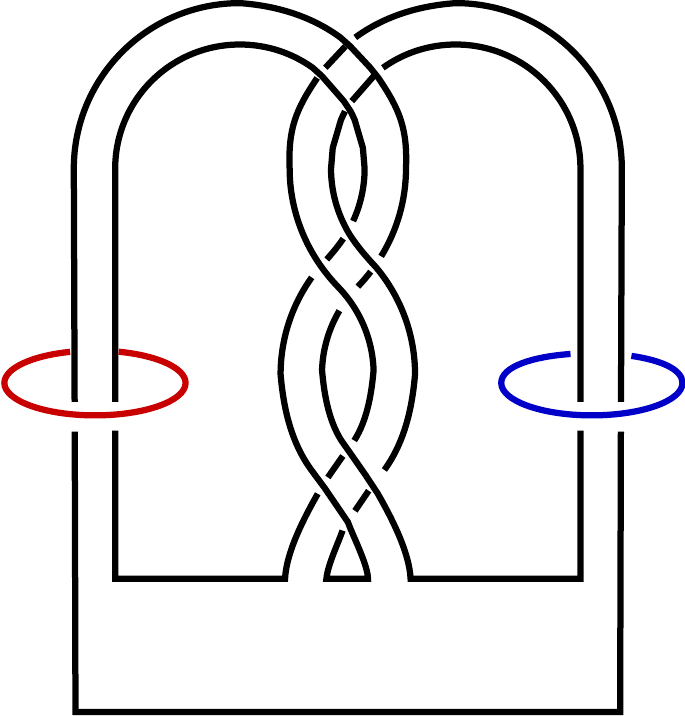}};
         \node at (0,-1.75) {$R$};
         \node[red] at (-1,.15) {$\eta_1$};
         \node[blue] at (1,.15) {$\eta_2$};
         \end{tikzpicture}
         \caption{The operator $(R,\eta_1,\eta_2)$.}\label{fig:doublingOp}
         \end{subfigure}
         \hfill
     \begin{subfigure}[t]{0.28\textwidth}
     \centering
         \begin{tikzpicture}
         \node at (0,0){\includegraphics[width=.9\textwidth]{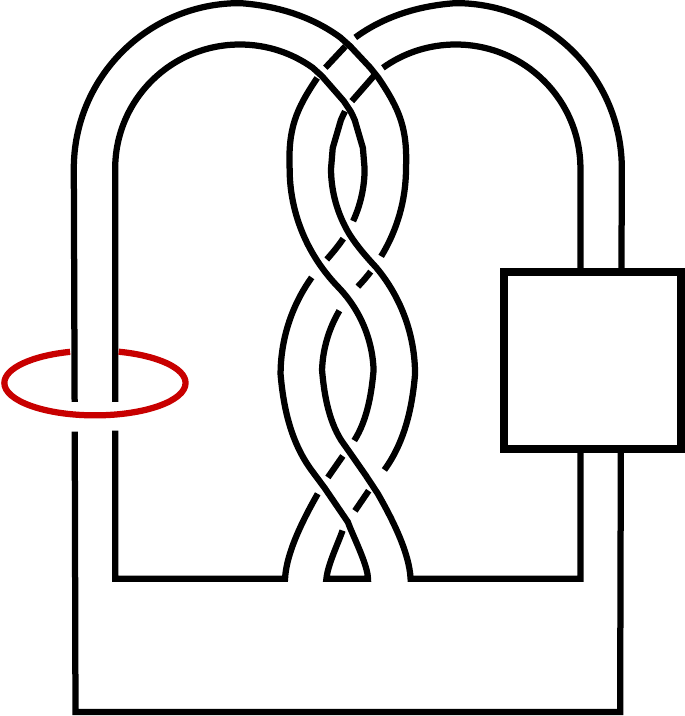}};
         \node[below] at (1.5,.3) {$J$};
         \node at (0,-1.75) {$L$};
         \node[red] at (-1,.15) {$\eta$};
         \end{tikzpicture}
         \caption{The link $(L,\eta)$.}\label{fig:two-component-link}
         \end{subfigure}
         \hfill
     \caption{A $2$-component link with vanishing $\beta$-invariants that is not weakly cobordant to the unlink.}\label{fig:non weakly cob vanishing beta}
\end{figure}

Now consider the link $(L,\eta)$ in Figure~\ref{fig:two-component-link}, where $\eta$ is the red unknotted component. The knot $J$ does not affect the linking numbers used to define the $\beta$-sequence, and hence $\beta^i(L,\eta)=0$ for all $i$, since these invariants vanish for a slice link. However, if $(L,\eta)$ were weakly cobordant to the unlink, then capping the components of a weak cobordism $(A,Y_2)$ with disks bounded by the components of the unlink would produce a disk $D$ in $B^4$ bounded by $L$ such that
\[
[\eta]\in \pi_1(B^4\smallsetminus D)^{(2)}.
\]
Recall that the derived series of a group is defined by $G^{(0)}=G$ and
\(
G^{(k+1)}=[G^{(k)},G^{(k)}].
\)
It then follows from \cite[Proposition~6.1]{DavParHarPreprint} that the knot $R(J)$ is $2$-solvable, and in particular is $1.5$-solvable, which is a contradiction.
\end{proof}

\section{Continuously embedded disks with $\pi_1=\Z$}

\newcommand{\cone}{\operatorname{cone}}%%

Although our main results concern \(3\)-component links, one theme of the paper is that linking data in infinite cyclic covers remains informative even in very weak topological settings. The invariant \(h(L)\), introduced in Section~4 and modeled on the Kojima--Yamasaki \(\eta\)-invariant~\cite{KojYam79}, is perhaps the clearest expression of this viewpoint. In the present section, we combine the same perspective with the Blanchfield pairing to show that dropping the local flatness hypothesis does not enlarge the class of knots admitting an embedded disk in \(B^4\) whose complement has fundamental group \(\mathbb Z\).

Knot concordance is usually studied in either the smooth category or the locally flat category. In these categories, concordance is nontrivial, since not every knot bounds a smooth or locally flat embedded disk in the $4$-ball. This is no longer true without the locally flat condition. Indeed, every knot $K$ bounds an embedded disk $D$, the simplest example being the cone on $K$. Observing that
\[
B^4\smallsetminus \cone(K) \cong (S^3\smallsetminus K)\times [0,\infty),
\]
one sees that the fundamental group of the complement of this disk is isomorphic to the fundamental group of the knot complement. This does not occur in the smooth or locally flat categories in general, and it leads to the following natural question.

\begin{question}
For a fixed knot $K$, what restrictions are there on the fundamental groups of complements of disks in $B^4$ bounded by $K$?
\end{question}

We address this question in the case where the group is $\Z$. It is a famous result of Freedman~\cite{Freedman1, Freedman-Quinn:1990-1, Garoufalidis-Teichner:2004-1} that a knot bounds a locally flat embedded disk $D\subset B^4$ with
$\pi_1(B^4\smallsetminus D)\cong \Z$
if and only if the knot has trivial Alexander polynomial. In this section, we use the Kojima--Yamasaki $\eta$-invariant and its connection with the Blanchfield form to prove that the same is true if the disk is merely continuously embedded,
%one further equivalent condition,
thereby establishing Theorem~\ref{thm:pi1-Z-disk}, whose statement we now recall.  For the interested reader, we remark that all of the results we use apply to knots in homology spheres which bound disks in (possibly non-smooth) homology balls.  Before applying \cite[Lemma~5]{KojYam79}, one may first need to use \cite[Theorem~8.2]{Freedman-Quinn:1990-1} to arrange that the complement of a disk in a homology ball admits a smooth structure, even when the homology ball itself does not.

\begin{reptheorem}{thm:pi1-Z-disk}
Let $K\subseteq S^3$ be a knot. The following are equivalent.
\begin{enumerate}
\item \label{cond: triv Alex poly} $K$ has trivial Alexander polynomial.
\item \label{cond: locally flat pi1=Z disk} $K$ bounds a locally flat embedded disk $D\subseteq B^4$ with $\pi_1(B^4\smallsetminus D)\cong \Z$.
\item \label{cond: pi1=Z disk} $K$ bounds an embedded disk $D\subseteq B^4$ with $\pi_1(B^4\smallsetminus D)\cong \Z$.
\end{enumerate}
\end{reptheorem}

\begin{proof}

As mentioned earlier, Freedman's theorem gives the equivalence of
\pref{cond: triv Alex poly} and \pref{cond: locally flat pi1=Z disk}, while
\pref{cond: locally flat pi1=Z disk} clearly implies
\pref{cond: pi1=Z disk}. It therefore remains to prove that
\pref{cond: pi1=Z disk} implies \pref{cond: triv Alex poly}.

Assume that \(K\) bounds an embedded disk \(D \subseteq B^4\) with
\(\pi_1(B^4 \smallsetminus D) \cong \mathbb{Z}\). Set
\[
W := B^4 \smallsetminus D
\qquad\text{and}\qquad
M := S^3 \smallsetminus K \subseteq W.
\]
Let \(\widetilde{W}\) be the universal cover of \(W\). Since
\(\pi_1(W) \cong \mathbb{Z}\), we have
\[
H_1(\widetilde{W}) = 0.
\]
Let \(\mathcal{A}(K)\) denote the Alexander module of \(K\). We will show that the Blanchfield pairing on \(\mathcal{A}(K)\), introduced in~\cite{Blanchfield57}, vanishes identically. Since the Blanchfield pairing of a knot is nonsingular (see, for example, \cite[Corollary 1.6]{FrPo17})
this will imply that \(\mathcal{A}(K)=0\), and hence that \(\Delta_K = 1\).

Recall the geometric description of the Blanchfield form. If
\(\alpha,\beta \in \mathcal{A}(K)\) are represented by \(1\)-cycles
%\(a,b \subseteq \widetilde{M}\),
\(a,b \in C_1(\widetilde{M})\),
and if \(Y \in C_2(\widetilde{M})\)
satisfies
\[
\partial Y = p(t)\,a
\]
for some nonzero \(p(t) \in \mathbb{Z}[t^{\pm 1}]\), then
\[
\Bl(\alpha,\beta)
=
\left[
\frac{1}{p(t)}
\sum_{k\in \mathbb{Z}} (Y,t^k b)t^k
\right]
\in \mathbb{Q}(t)/\mathbb{Z}[t^{\pm 1}].
\]

Next, consider the equivariant intersection pairing
\[
\lambda \colon
H_2(\widetilde{W}) \times H_2(\widetilde{W},\widetilde{M})
\longrightarrow
\mathbb{Z}[t^{\pm 1}];
\qquad
\lambda(x,y)
=
\sum_{k\in\mathbb{Z}} (P,t^k Q)t^k,
\]
where \(P \in Z_2(\widetilde{W})\) represents \(x\) and
\(Q \in C_2(\widetilde{W})\) represents \(y\), so that
\(\partial Q \subseteq \widetilde{M}\). Before taking intersections,
we push \(P\) slightly into the interior of \(\widetilde{W}\).
As in Remark~\ref{rmk: what if I-equiv}, it follows from
\cite[Lemma~5]{KojYam79} that this pairing vanishes identically.

Now let $\alpha,\beta \in \mathcal{A}(K)$. Choose $1$-cycles
$a,b \in C_1(\widetilde{M})$
%$a,b \subseteq \widetilde{M}$
representing $\alpha$ and $\beta$.
Since $H_1(\widetilde{W})=0$, there exist $2$-chains
$X_\alpha,X_\beta \in C_2(\widetilde{W})$ such that
\[
\partial X_\alpha=a
\qquad\text{and}\qquad
\partial X_\beta=b.
\]
Since $\mathcal{A}(K)$ is a torsion $\mathbb{Z}[t^{\pm 1}]$-module, there
exist a nonzero Laurent polynomial $p(t)\in\mathbb{Z}[t^{\pm 1}]$ and a
$2$-chain $Y\in C_2(\widetilde{M})$ such that
\[
\partial Y=p(t)a.
\]
Set
\[
Z:=Y-p(t)X_\alpha \in Z_2(\widetilde{W}).
\]
The chain $X_\beta$ represents a class
$[X_\beta]\in H_2(\widetilde{W},\widetilde{M})$. 
%After pushing $Z$ slightly into the interior of $\widetilde{W}$, 
The vanishing of the
equivariant intersection pairing $\lambda$ gives
\begin{align*}
0
&= \lambda([Z],[X_\beta]) \\
&= \sum_{k\in\mathbb{Z}} (Y,t^k b)t^k
   - p(t)\sum_{k\in\mathbb{Z}} (X_\alpha,t^k X_\beta)t^k.
\end{align*}
Therefore
\[
\left[
\frac{1}{p(t)}
\sum_{k\in\mathbb{Z}} (Y,t^k b)t^k
\right]
=
\left[
\sum_{k\in\mathbb{Z}} (X_\alpha,t^k X_\beta)t^k
\right]
=[0]\in \mathbb{Q}(t)/\mathbb{Z}[t^{\pm 1}]
\]
%in $\mathbb{Q}(t)/\mathbb{Z}[t^{\pm 1}]$, since the right-hand side is represented by a Laurent polynomial.
By the geometric definition of the
Blanchfield form, the left-hand side is precisely $\Bl(\alpha,\beta)$.
Thus
\[
\Bl(\alpha,\beta)=0
\qquad\text{for all }
\alpha,\beta \in \mathcal{A}(K),
\]
which completes the proof.
\end{proof}

% So the Blanchfield form on \(\mathcal{A}(K)\) is identically zero. Since the
% Blanchfield form of a knot is nonsingular, it follows that
% \[
% \mathcal{A}(K)=0.
% \]
% Equivalently, \(\Delta_K = 1\). This proves that
% \pref{cond: pi1=Z disk} implies \pref{cond: triv Alex poly}, and
% completes the proof.

The proof of the preceding theorem admits a slightly more local form.
Let \(D\subset B^4\) be any continuously embedded disk with boundary \(K\), and use the notation
\[
    W:=B^4\smallsetminus D,
    \qquad\text{and}\qquad
    M:=S^3\smallsetminus K.
\]
Let \(\widetilde W\) and \(\widetilde M\) denote the infinite cyclic covers of \(W\) and \(M\), respectively. Let $\mathcal{A}(D)=H_1(\widetilde W)$
denote the Alexander module of \(D\). The inclusion \(M\hookrightarrow W\) induces a homomorphism
\[
    \mathcal{A}(K)
    \longrightarrow
    \mathcal{A}(D).
\]
In the proof of the preceding theorem, the hypothesis
\(\pi_1(W)\cong \mathbb Z\) was used only to ensure that
\(\mathcal{A}(D)=0\), so that every class in \(\mathcal{A}(K)\) bounds a
\(2\)-chain in \(\widetilde W\). For a general continuously embedded disk,
the same argument applies precisely to those classes of \(\mathcal{A}(K)\)
that become null-homologous in \(\widetilde W\), that is, to the kernel of the above map. Since the argument is quite similar to the proof of Theorem~\ref{thm: pi1=Z}, we omit the details.

\begin{theorem}\label{thm: ker is isotropic}
If a knot $K$ bounds a continuously embedded disk $D$, then
$\ker(\mathcal{A}(K)\to \mathcal{A}(D))$ is isotropic with respect to the
Blanchfield form. That is, if
$\alpha,\beta\in \ker(\mathcal{A}(K)\to \mathcal{A}(D))$, then
$\Bl(\alpha,\beta)=0$. \qed
\end{theorem}

\section{Further questions}\label{sect: problems}
We close with a few problems that merit further investigation. First, as observed in Section~\ref{subsec:computations}, a theorem of Jin~\cite[Theorem~3.5]{Jin91} determines precisely which pairs of sequences can be realized as $\beta(L_1,L_2)$ and $\beta(L_2,L_1)$.

\begin{problem}\label{problem: classify gamma sequences}
Determine which triples of sequences can be realized by a $3$-component link
$L=(L_1,L_2,L_3)$ with distinguished component $L_1$, together with a Seifert surface $G$ for $L_1$, as
\[
(\beta(L_1,L_2), \beta(L_1,L_3), \gamma(L_1,L_2,L_3,G)).
\]
\end{problem}

Theorem~\ref{thm:classify} classifies when $3$-component links are weakly cobordant to the unlink in terms of the concordance class of the distinguished component, an equality at the level of Alexander modules, and the $\beta$- and $\gamma$-invariants. What happens for links with more than three components?

\begin{problem}
Extend the definitions of weak cobordism, derivatives, and the $\beta$- and $\gamma$-invariants to $n$-component links with a distinguished component. Determine the classification of weak cobordism for $n$-component links.
\end{problem}

{One such extension of the $\gamma$-invariant appears in \cite[Theorem 1.2]{Tsukamoto-Yasuhara:2007} where they are used to study the Conway potential function.}  Addressing the problem above should lead to new lifts of infinite families of Milnor invariants. However, it seems unlikely that an approach based on this notion of derivative can lift families of Milnor invariants not of the form $\bar{\mu}({1^kJ)}$ with $|J|=2$. We therefore ask more generally which other families of Milnor invariants admit integer-valued lifts.

\begin{problem}
Find integer-valued lifts of further infinite families of Milnor invariants.
\end{problem}

In~\cite{CST:2017}, Conant, Schneiderman, and Teichner translate Cochran's $\beta$-invariants into the language of Whitney towers. More precisely, they define what they call a Cochran tower, prove that every $2$-component link with vanishing linking number bounds such a tower, and show how to compute the $\beta$-invariants from its intersections.

\begin{problem}
Define an appropriate class of Whitney towers for $3$-component links from which the $\beta$- and $\gamma$-invariants can be computed.
\end{problem}

For any knot $K\subseteq S^3$, its cone on $K$, denoted  $\operatorname{cone}(K)$, is a continuously embedded disk with
\[
\pi_1\bigl(B^4\smallsetminus \operatorname{cone}(K)\bigr)
\cong
\pi_1(S^3\smallsetminus K).
\]
This is a phenomenon that cannot occur in the smooth or locally flat categories. On the other hand, Theorem~\ref{thm:pi1-Z-disk} shows that it is quite restrictive for a knot to bound a disk whose complement has infinite cyclic fundamental group. More generally, Theorem~\ref{thm: ker is isotropic} shows that there are restrictions on the possible size of
$\ker(\mathcal{A}(K)\to \mathcal{A}(D))$.

\begin{question}
Fix a knot $K$. Which groups arise as
\[
\pi_1(B^4\smallsetminus D),
\]
where $D\subset B^4$ ranges over continuously embedded disks bounded by $K$? Which submodules of $\mathcal{A}(K)$ arise as
\[
\ker(\mathcal{A}(K)\to \mathcal{A}(D))?
\]
\end{question}

%\chris{Motivated by section 7:  What groups can be $\pi_1(continuous embedded disk)$ Refine to what are alexander module kernels.}

%\chris{I do not know enough history to write this.  Maybe discuss.  } 
In~\cite{Orr:1991}, Orr reformulates the $\beta$-invariants in terms of a sequence of classes in $\pi_3(S^2\vee S^3)$. Can an analogous interpretation be given for the $\gamma$-invariants?

\begin{problem}
Can $\gamma(L)$ be interpreted as a sequence of classes in $\pi_3(X)$ for some fixed space $X$?
\end{problem}

In Proposition~\ref{thmbody:seifert-matrix}, we recover the $\gamma$-sequence in terms of the Seifert matrix of a Seifert surface for the distinguished component. The Milnor triple linking number can be computed in terms of the intersections of a triple of Seifert surfaces for $L$~\cite{MelMel03}. This formulation becomes particularly simple for systems of surfaces that intersect only in clasps, called C-complexes~\cite{DaRo2017}; see also~\cite{Cooper82,Cimasoni04} for the precise definition. These objects admit analogues of Seifert matrices. We ask whether these generalized Seifert matrices can recover the $\gamma$-sequence, or the $\beta$-sequences, of $L$.

\begin{problem}
Let $L=(L_1,L_2,L_3)$ be a $3$-component link, and let $G=(G_1,G_2,G_3)$ be a C-complex for $L$. Can the $\beta$- and $\gamma$-sequences of $L$ be determined from the generalized Seifert matrices associated to $G$?
\end{problem}

In Theorem~\ref{thm:notweaklytotheunlink}, we found examples of links for which all $\beta$- and $\gamma$-invariants vanish, but which are not locally flat weakly cobordant to the unlink. Since the proof of Theorem~\ref{thm:notweaklytotheunlink} does not work in the continuous category, we ask the following question.

\begin{problem}\label{problem: classify topological weak cobordism}
Do the $\beta$- and $\gamma$-invariants classify continuous weak cobordism of $3$-component links?
\end{problem}

%\chris{Lift More MIlnor invariants}

 %\chris{Feb 3 we decided to scrap this and replace with an open questions section.  Finite type invt?  Whitney tower?  Other Milnor invt. }

 %Classification of when links are weakly cobordant. / when links have the same beta's and eta's

\bibliographystyle{alpha}
\bibliography{biblio}
\end{document}